\documentclass[11pt,text height = 615pt, text width = 360pt,final]{amsart}

\usepackage[pdftex]{graphicx,color,hyperref}
\usepackage[initials,nobysame]{amsrefs}
\usepackage{amsmath}
\usepackage{amssymb}
\usepackage{ascmac}
\usepackage[dvipsnames]{xcolor}
\usepackage{mathtools}
\usepackage{caption}
\usepackage{subcaption}
\mathtoolsset{showonlyrefs}
\usepackage[shortlabels]{enumitem}

\makeatletter
\@addtoreset{equation}{section}
\def\theequation{\thesection.\arabic{equation}}
\makeatother

\newtheorem{theorem}{Theorem}[section]
\newtheorem{proposition}[theorem]{Proposition}

\newtheorem{lemma}[theorem]{Lemma}
\theoremstyle{definition}
\newtheorem{definition}[theorem]{Definition}

\theoremstyle{remark}
\newtheorem{remark}[theorem]{Remark}

\DeclareMathOperator{\sech}{sech}
\DeclareMathOperator{\cn}{cn}
\DeclareMathOperator{\dn}{dn}
\DeclareMathOperator{\sn}{sn}
\DeclareMathOperator{\am}{am}

\newcommand{\norm}[2]{\ensuremath{\Vert #1 \Vert_{#2}}}
\renewcommand{\labelenumi}{(\roman{enumi})}

\begin{document}

\title[Embeddedness of elastic flows]{Optimal thresholds for preserving embeddedness of elastic flows}
\author[T.~Miura]{Tatsuya Miura}
\address[T.~Miura]{Department of Mathematics, Tokyo Institute of Technology, Meguro, Tokyo 152-8511, Japan}
\email{miura@math.titech.ac.jp}
\author[M.~M\"{u}ller]{Marius M\"{u}ller}
\address[M.~M\"{u}ller]{
Mathematisches Institut, Albert-Ludwigs-Universität Freiburg, Ernst-Zermelo-Straße 1, 79104 Freiburg im Breisgau, Germany;
present address: Mathematisches Institut, Universität Leipzig, Augustusplatz 10, 04109 Leipzig, Germany}
\email{marius.mueller@math.uni-leipzig.de}
\author[F.~Rupp]{Fabian Rupp}
\address[F.~Rupp]{Institute of Applied Analysis, Ulm University, Helmholtzstra\ss e 18, 89081 Ulm, Germany; present address: Faculty of Mathematics, University of Vienna, Oskar-Morgenstern-Platz 1, 1090 Vienna, Austria}
%\email{fabian.rupp@uni-ulm.de}

\email{fabian.rupp@univie.ac.at}
\keywords{Embeddedness preservation, elastic flow, elastica, self-intersection, geometric inequality, rotation number}
\subjclass[2020]{53E40 (primary), 49Q10, 53A04 (secondary)}

\begin{abstract}
  We consider elastic flows of closed curves in Euclidean space.
  We obtain optimal energy thresholds below which elastic flows preserve embeddedness of initial curves for all time.
  The obtained thresholds take different values between codimension one and higher.
  The main novelty lies in the case of codimension one, where we obtain the variational characterization that the thresholding shape is a minimizer of the bending energy (normalized by length) among all nonembedded planar closed curves of unit rotation number.
  It turns out that a minimizer is uniquely given by a nonclassical shape, which we call ``elastic two-teardrop''.
\end{abstract}

\maketitle

\section{Introduction}\label{sect:introduction}

In this paper we consider the embeddedness-preserving property of elastic flows of closed curves in Euclidean space in any codimension.

A one-parameter family of immersed closed curves $\gamma:\mathbb{T}^1\times[0,\infty)\to\mathbb{R}^n$, where $\mathbb{T}^1:=\mathbb{R}/\mathbb{Z}$, is called {\em elastic flow} (or {\em length-penalized elastic flow}) if for a given constant $\lambda>0$ the family $\gamma$ satisfies the following equation:
\begin{equation}\label{eq:elasticflow}
  \partial_t\gamma = -2\nabla_s^2\kappa-|\kappa|^2\kappa+\lambda\kappa,
\end{equation}
where $\kappa$ denotes the curvature vector $\kappa:=\partial_s^2\gamma$ and $\nabla_s$ denotes the normal derivative with respect to the arclength parameter $s$, that is $\nabla_s\psi=\partial_s\psi-(\partial_s\psi,T)T$, where $T:=\partial_s\gamma$ denotes the unit tangent.
In this paper we call this flow {\em $\lambda$-elastic flow} in order to make the value of $\lambda$ explicit.
The $\lambda$-elastic flow may be regarded as the $L^2$-gradient flow of the modified (or length-penalized) bending energy $E_\lambda$, which can be defined in terms of the bending energy $B[\gamma]:=\int_\gamma|\kappa|^2 \mathrm{d}s$ and the length $L[\gamma]:=\int_\gamma \mathrm{d}s$ by
\begin{equation}\label{eq:energy_E_lambda}
  E_\lambda[\gamma]:=B[\gamma]+\lambda L[\gamma] = \int_\gamma(|\kappa|^2+\lambda) \; \mathrm{d}s
\end{equation}
for a given $\lambda>0$.
In particular, the energy $E_\lambda$ generically decreases along the flow.

Similarly, a family $\gamma$ is called {\em fixed-length elastic flow} if it solves \eqref{eq:elasticflow}, where $\lambda$ depends on the solution and is given in the form of
\begin{equation}\label{eq:def lambda}
  \lambda(t)= \lambda[\gamma(\cdot,t)]=\frac{\int_{\gamma(\cdot,t)}\langle 2\nabla_s^2\kappa+|\kappa|^2\kappa,\kappa \rangle \; \mathrm{d}s}{\int_{\gamma(\cdot,t)}|\kappa|^2 \; \mathrm{d}s}.
\end{equation}
The fixed-length elastic flow may be regarded as the $L^2$-gradient flow of the bending energy $B$ under the fixed-length constraint $L[\gamma]=L_0$ for a given $L_0>0$.
For later use we also define the (scale-invariant) normalized bending energy $\bar{B}$ by \begin{equation}
    \bar{B}[\gamma]:=L[\gamma]B[\gamma].
\end{equation}
The energy $\bar{B}$ decreases along the fixed-length elastic flow.

Long time existence of elastic flows from smooth initial data as well as smooth convergence to stationary solutions (which are elasticae due to %cf. 
Definition \ref{def:Eelastica}) are known to hold in general, see e.g.\ \cite{DKS,DLPSTE,DPS16,MantegazzaPozzetta,LiYau1,Length Preserving} and also the survey \cite{MPP}.
However, since elastic flows are of higher order, the global behavior of solutions is less understood.
For example, due to the lack of maximum principle, generic higher order flows do not possess many kinds of \emph{positivity preserving} type properties, such as embeddedness or convexity, cf.\ \cite{Blatt}.

Our focus will be on embeddedness along elastic flows.
In previous studies the authors found the following optimal energy threshold for all-time embeddedness in \cite{LiYau1} ($n=2$) and \cite{LiYau2} ($n\geq2$):
Let $C_8=\bar{B}[\gamma_8]>0$ denote the energy $\bar{B}$ of a figure-eight elastica $\gamma_8$, %cf.\
see Definition \ref{def:def_fig_eight} and Figure \ref{fig:8}.
If an immersed closed curve $\gamma_0$ has the property that $\bar{B}[\gamma_0]<C_8$ (resp.\ $\frac{1}{4\lambda}E_\lambda[\gamma_0]^2<C_8$), then the fixed-length elastic flow (resp.\ $\lambda$-elastic flow) starting from $\gamma_0$ is embedded for all time $t\geq0$.
This threshold is optimal since a figure-eight elastica is a nonembedded stationary solution of the flow.
However, these results do not capture {\em embeddedness breaking along the flow} since the figure-eight elastica is {\em initially not embedded}.

Here we consider a slightly different problem, which is more natural in view of embeddedness ``preserving'': {\em Suppose that an initial closed curve is embedded.
Then, what is the optimal (maximal) energy threshold below which the elastic flow must remain embedded for all time?}
Our main result reveals that this subtle difference yields a substantial improvement of the threshold value in the planar case $n=2$, while in higher codimensions $n\geq3$ the same threshold $C_8$ is still optimal.
We now introduce a new constant $C_{2T}=\bar{B}[\gamma_{2T}]$ ($>C_8$) given by the energy $\bar{B}$ of an {\em elastic two-teardrop} $\gamma_{2T}$, which is a nonclassical shape and one of our new findings, %cf.\ 
see Definition \ref{def:two_teardrop} and Figure \ref{fig:2T}.
We can represent both $C_{2T}$ and $C_8$ by elliptic integrals rather explicitly, %cf.\
see \eqref{eq:defC2T} and \eqref{eq:defC8}, respectively.
Then we define our new threshold by
\begin{align}
  C^*(n):=
  \begin{cases}
    C_{2T} & (n=2),\\
    C_8 & (n\geq3).
  \end{cases}
\end{align}
The numerical values are $C_{2T}\simeq 146.628$ and $C_8\simeq 112.439$.

\begin{figure}[ht]
    \centering
    \begin{subfigure}[b]{0.4\textwidth}
         \centering
         \includegraphics[width=25mm]{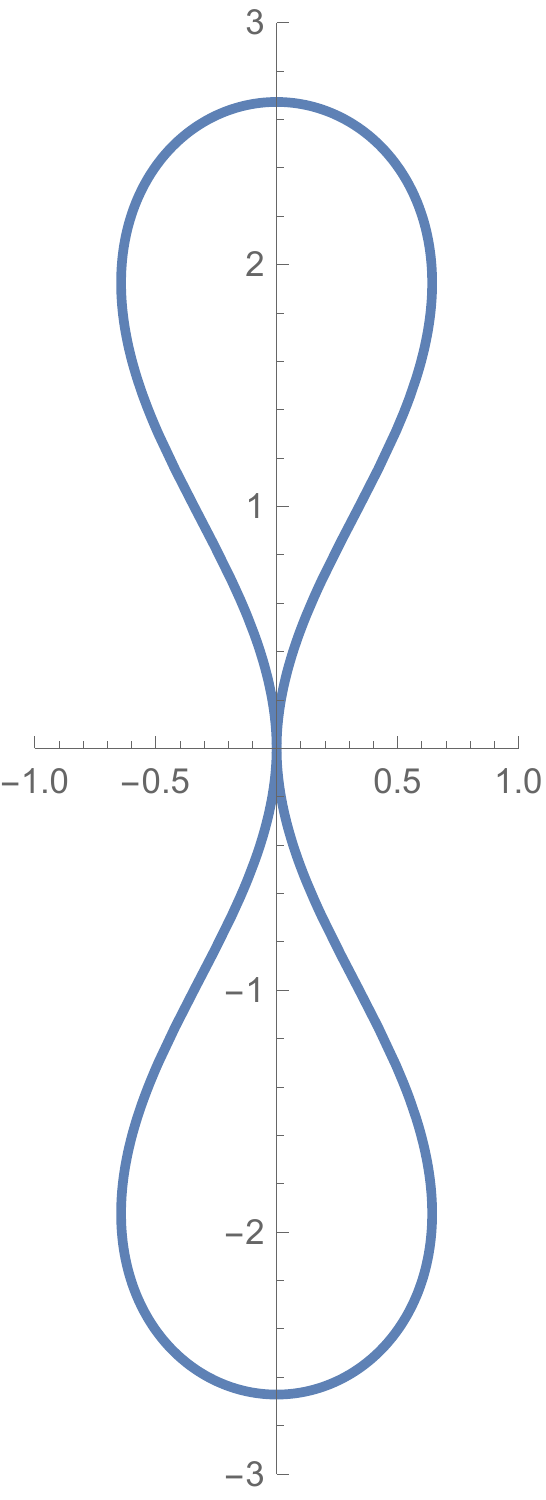}
         \caption{Elastic two-teardrop.}
         \label{fig:2T}
    \end{subfigure}
    \begin{subfigure}[b]{0.4\textwidth}
         \centering 
         \includegraphics[width=25mm]{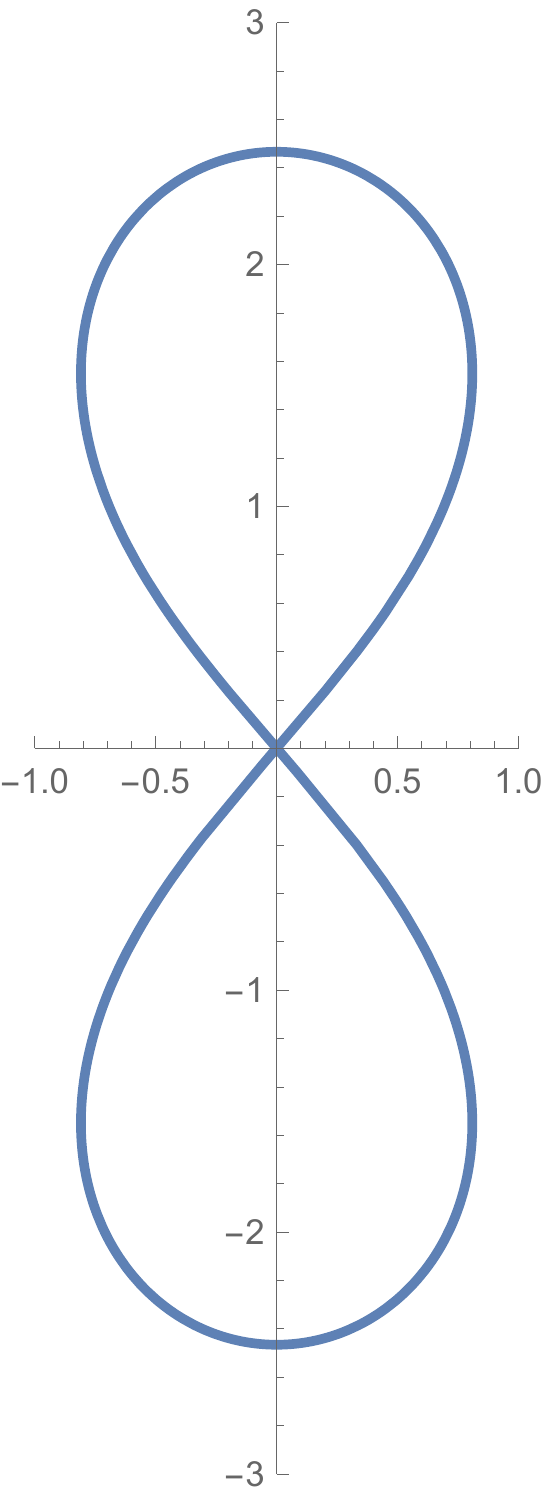}
         \caption{Figure-eight elastica.}
         \label{fig:8}
    \end{subfigure}
    \caption{Optimal configurations among nonembedded closed curves.}
    \label{fig:2Tand8}
\end{figure}

Our main result then reads as follows.

\begin{theorem}\label{thm:main}
  If a closed smooth curve $\gamma_0:\mathbb{T}^1\to\mathbb{R}^n$ is embedded, and if
  \begin{align*}
    \bar{B}[\gamma_0] &\leq C^*(n) \\
    \big(\text{resp.}\quad \tfrac{1}{4\lambda}E_\lambda[\gamma_0]^2 &\leq C^*(n) \ \text{for some}\ \lambda>0 \ \big),
  \end{align*}
  then the fixed-length elastic flow (resp.\ $\lambda$-elastic flow) with initial datum $\gamma_0$ remains embedded for all time $t\geq0$.

  In addition, for any $\varepsilon>0$ (resp.\ $\varepsilon,\lambda>0$) there exists an embedded closed smooth curve $\gamma_0:\mathbb{T}^1\to\mathbb{R}^n$ such that
  \begin{align*}
    \bar{B}[\gamma_0] &\in \big( C^*(n), C^*(n)+\varepsilon \big]\\
    \Big( \text{resp.}\quad\tfrac{1}{4\lambda}E_\lambda[\gamma_0]^2 &\in \big( C^*(n), C^*(n)+\varepsilon \big] \Big)
  \end{align*}
  and such that the fixed-length elastic flow (resp.\ $\lambda$-elastic flow) with initial datum $\gamma_0$ loses its embeddedness at some time $t_0>0$.
\end{theorem}

\begin{remark}
  In the planar case $n=2$ the limit profile of each elastic flow must be a circle whenever an initial curve is embedded, since the rotation number is preserved along the flow, while the only elastica with unit rotation number is a circle.
  For higher codimensions $n\geq3$ this is not the case since there are other embedded elasticae.
  However, below the threshold $C_8$  both flows still converge to circles.
  This follows by a more quantitative argument, namely by energy quantization of closed elasticae, cf.\ \cite[Section 4]{LiYau2}.
\end{remark}

In the following, we briefly sketch our proof strategy in the case of the length-preserving flow.
Since the normalized bending energy decreases, the main issue for the first part (embeddedness preserving) is to find an appropriate sub-level set of $\bar{B}$ in which all admissible closed curves must be embedded.
On the other hand, in order to prove the optimality part (embeddedness breaking), the above sub-level set must be `widest' and `approachable by embedded curves'.
This observation naturally leads us to study a minimization problem for $\bar{B}$ among all closed curves that are not embedded but approachable by embedded ones.
Once this minimization problem is solved, then we may take the minimum value as the desired threshold.
We then perform a delicate perturbation of the optimal configuration to construct an embedded initial curve which yields loss of embeddedness.
The proof of embeddedness breaking is strongly inspired by \cite{Blatt}, but we need an additional topological argument in higher codimensions.

We now discuss more on how to detect the optimal thresholds.
Since we are interested in minimization problems for $\bar{B}$, from now on we specify the natural $H^2$-Sobolev regularity for curves.
We first recall the following general estimate for nonembedded closed curves, which is recently obtained by the last two authors for $n=2$ \cite{LiYau1} and by the first author for $n\geq2$ \cite{LiYau2}.

\begin{theorem}[\cite{LiYau1,LiYau2}]\label{thm:figure-eight}
 Let $n\geq2$ and $\gamma:\mathbb{T}^1\to\mathbb{R}^n$ be an immersed closed $H^2$-curve.
 If $\gamma$ has a self-intersection, then
  \begin{equation}
    \bar{B}[\gamma] \geq C_8,
  \end{equation}
  where equality is attained if and only if $\gamma$ is a figure-eight elastica %cf.\ 
  (in the sense of Definition \ref{def:def_fig_eight}).
\end{theorem}

This statement is luckily informative enough for our purpose whenever $n\geq3$, even though its formulation does not take any approachability into account.
This is because a figure-eight elastica is approachable by embedded curves if $n\geq3$ via an out-of-plane perturbation.
However, for $n=2$ a figure-eight elastica is not even regularly homotopic to embedded curves; thus we need to impose an additional constraint on the minimizing problem.
It turns out that a sufficient constraint is to fix the \emph{rotation number} to be $1$ (as with embedded curves); such a class contains all approachable curves by a continuity argument.
For a planar curve $\gamma$, we define the (absolute) rotation number by $N[\gamma]:=| \frac{1}{2\pi}\int_\gamma k \; \mathrm{d}s|$, where $k$ denotes the signed curvature; the choice of the sign does not affect the value of $N$.
The key ingredient in the planar case is

\begin{theorem}\label{thm:1.4}
  Let $\gamma:\mathbb{T}^1\to\mathbb{R}^2$ be an immersed closed $H^2$-curve.
  If $\gamma$ has a self-intersection and $N[\gamma]=1$, then
  \begin{equation}
    \bar{B}[\gamma] \geq C_{2T},
  \end{equation}
  where equality is attained if and only if $\gamma$ is an elastic two-teardrop  %cf.\ 
  (in the sense of Definition \ref{def:two_teardrop}).
  Moreover, there exists no other solution to the corresponding variational inequality \eqref{eq:pertfunc}.
\end{theorem}

The optimal ``two-teardrop'' is now approachable by embedded curves, as desired. Note that the teardrop shape is reminiscent of the profile curve of the Willmore surface achieved by applying a M\"obius inversion to a catenoid. However, the curves exhibit distinct shapes.

% More information about connection to the inverted catenoid?  

A remarkable point is that the elastic two-teardrop is of class $C^{2,1}=W^{3,\infty}$ but not $C^3$, in particular not globally an elastica.
This loss of regularity is caused by the constraint on self-intersections.
This phenomenon does not appear in Theorem \ref{thm:figure-eight} as a figure-eight elastica is by chance smooth, but is generically observed under the higher-multiplicity constraint, see \cite[Theorem 1.3]{LiYau2}.
Theorem \ref{thm:1.4} reveals that the loss of regularity occurs even in the multiplicity-two case if we fix the rotation number $N$.
This also implies the presence of a nonclassical local minimizer (two-teardrop) without fixing $N$ since fixing $N$ is an open condition.
It is also remarkable that the loss of regularity occurs only when $N=1$; accordingly, Theorem \ref{thm:1.4} classifies all possible solutions to the variational inequality \eqref{eq:pertfunc} and their stability, see Remark \ref{rem:classification} for details.

We prove Theorem \ref{thm:1.4} with variational techniques.
The existence of a minimizer follows by a direct method.
However, because of the constraints, one does not have a global Euler--Lagrange equation but just 
%a 
the variational inequality %, cf.\ 
\eqref{eq:pertfunc}, which yields a free-boundary-type problem.
To overcome this issue we first give a detailed analysis to reduce the possibility of self-intersections of solutions to \eqref{eq:pertfunc}.
In fact, we prove that the only possible case is a single tangential self-intersection with opposite tangent directions, so that the objective curve can be divided into two parts --- each of which is an embedded closed curve with a single cuspidal singularity and satisfies the elastica equation except at the cusp.
We call such a curve an {\em embedded cuspidal elastica} (ECE).
Our main effort is devoted to an exhaustive classification of all ECEs, where we conclude that there are only two possibilities; {\em teardrop elasticae} and {\em heart-shaped elasticae}, %cf.\
see Figure \ref{img:fig2}.
We then perform a further analysis of the shapes of all possible composites of them and deduce that the composite of a teardrop elastica and its reflection is in fact the unique solution to \eqref{eq:pertfunc}.
In particular, this implies uniqueness of minimizers.

\begin{figure}[ht]
    \centering
    \begin{subfigure}[b]{0.5\textwidth}
    \centering
    \includegraphics[width=50mm]{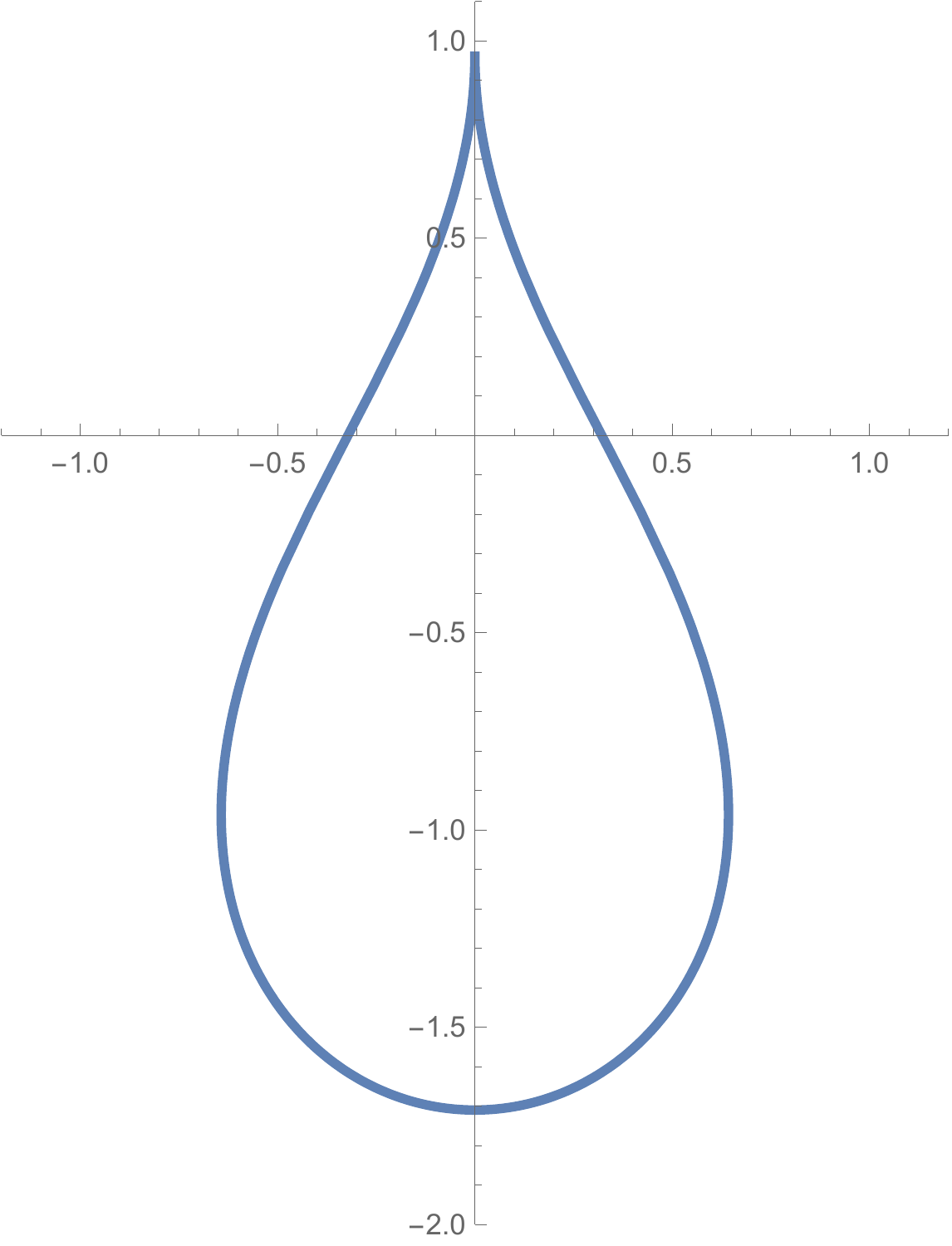}
    \subcaption{Teardrop elastica.}
    \label{img:teardrop}
    \end{subfigure}
    \begin{subfigure}[b]{0.45\textwidth}
    \includegraphics[width=50mm]{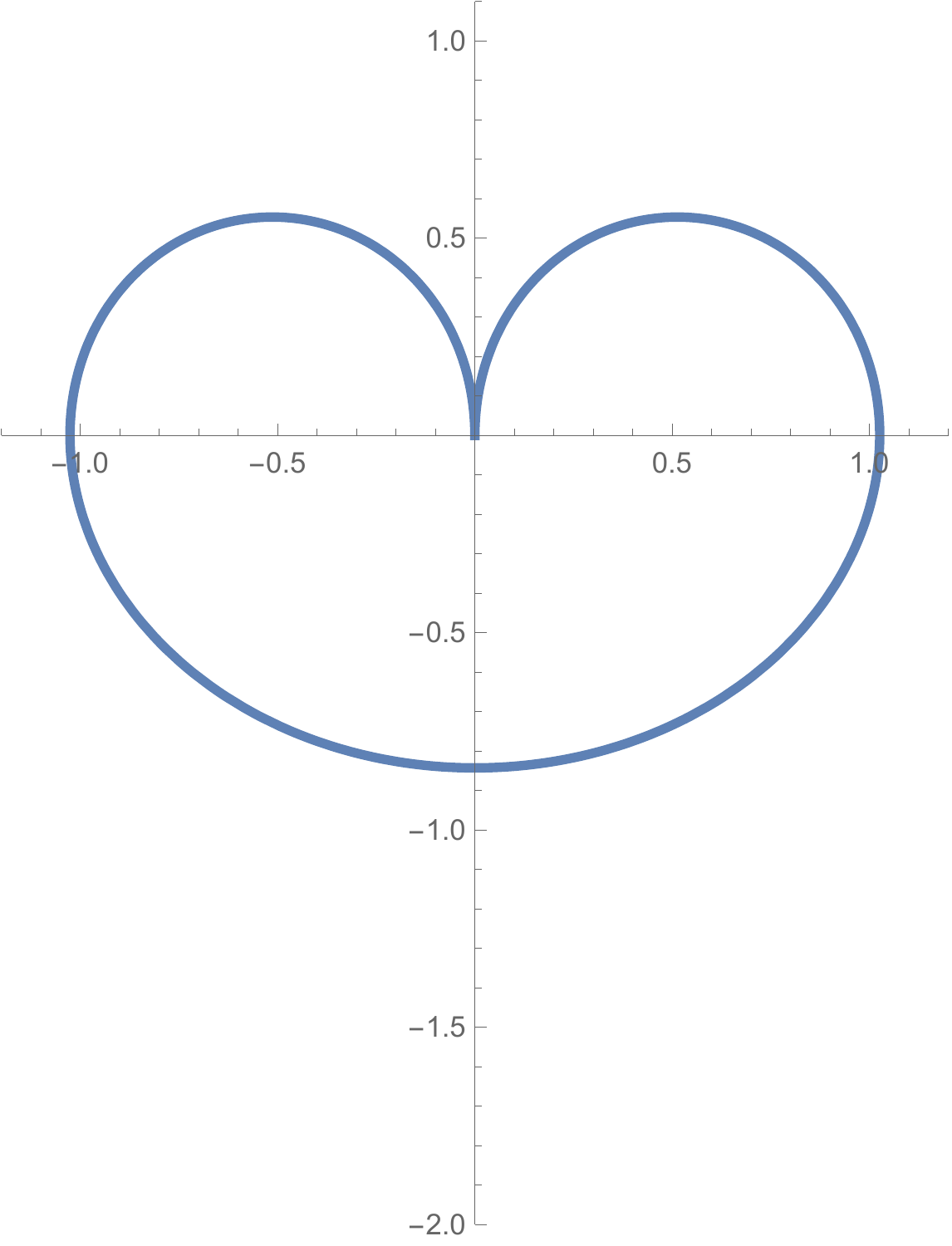}
    \subcaption{Heart-shaped elastica.}\label{img:heart}
    \end{subfigure}
    \caption{Embedded cuspidal elasticae (ECEs).}
    \label{img:fig2}
\end{figure}

The variational analysis of $\bar{B}$ among self-intersecting curves is also important in view of its strong connection to elastic knots, which model knotted springy wires, cf.\ \cite{GallottiPierreLouis,GeRvdM}.
Along the way of the above proof (in %cf.\ 
Lemma \ref{lem:rescaling}) we encounter a unique critical composite of a teardrop elastica and a heart-shaped elastica as in Figure \ref{img:fig3a}, and this shape matches a known candidate of an \emph{elastic knot} for the figure-eight knot class $4_1$, which has been previously observed experimentally and numerically, cf.\ \cite{AvSo14,BartelsReiter,GRvdM} and Figure \ref{img:fig3b}.
In fact, we conjecture that our critical \emph{teardrop-heart} gives an explicit parametrization of an (energy-minimal) elastic knot of class $4_1$ in the sense of Gerlach--Reiter--\mbox{von der Mosel} \cite{GeRvdM}, since Bartels--Reiter's numerical computation suggests that such a planar shape has less energy than another typical candidate of spherical (non-planar) shape, cf.\ \cite[Section 5.3]{BartelsReiter}.

\begin{figure}[ht]
    \centering
    \begin{subfigure}[b]{0.45\textwidth}
    \centering
    \includegraphics[width=50mm]{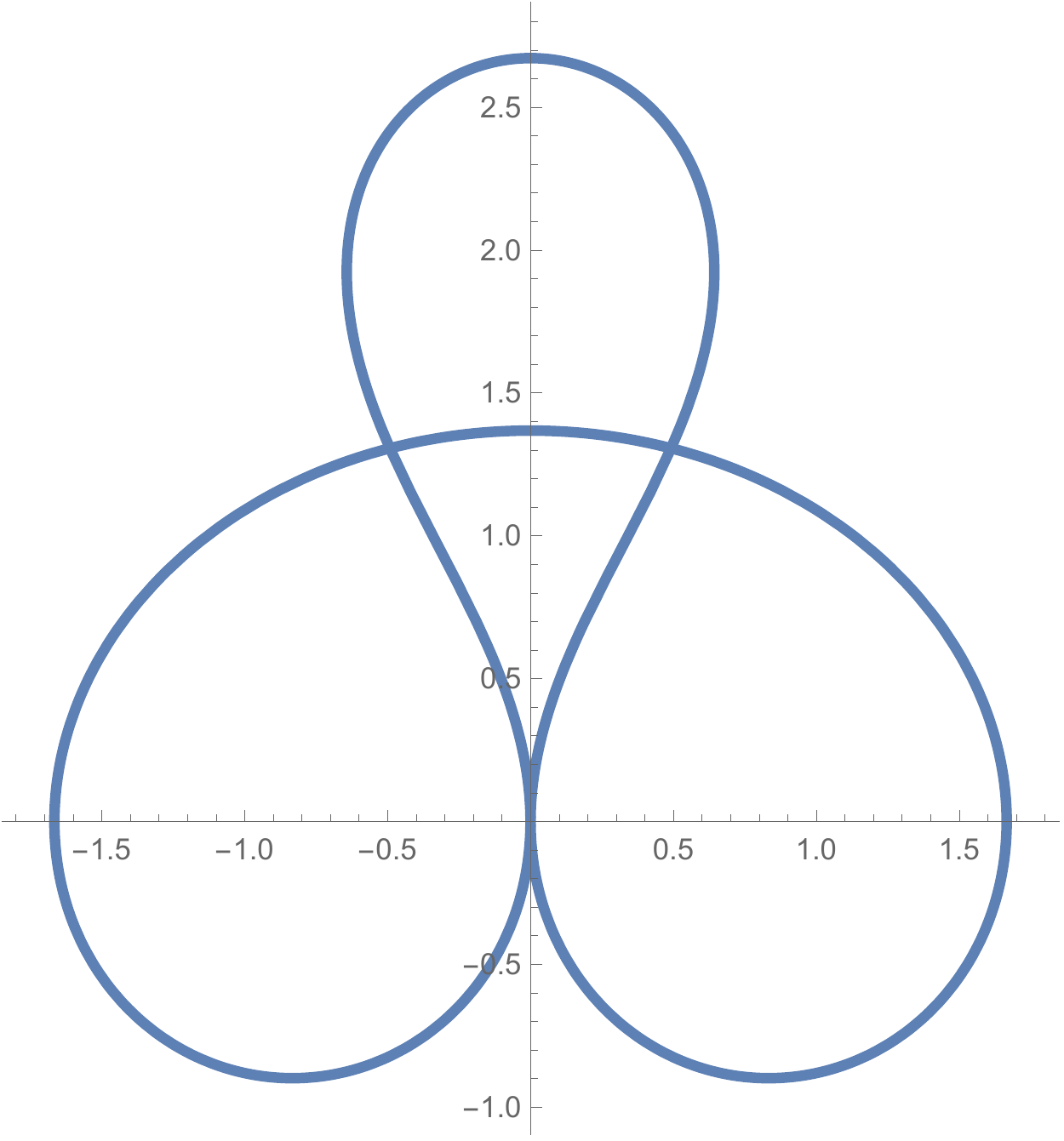}
    \subcaption{The critical teardrop-heart which arises in our analysis.}
    \label{img:fig3a}
    \end{subfigure}
    \hfill
    \begin{subfigure}[b]{0.45\textwidth}
    \hfill
    \includegraphics[width=50mm]{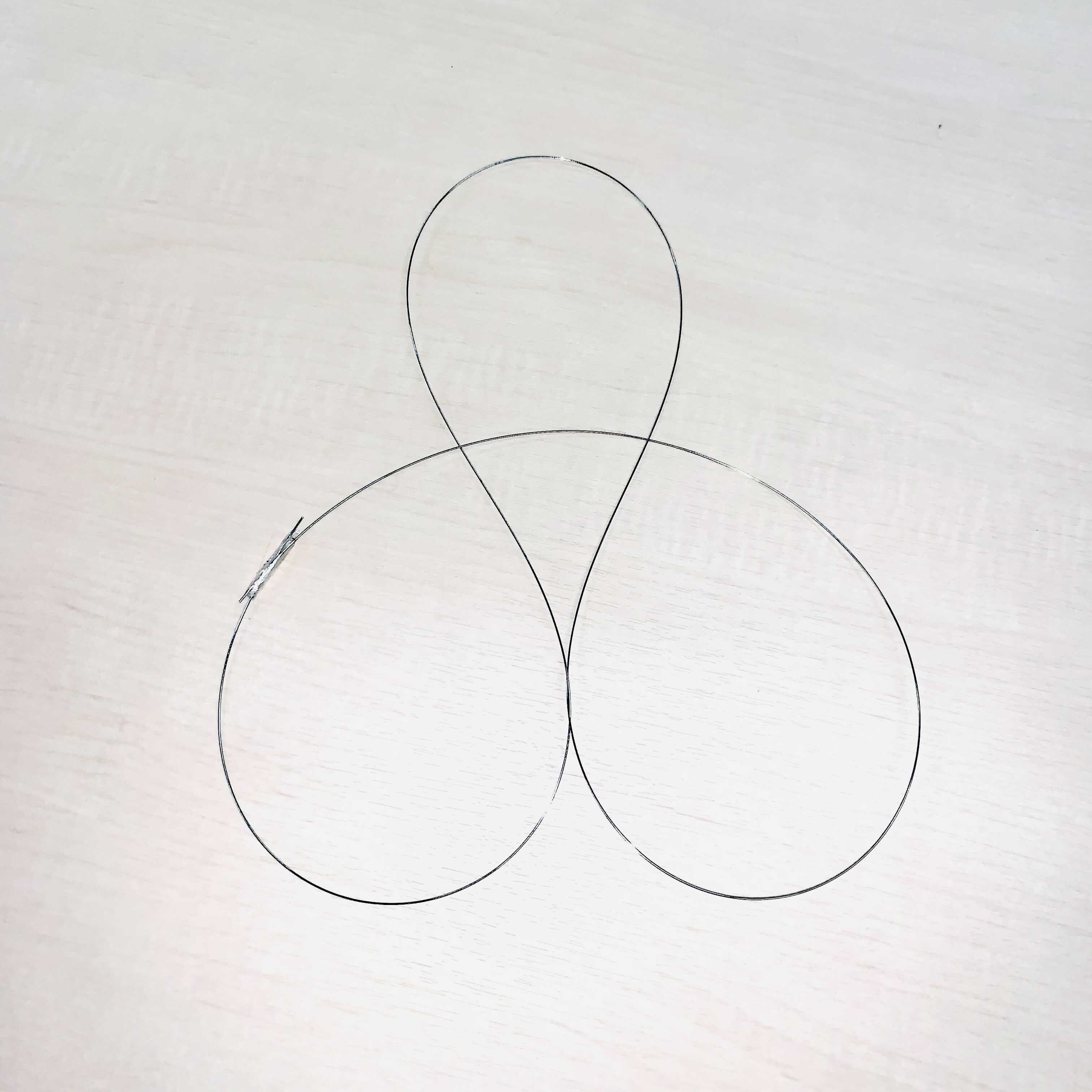}
    \subcaption{A springy wire representing a figure-eight knot strives to achieve a teardrop-heart configuration.}
    \label{img:fig3b}
    \end{subfigure}
    \caption{Elastic teardrop-hearts.}
    \label{img:fig3}
\end{figure}

Finally, we mention some relevant results on different flows for closed curves.
The possibility of losing embeddedness or convexity is indicated by Linn\'{e}r \cite{Linner1989} in 1989 for a certain ($H^1$-)gradient flow of the bending energy, which is different from the elastic flows (see also \cite{Linner1998}).
For the surface diffusion flow, which is also different but of higher order and regarded as an $H^{-1}$-gradient flow of the length, Giga--Ito constructed examples losing embeddedness \cite{GigaIto1998} and convexity \cite{GigaIto1999}, which are later extended by Blatt to a wide class of higher order flows \cite{Blatt}.
We remark that the analysis for the surface diffusion flow is more involved because of possible singularities in finite time, cf.\ \cite{Chou}.
Up to now global existence is ensured only for perturbations of circles, see e.g.\ \cite{ElliottGarcke,EscherMayerSimonett,Wheeler} (and also \cite{MiuraOkabe} for a multiply-covered case).
In particular, Wheeler's result \cite{Wheeler} gives an explicit (but non-optimal) quantitative sufficient condition for all-time embeddedness.

This paper is organized as follows:
In Section \ref{sect:minimization} we prove Theorem \ref{thm:1.4}.
In Section \ref{sect:elasticflows} we apply Theorem \ref{thm:1.4} and Theorem \ref{thm:figure-eight} to prove Theorem \ref{thm:main}.

\subsection*{Acknowledgments.}
Tatsuya Miura is supported by JSPS KAKENHI Grant Numbers 18H03670, 20K14341, and 21H00990, and by Grant for Basic Science Research Projects from The Sumitomo Foundation. 
Fabian Rupp is supported by the DFG (Deutsche Forschungsgemeinschaft), project no.\ 404870139. Moreover, the authors are grateful to the referee for their valuable comments on the original manuscript.

\section{The minimization problem}\label{sect:minimization}

This section is devoted to the proof of Theorem  \ref{thm:1.4}. 
First we fix some notation. 
We define 
\begin{equation}
    H^{2}_{imm}(\mathbb{T}^1; \mathbb{R}^2) := \{ \gamma \in H^2(\mathbb{T}^1; \mathbb{R}^2) : |\gamma'(x)|  \neq 0 \; \textrm{for all} \;  x \in \mathbb{T}^1 \}.
\end{equation}
Analogously we define $C^k_{imm}( \mathbb{T}^1; \mathbb{R}^2)$ and $C^k_{imm}([a,b];\mathbb{R}^2)$ for all $k \geq 1$. 
Further, we define the admissible set 
% \begin{equation}\label{eq:admset}
%     \mathcal{A}_0 := \{ \gamma \in H^2_{imm}(\mathbb{T}^1; \mathbb{R}^2) : N[\gamma] = 1, \; \exists \;  x_1, x_2 \in \mathbb{T}^1 , x_1 \neq x_2 : \gamma(x_1) = \gamma(x_2) \}.
% \end{equation}
\begin{equation}\label{eq:admset}
    \mathcal{A}_0 := \{ \gamma \in H^2_{imm}(\mathbb{T}^1; \mathbb{R}^2) : \text{$N[\gamma] = 1$ and $\gamma$ is not injective} \}.
\end{equation}
The first part of Theorem \ref{thm:1.4} can now be formulated equivalently as
\begin{equation}\label{eq:infimumchar}
    \inf_{\gamma \in \mathcal{A}_0} \bar{B}[\gamma] \geq  C_{2T} = \bar{B}[\gamma_{2T}],
\end{equation}
where a rigorous definition of the minimizer $\gamma_{2T}$ is given in Definition \ref{def:two_teardrop}.
The proof of \eqref{eq:infimumchar} is the goal of this section.

\subsection{Preliminaries about Euler's elasticae}

Before we start we fix an important term that we will use throughout this article.
\begin{definition}\label{def:Eelastica}
A regular curve $\gamma : I \rightarrow \mathbb{R}^n$ is called ($\lambda$-)\emph{elastica} (for some $\lambda \in \mathbb{R}$) if it solves the \emph{elastica equation}
\begin{equation}\label{eq:elastica}
    2\nabla_s^2 \kappa + |\kappa|^2\kappa - \lambda \kappa = 0. 
\end{equation}
 If $\lambda$ is not specified, we simply say \emph{elastica}.
\end{definition}

The elastica equation appears in our context since it describes \emph{critical points} of $\bar{B}$ in $H^2_{imm}( \mathbb{T}^1; \mathbb{R}^2)$ (without any constraint). We notice that critical points of $\bar{B}$ without constraint are automatically smooth, %cf. 
see \cite[Chapter 5]{Eichmann}. 

 In this section we recall some classical preliminaries about those elasticae. 
 The first result already classifies all possible elasticae in $\mathbb{R}^2$ explicitly and exhaustively. See Appendix \ref{app:elliptic} for a brief review on elliptic functions.

\begin{proposition}[{Planar elasticae, see e.g.  \cite[Proposition B.8]{LiYau1}}] \label{prop:elaclassi}
Let $I \subset \mathbb{R}$ be an interval and let $\gamma \in C^\infty(I;\mathbb{R}^2)$ be an elastica with signed curvature $k[\gamma]$.
Then, up to rescaling, reparametrization and isometries of $\mathbb{R}^2$, $\gamma$ is given by one of the following \emph{elastic prototypes}.  
\begin{enumerate}[label={\upshape(\roman*)}]
\item  (Linear elastica) $\gamma$ is a line, $k[\gamma] = 0$. 
\item (Wavelike elastica) There exists $m \in (0,1)$ such that 
\begin{equation}
 \gamma(s) = \begin{pmatrix}
 2 E(\am( s, m),m )-  s  \\-2 \sqrt{m} \cn(s,m) 
\end{pmatrix}  .
 \end{equation}
 Moreover $k[\gamma] = 2 \sqrt{m} \cn(s,m).$
 \item(Borderline elastica)
  \begin{equation}
\gamma (s ) = \begin{pmatrix}
2 \tanh(s) - s  \\ - 2 \sech(s) 
\end{pmatrix}. 
 \end{equation}
 Moreover $k[\gamma] =  2 \sech(s).$ 
 \item (Orbitlike elastica) There exists $m \in (0,1)$ such that 
  \begin{equation}
 \gamma(s) = \frac{1}{m} \begin{pmatrix}
 2 E(\am(s,m),m)  + (m -2)s \\- 2\dn(s,m)
\end{pmatrix}  .
\end{equation}
Moreover $k[\gamma] = 2 \dn(s,m).$  
\item (Circular elastica) $\gamma$ is a circle. In this case $k[\gamma] = \frac{1}{R}$, where $R$ is the radius of the circle. 
\end{enumerate}
\end{proposition}
In both the wavelike and the orbitlike case, the \emph{modulus} $m$ is the main shape parameter for the curve.

Throughout this article we will use several important elasticae, which are listed in the table below. 
\begin{table}[ht]
    \small
    \begin{tabular}{|c|c|c|c|}
\hline 
     Name & Type & Modulus $m$ & Reference \\ \hline
     Figure-eight elastica & (ii): wavelike & $m_8 \simeq 0.8261$ & Definition \ref{def:def_fig_eight}, Figure \ref{fig:2Tand8} \\ \hline
     Teardrop elastica & (ii): wavelike & $m_T \simeq {0.7312}$ & Definition \ref{def:teardropelastica}, Figure \ref{img:teardrop} \\ \hline 
     Heart-shaped elastica & (iv): orbitlike & $m_H \simeq 0.8436$ & Definition \ref{def:heartela}, Figure \ref{img:heart} \\ \hline 
\end{tabular}
\caption{Some important elasticae.}
\label{tab1}
\vspace{-0.5cm}
\end{table}

Since this article studies closed curves it is important to identify closed elasticae. It is classical (%cf.,
see e.g. \cite[Lemma 5.4]{LiYau1}) that only two configurations in $\mathbb{R}^2$ yield closed curves. The first one is given by the circular elastica. The second one is the \emph{figure-eight elastica}, defined as follows.

\begin{definition}\label{def:def_fig_eight}
A smooth curve $\gamma : I \rightarrow \mathbb{R}^2$ is called \emph{figure-eight elastica} if it coincides up to scaling, isometries and reparametrization with 
\begin{equation}\label{eq:def_fig_eight}
    \gamma_8(x) := \begin{pmatrix} 2 E(x,m_8) - F(x,m_8) \\ - 2 \sqrt{m_8} \cos(x) \end{pmatrix} \quad (x \in [0,2\pi]),
\end{equation}
where $m_8 \in (0,1)$ is the unique zero of $m \mapsto 2E(m) - K(m)$ (cf.\
\cite[Lemma B.4]{LiYau1}). The notation $\gamma_8$ will be used exclusively for the specific parametrization in \eqref{eq:def_fig_eight}.
Notice that $k[\gamma_8](x) = 2 \sqrt{m_8} \cos(x)$.
We also define
\begin{equation}\label{eq:defC8}
    C_8:=\bar{B}[\gamma_8] \Big( = 32(2m_8-1)K(m_8)^2 \Big).
\end{equation}
\end{definition}
This is actually a reparametrization of  case (ii) of Proposition \ref{prop:elaclassi} with $m=m_8$. Indeed, $s \mapsto \gamma_8 ( \mathrm{am}(s,m_8) )$ falls into this class. The reason why we choose this different parametrization is that the second component is very easy to express.

Having characterized all closed planar elasticae we can formulate the following result, implying that a minimizer in $\mathcal{A}_0$ cannot be found in the class of elasticae.  

\begin{lemma}\label{lem:ruleoutelastica}
The set $\mathcal{A}_0$ does not contain an elastica.
\end{lemma}

\begin{proof}
    By \cite[Lemma 5.4]{LiYau1} the only closed elasticae with a self-intersection are (up to scaling and isometries) given by $\omega$-fold  circles ($\omega \geq2$) and $\omega$-fold  figure-eight elasticae ($\omega \geq 1$).
    For an $\omega$-fold covering of the circle one readily checks that $N[\gamma] = \omega \geq 2$, which means $\gamma \not \in \mathcal{A}_0$. 
    If $\gamma$ is a (one-fold) figure-eight elastica (%cf.\
   as in Definition \ref{def:def_fig_eight}) one has
    \begin{equation}
        N[\gamma]= N [\gamma_8] =  \frac{\sqrt{m_8}}{\pi} \int_0^{2\pi} \frac{\cos(\theta)}{\sqrt{1- m_8 \sin^2(\theta)}} \; \mathrm{d}\theta = 0.
    \end{equation}
    Hence the rotation number of the figure-eight is zero, and the same holds true for its multiple covers. In particular none of those curves lie in $\mathcal{A}_0$. 
\end{proof}

Even though this result sounds not promising at first sight we will actually conclude many properties of minimizers from the fact that they \emph{cannot} be elasticae.

\subsection{Existence of minimizers and the variational inequality}

In this section we prove existence of minimizers via the direct method. We first examine the structure of the admissible set $\mathcal{A}_0$ defined in \eqref{eq:admset}.

\begin{proposition}\label{prop:admissibleset}
The set $\mathcal{A}_0$ is weakly closed in $H^2_{imm}( \mathbb{T}^1; \mathbb{R}^2)$ (with the weak relative topology of $H^2(\mathbb{T}^1;\mathbb{R}^2)$). 
\end{proposition}
\begin{proof}
    Suppose that $(\gamma_j)_{j \in \mathbb{N}} \subset \mathcal{A}_0$ is a sequence and $\gamma \in H^2_{imm}(\mathbb{T}^1;\mathbb{R}^2)$  such that $\gamma_j \rightharpoonup \gamma$ weakly in $H^2(\mathbb{T}^1;\mathbb{R}^2)$. By Sobolev embedding we have $\gamma_j \rightarrow \gamma$ in $C^1_{imm}(\mathbb{T}^1;\mathbb{R}^2)$. From \cite[Lemma 4.1 and Lemma 4.3]{LiYau1} we infer that the set of noninjective immersions is closed in $C^1_{imm}(\mathbb{T}^1;\mathbb{R}^2)$, and hence $\gamma$ is noninjective.  Thus there exist $x_1,x_2\in \mathbb{T}^1, x_1 \neq x_2$ such that $\gamma(x_1) = \gamma(x_2)$. From the fact that $N[\gamma_j] = 1$ and 
    using that by \cite[Lemma 4.9]{ToriofRev}
     $N[\cdot]$ is weakly continuous in $H^2_{imm}(\mathbb{T}^1;\mathbb{R}^2)$ 
    %(cf.\ \cite[Lemma 4.9]{ToriofRev})
    we infer $N[\gamma] = 1$. All in all we conclude that $\gamma \in \mathcal{A}_0$.
\end{proof}

In the course of the minimization procedure we will make use of the many invariances of $\bar{B}$. Recall that $\bar{B}$ is invariant with respect to scaling, Euclidean isometries and reparametrization. 

\begin{proposition}\label{prop:existence}
There exists $\gamma_0 \in \mathcal{A}_0$ such that 
\begin{equation}
    \bar{B}[\gamma_0] = \inf_{\gamma \in \mathcal{A}_0} \bar{B}[\gamma]. 
\end{equation}
\end{proposition}
\begin{proof}
    Let $(\gamma_j)_{j \in \mathbb{N}} \subset \mathcal{A}_0$ be such that 
$
        \bar{B}[\gamma_j] \rightarrow \inf_{\gamma \in \mathcal{A}_0} \bar{B}[\gamma]. 
$ Since $\bar{B}$ is scaling invariant, we can without loss of generality assume that $L[\gamma_j] = 1$ for all $j \in \mathbb{N}.$ By reparametrization invariance we may as well assume that $|\gamma_j'(x)| = L[\gamma_j] = 1$ for all $x \in \mathbb{T}^1$ and all $j \in \mathbb{N}$. By translation invariance we may assume $\gamma_j(0) = (0,0)$ for all $j$. We show next that $(\gamma_j)_{j \in \mathbb{N}}$ is bounded in $H^2(\mathbb{T}^1;\mathbb{R}^2)$. To this end, observe that
    \begin{equation}
        \bar{B}[\gamma_j] = \int_0^1 |\partial_s^2 \gamma_j|^2 \; \mathrm{d}s = \int_0^1 |\gamma_j''(x)|^2 \; \mathrm{d}x.
    \end{equation}
    This implies that 
 $
        (||\gamma_j''||_{L^2})_{j \in \mathbb{N}} 
  $ is bounded. Moreover, $||\gamma_j'||_{L^2} = L[\gamma_j]=1 $ is also uniformly bounded in $j$. Further, $\gamma_j(0)= (0,0)$ implies
  \begin{equation}
      |\gamma_j(x)| = \left \vert \int_0^x \gamma_j'(y) \; \mathrm{d}y \right\vert \leq L[\gamma_j] = 1
  \end{equation}
  and hence also $(||\gamma_j||_{L^2})_{j \in \mathbb{N}}$ is uniformly bounded in $j$. This yields that $(\gamma_j)_{j \in \mathbb{N}}$ is bounded in $H^2(\mathbb{T}^1;\mathbb{R}^2)$. We can now extract a subsequence (which we do not relabel) such that $\gamma_j \rightharpoonup \gamma_0$ for some $\gamma_0 \in H^2(\mathbb{T}^1;\mathbb{R}^2)$. By Sobolev embedding one has also $\gamma_j \rightarrow \gamma_0$ in $C^1(\mathbb{T}^1;\mathbb{R}^2)$. We now claim that $\gamma_0 \in H^2_{imm}(\mathbb{T}^1;\mathbb{R}^2)$. Indeed, one has for all $x \in \mathbb{T}^1$
 \begin{equation}
     |\gamma_0'(x)| = \lim_{j \rightarrow \infty}|\gamma_j'(x)|  = \lim_{j \rightarrow \infty} L[\gamma_j] = 1.
 \end{equation}
 In particular, $\gamma_0 \in H^2_{imm}( \mathbb{T}^1; \mathbb{R}^2)$ is parametrized by arclength and $L[\gamma_0] = 1$. Moreover, by Proposition \ref{prop:admissibleset} we infer that $\gamma_0 \in \mathcal{A}_0$. 
 In addition, weak lower semicontinuity of the $L^2$-norm implies
 \begin{align}
     \bar{B}[\gamma_0]& = \int_0^1 |\gamma_0''(x)|^2 \; \mathrm{d}x \leq \liminf_{j \rightarrow \infty} \int_0^1 |\gamma_j''(x)|^2 \; \mathrm{d}x = \liminf_{j \rightarrow \infty} \bar{B}[\gamma_j] = \inf_{\gamma \in \mathcal{A}_0} \bar{B}[\gamma]. 
 \end{align}
 Therefore $\gamma_0$ is a minimizer. 
\end{proof}
 In the following we will mainly examine a broader class than the class of minimizers --- namely solutions of the  \emph{variational inequality}, defined as follows. 
\begin{definition}[Variational inequality]\label{def:cript} A curve $\gamma \in \mathcal{A}_0$ is called a solution to the  \emph{variational inequality} of $\bar{B}$ if 
\begin{equation}\label{eq:pertfunc}
    \frac{\mathrm{d}}{\mathrm{d}\varepsilon}  \Big\vert_{\varepsilon = 0} \bar{B}[\gamma_\varepsilon] \geq 0 \; \; \quad  \textrm{for all} \; (\varepsilon \mapsto \gamma_\varepsilon) \in C^1([0,\varepsilon_0); \mathcal{A}_0 ) \; %\textrm{such that} 
    \textrm{with} \; \gamma_0 = \gamma,
\end{equation}
where $C^1([0,\varepsilon_0);\mathcal{A}_0)$ is the set of all perturbations 
$$(\varepsilon \mapsto \gamma_\varepsilon) \in C^1([0,\varepsilon_0);H^2_{imm}( \mathbb{T}^1, \mathbb{R}^2))$$ 
such that $\gamma_\varepsilon \in \mathcal{A}_0$ for all $\varepsilon \in [0, \varepsilon_0).$

\end{definition}
In the sequel we will only use linear perturbations of the form $\gamma_\epsilon = \gamma + \epsilon \phi\in\mathcal{A}_0$.
By the Frechet differentiability of $L$ and $B$, any solution $\gamma$ to \eqref{eq:pertfunc} satisfies that for all $\phi \in C^\infty(\mathbb{T}^1;\mathbb{R}^2)$ such that $\gamma + \varepsilon \phi \in \mathcal{A}_0$ for any small $\varepsilon>0$,
\begin{equation}
    \frac{\mathrm{d}}{\mathrm{d}\varepsilon}  \Big\vert_{\varepsilon = 0} \bar{B}[\gamma_\varepsilon] = L[\gamma] DB[\gamma](\phi) + B[\gamma] DL[\gamma](\phi) \geq 0.
\end{equation}
It is obvious that each minimizer $\gamma_0 \in \mathcal{A}_0$ solves the variational inequality. Solutions of the variational inequality can be seen as `critical points' of the energy $\bar{B}$ in a generalized sense.   

In the context of a standard critical point one would usually expect an equality statement in \eqref{eq:pertfunc} and also allow for negative values of $\varepsilon$ in the perturbations. There is no need for that ---  a perturbation in the direction of $\phi$ with a negative value of $\varepsilon$ corresponds to a perturbation with $-\phi$ with a positive value of $\varepsilon$. In our context it is important to distinguish between perturbations with $\phi$ and $-\phi$, since it may happen that only one of these is admissible in $\mathcal{A}_0$.
We stress in this context that if we have a perturbation curve $(\varepsilon \mapsto \gamma_\varepsilon) \in C^1((-\varepsilon_0,\varepsilon_0); \mathcal{A}_0)$ with $\gamma_0 = \gamma$ we infer 
    \begin{equation}\label{eq:varineqpm}
    0 =  \frac{\mathrm{d}}{\mathrm{d}\varepsilon} \Big\vert_{\varepsilon = 0} \bar{B}[\gamma_\varepsilon].  
\end{equation}

If $\gamma \in \mathcal{A}_0$ is not an inner point of $\mathcal{A}_0$ in the $H^2$-topology, some perturbations are not allowed in \eqref{eq:pertfunc}, which means that standard Euler-Lagrange methods and regularity theory might not apply.
It will actually turn out that no minimizer $\gamma_0\in \mathcal{A}_0$ is an inner point. This is why the minimizer $\gamma_{2T}$ will not be a (global) solution of the elastica equation.  

The following lemma characterizes which perturbations are sufficient to conclude that the elastica equation is solved. 

    \begin{lemma}[{%cf.\ 
    see \cite[Proof of Lemma 5.8]{LiYau1}}] \label{lem:localize} Let $\gamma \in H^2_{imm}((a,b);\mathbb{R}^2)$. 
    Then the following statements are equivalent.
    \begin{enumerate}[label={\upshape(\roman*)}]
     \item {For all $\phi \in C_0^\infty ( (a,b); \mathbb{R}^2)$ one has
        \begin{equation}
           L[\gamma] D B [\gamma](\phi)  +  B[\gamma] DL[\gamma](\phi) \geq 0.
        \end{equation}}
        \item  For all $\phi \in C_0^\infty ( (a,b); \mathbb{R}^2)$ one has
        \begin{equation}
           L[\gamma] D B [\gamma](\phi)  +  B[\gamma] DL[\gamma](\phi) = 0.
        \end{equation}
        \item For all $x \in (a,b)$ there exists an open neighborhood $U_x \subset (a,b)$ such that for all $\phi \in C_0^\infty(U_x;\mathbb{R}^2)$ one has  
        \begin{equation}
           L[\gamma] D B [\gamma](\phi)  +  B[\gamma] DL[\gamma](\phi) = 0.
        \end{equation}
    \end{enumerate}
    If one of the above statements holds true then $\gamma \in C^\infty_{imm}({[a,b]};\mathbb{R}^2)$ and $\gamma$ solves the elastica equation \eqref{eq:elastica} on $[a,b]$ for $\lambda = \frac{B[\gamma]}{L[\gamma]}$. 
    The analogous statement remains true if one replaces $(a,b)$ by $\mathbb{T}^1$.
    \end{lemma}
    
    Using these findings we will characterize solutions of the variational inequality. 

\subsection{Self-intersection properties and regularity of solutions to the variational inequality}
In this section we study some properties of solutions of \eqref{eq:pertfunc} concerning self-intersection. Precisely, we will prove that each solution to \eqref{eq:pertfunc} may have only one tangential self-intersection. The arguments used in this section are similar to \cite[Section 5]{LiYau1}.

For arbitrary $\gamma \in \mathcal{A}_0$ we introduce the notation 
\begin{equation}
    S[\gamma] := \{ p \in \mathbb{R}^2 : \mathcal{H}^{0}(\gamma^{-1}( \{p \})) > 1 \},
\end{equation}
where $\mathcal{H}^0$ denotes the counting measure. For $p \in S[\gamma]$ we define the quantity 
$$\mathrm{mult}[\gamma](p) := \mathcal{H}^{0} ( \gamma^{-1} ( \{ p \} )).$$
Moreover the set of \emph{tangential} self-intersections is denoted by
\begin{equation}
    S_{\tan}[\gamma] := \{ p \in S[\gamma] : \mathrm{det}(\gamma'(x_1), \gamma'(x_2)) = 0 \;  \textrm{for some} \; x_1\neq x_2 \in \gamma^{-1}(\{ p \} ) \} .
\end{equation}
Notice that $\mathrm{det}(\gamma'(x_1),\gamma'(x_2)) = 0$ yields (by linear dependence of %$\{\gamma'(x_1),\gamma'(x_2) \}$
$\gamma'(x_1)$ and $\gamma'(x_2)$)
that $T_\gamma(x_1) = \pm T_\gamma(x_2)$, where {$T_\gamma = \frac{\gamma'}{|\gamma'|}$} denotes the \emph{unit tangent} of $\gamma$. 
In this section we will prove

 \begin{proposition}\label{prop:selfinters}
    Let $\gamma_0 \in \mathcal{A}_0$ be a solution to \eqref{eq:pertfunc}. 
    Then 
    $S[\gamma_0] = S_{\tan}[\gamma_0] = \{ p \}$ for some $p \in \mathbb{R}^2$ and $\mathrm{mult}[\gamma](p) = 2$.
    In addition, for the two distinct points $a,b\in\gamma_0^{-1}(\{p\})$, the curves $\gamma_0\vert_{[a,b]}$ and $\gamma_0\vert_{[b,a]}$ are smooth and solve the elastica equation.  
    (In particular, $\gamma_0\vert_{[a,b)}$ and $\gamma_0\vert_{[b,a)}$ are injective.) Moreover, $T_{\gamma_0}(a) = -T_{\gamma_0}(b)$.
    \end{proposition}

We {interpret} here $[a,b]$ in a standard way if $a < b$ in $[0,1)$ and otherwise {we consider} $[a,b+1]$, in accordance with the identification $\mathbb{T}^1 \simeq \mathbb{R}/ \mathbb{Z}$.

The above proposition characterizes the self-intersection properties and regularity of solutions of \eqref{eq:pertfunc} --- in an optimal way!
Indeed, we have already shown in Lemma \ref{lem:ruleoutelastica} that there must remain at least one exceptional point where the elastica equation is not solved.
    
We start with some preparations for the proof of Proposition \ref{prop:selfinters}. To this end, we first look at perturbations that do not affect the set of self-intersections.  

    \begin{lemma}\label{lem:local}
    Suppose that $\gamma_0 \in \mathcal{A}_0$ is a solution to \eqref{eq:pertfunc}, $x \in \mathbb{T}^1$ and $\gamma_0(x) \not \in S[\gamma_0]$. Then there exists an open neighborhood $U_x \subset \mathbb{T}^1$ of $x$ such that 
    \begin{equation}
        \mbox{$L[\gamma_0] DB[\gamma_0](\phi) + B[\gamma_0] DL[\gamma_0](\phi) = 0$ for all $\phi \in C_0^\infty(U_x; \mathbb{R}^2)$.}
    \end{equation}
    \end{lemma}
    \begin{proof}
    The proof follows the lines of \cite[Lemma 5.7]{LiYau1}, with the tiny additional difficulty that the rotation number needs to be discussed. Since $\gamma_0^{-1}(S[\gamma_0])$ is closed, there exists $U_x \subset \mathbb{T}^1$, an open neighborhood of $x$, such that for all $\phi \in C_0^\infty(U_x)$ and $\varepsilon \in \mathbb{R}$ the perturbed curve
    $\gamma_0 + \varepsilon \phi$ has a self-intersection.
    The fact that $N[\cdot]$ is integer-valued and $H^2_{imm}$-continuous implies also that $N[\gamma_0+ \varepsilon \phi] =1$ for $|\varepsilon|$ suitably small and fixed $\phi \in C_0^\infty(U_x)$.
    In particular $\gamma_0 + \varepsilon \phi \in \mathcal{A}_0$ for such $\varepsilon$ and $\phi$.
    By \eqref{eq:varineqpm} we conclude
        \begin{equation}
            0 = \frac{\mathrm{d}}{\mathrm{d}\varepsilon} \Big\vert_{\varepsilon = 0} L[\gamma_0 + \varepsilon \phi] B[\gamma_0+ \varepsilon \phi ] = L[\gamma_0] DB[\gamma_0](\phi) + B[\gamma_0] DL[\gamma_0](\phi). \qedhere
        \end{equation}
    \end{proof}
    
    This implies that the elastica equation is solved at each point that is not a point of self-intersection. 
    
    \begin{proof}[Proof of Proposition \ref{prop:selfinters}]
    Let $\gamma_0 \in \mathcal{A}_0$ be a solution to \eqref{eq:pertfunc}. The proof is divided {into} several steps.
    
    \textbf{Step 1:} We show $S[\gamma_0] = \{p \}$ for some $p \in \mathbb{R}^2$.
    To prove this we follow the lines of \cite[Lemma 5.8]{LiYau1}.
      Assume that there exist two distinct points $p, q \in S[\gamma_0]$.  Fix $x \in \mathbb{T}^1$. Then either $\gamma_0(x) \neq p$ or $\gamma_0(x) \neq q$. Without loss of generality we may assume that $\gamma_0(x) \neq p$. Since $\gamma_0^{-1}(\{ p \}) \subsetneq \mathbb{T}^1$ is closed one can find an open neighborhood $U_x$ of $x$ such that $U_x \cap \gamma_0^{-1}(\{ p \}) = \emptyset$. One readily checks that for each $\phi \in C_0^\infty(U_x; \mathbb{R}^2)$ there holds $\gamma_0 + \varepsilon \phi \in \mathcal{A}_0$ for $|\varepsilon|$ suitably small (since $p \in S[\gamma_0 + \varepsilon \phi]$ and $N[\gamma_0 + \varepsilon \phi] = 1$ for $|\varepsilon| \ll 1$). %This at hand
      With this in hand
      we compute by \eqref{eq:varineqpm} that for all $\phi \in C_0^\infty(U_x;\mathbb{R}^2)$ there holds
      \begin{equation}
          0 = \frac{\mathrm{d}}{\mathrm{d}\varepsilon} \Big\vert_{\varepsilon = 0} L[\gamma_0 + \varepsilon \phi] B[\gamma_0+ \varepsilon \phi ] = L[\gamma_0] DB[\gamma_0](\phi) + B[\gamma_0] DL[\gamma_0](\phi).
      \end{equation}
      Since $x \in \mathbb{T}^1$ was arbitrary one concludes by Lemma \ref{lem:localize} that $\gamma_0 \in \mathcal{A}_0$ is smooth and solves the elastica equation. 
      This is a contradiction to Lemma \ref{lem:ruleoutelastica}.
      
      \textbf{Step 2:} We show $\mathrm{mult}[\gamma_0](p)=2$ for the unique point $p \in S[\gamma]$. 
      To show this we assume $\mathrm{mult}[\gamma_0](p)\geq 3$. Then each $x \in \mathbb{T}^1$ has an open neighborhood $U_x$ that satisfies condition (iii) of Lemma \ref{lem:localize}, since $U_x$ can be taken so small that $\gamma^{-1}(p)\setminus U_x$ contains at least two points (cf.\ \cite[Lemma 5.9]{LiYau1}).
      Thereupon, Lemma \ref{lem:localize} yields that $\gamma_0$ is smooth and solves the elastica equation. This is again a contradiction to Lemma \ref{lem:ruleoutelastica}.
      
      \textbf{Step 3:} We show $S[\gamma_0] = S_{\tan}[\gamma_0]$.
      If we assume that the unique self-intersection point $p \in \mathbb{R}^2$ is non-tangential, any small perturbation keeps the self-intersection so that $\gamma_0$ solves the elastica equation (%cf.\ 
      see also \cite[Lemma 5.12]{LiYau1}).
      This is again a contradiction.
      We have shown that $S[\gamma_0] = S_{\tan}[\gamma_0] = \{ p \}$ for a singleton $p \in \mathbb{R}^2$ with $\gamma_0^{-1}(\{p\}) = \{ a,b \}$ for two distinct values $a,b \in \mathbb{T}^1$.
      
      \textbf{Step 4:} We show that the curves $\gamma_0\vert_{[a,b]}$ and $\gamma_0\vert_{[b,a]}$ are smooth elasticae (which are trivially injective except at their endpoints).
      Indeed, since $\gamma_0(x) \not \in S[\gamma_0]$ for all $x \in (a,b)$ and all $x \in (b,a+1)$ one infers from Lemma \ref{lem:local} that point (iii)  of Lemma \ref{lem:localize} holds true on $[a,b]$ and $[b,a+1]$. Using Lemma \ref{lem:localize} we obtain the claim.
      
    \textbf{Step 5}: We show $T_{\gamma_0}(a) = - T_{\gamma_0}(b).$ By Step 3 one already has $T_{\gamma_0}(a) = \pm T_{\gamma_0}(b)$.
    Assume that ``$\pm=+$''. Choose a reparametrization of $\gamma_0$ with constant speed, which we call again $\gamma_0$ by abuse of notation.  One readily checks (cf.\  \cite[Lemma A.6]{LiYau1}) that $\gamma_{0} \in \mathcal{A}_0$.
    Moreover, we infer from our assumption that 
    $\gamma_{0}(a) = \gamma_{0}(b)$ and  $ \gamma_{0}'(a) = \gamma_{0}'(b).$
    In particular $\gamma_{01} := \gamma_{0}\vert_{[a,b]}$ and $\gamma_{02} := \gamma_{0}\vert_{[b,a]}$ are two $C^1$-closed curves. Notice that suitable reparametrizations of  both such curves lie in $H^2_{imm}(\mathbb{T}^1, \mathbb{R}^2)$.
    Since $\gamma_{0}$ may not have self-intersections except for $\gamma_{0}(a) = \gamma_{0}(b) = p$ we obtain that  $\gamma_{01}$ and $\gamma_{02}$ are closed embedded curves. 
    By Hopf's Umlaufsatz (%cf.\ 
   see also \cite[Lemma A.5]{LiYau1}) one infers that
    $N[\gamma_{0i}] =  1$, that is, $\frac{1}{2\pi}\int_{\gamma_{0i}} k \; \mathrm{d}s\in\{\pm1\}$ for $i=1,2$, and hence $N [\gamma_{0}]=|\frac{1}{2\pi} \int_{\gamma_0}  k \; \mathrm{d}s| = |\frac{1}{2\pi} \int_{\gamma_{01}}  k \; \mathrm{d}s + \frac{1}{2\pi}\int_{\gamma_{02}} k \; \mathrm{d}s| \in\{0,2\}$.
    This is a contradiction to $N[\gamma_0] = 1$ as $\gamma_0 \in \mathcal{A}_0$.
    \end{proof}
    
    An important consequence of Proposition \ref{prop:selfinters} is that each solution of \eqref{eq:pertfunc} $\gamma_0 \in \mathcal{A}_0$ is composed of two \emph{embedded cuspidal elasticae}, defined as follows.

\begin{definition}[Embedded cuspidal elastica: ECE]\label{def:selftouchdrop}
We call a smooth curve $\gamma  \in C^\infty_{imm}([a,b]; \mathbb{R}^2)$ an \emph{embedded cuspidal elastica (for short: ECE)} if $\gamma$ is an elastica such that $\gamma\vert_{[a,b)}$ is injective, $\gamma(a) = \gamma(b)$, and $T_\gamma(a) = -T_\gamma(b)$.
\end{definition}

The ECE property already gives a pretty explicit characterization of the solutions to the variational inequality --- we will be able to classify all ECEs. This will reduce the amount of candidates for solutions dramatically. In order to characterize solutions of \eqref{eq:pertfunc} exhaustively, we need to understand more about the regularity at the unique self-intersection point $p = \gamma(a) = \gamma(b)$ determined in Proposition \ref{prop:selfinters}. We will derive an optimal global regularity statement that can be understood as a coupling condition. 
\begin{lemma}[Global regularity, %cf.\
see also Appendix \ref{app:optireg}]\label{lem:optireg}
    Each solution $\gamma \in \mathcal{A}_0$ of the variational inequality \eqref{eq:pertfunc} has a %(constant-speed-)reparametrization
    reparametrization (of constant speed)
    that lies in $W^{3,\infty}( \mathbb{T}^1; \mathbb{R}^2)$. In particular, $k[\gamma] \in C^0( \mathbb{T}^1; \mathbb{R}).$
\end{lemma}
\begin{proof}[Sketch of Proof]
 The $W^{3,\infty}$-regularity follows essentially by the same principle as Dall'Acqua--Deckelnick's proof for an obstacle problem \cite[Theorem 5.1]{AnnaObst}, which obtains regularity from \emph{one-sided perturbations}. 
 In fact, around the unique tangential self-intersection, the curve is represented by two graphs, and each of them allows one-sided perturbations, in the direction that maintains self-intersections. 
 A crucial implication of this is that $D\bar{B}$ can locally be represented by a Radon measure.
 If this Radon measure is finite, standard techniques yield the desired regularity.
 In \cite[Theorem 5.1]{AnnaObst} this finiteness follows from an obstacle condition, while in our situation it does from the self-intersection properties, see Appendix \ref{app:optireg} for details.
\end{proof}
We have obtained an additional coupling condition at the self-intersection point of a solution of \eqref{eq:pertfunc}. All in all, each solution consists of two ECEs whose curvatures match up at the endpoints.

\subsection{Classification of ECEs}\label{sec:ECE}

Our goal in the next section is to characterize all ECEs. 
The main tool we will use is the explicit parametrization of planar elasticae, given in Proposition \ref{prop:elaclassi}. 

Before we start with our search for ECEs, we can rule out the prototypes (i), (iii), and (v) and all their rescalings, reparametrizations and isometric images: The linear case (i) and  the circular case (v) are obvious, while the borderline case (iii) can also be ruled out immediately by the fact that the \emph{tangential angle} $\theta := \mathrm{arg}( T_\gamma)\in [0,2\pi]$ is strictly increasing between $0$ and $2\pi$, %cf.\
see \cite[Eq. (3.6)]{TatsuyaPhase}. 
Indeed, if the borderline elastica $\gamma$ had a self-intersection with antipodal tangents at $p$ and $\gamma^{-1}( \{p \}) = \{ a,b \}$, then $\theta(b)- \theta(a) = \pm \pi$, implying that $\gamma\vert_{[a,b]}$ can be represented (after rotation) as a graph of a convex function. But this contradicts the assumption that $\gamma(a) = \gamma(b)$.

We now examine the wavelike case and the orbitlike case in Sections \ref{sec:wavECE} and \ref{sec:orbECE}, respectively.

\subsubsection{Wavelike ECEs}\label{sec:wavECE}

We prove in this section that there exists (up to scaling, reparametrization and isometries of $\mathbb{R}^2$) only one wavelike ECE --- the teardrop elastica, cf.\ Figure \ref{img:teardrop}.

By Proposition \ref{prop:elaclassi} the modulus $m \in (0,1)$ characterizes a wavelike elastica uniquely up to scaling, reparametrization and isometries of $\mathbb{R}^2$. We will show that only one modulus $m= m_T$ leads to an ECE. For notational simplicity we define 
\begin{equation}\label{eq:def alpha}
    \alpha(m):= \arcsin\sqrt{\frac{1}{2m}}\in (0,\tfrac{\pi}{2}] \quad (m\geq \tfrac{1}{2}).
\end{equation}

The modulus $m_T$ is characterized as the unique root of 
\begin{equation}\label{eq:fwavelike}
    f: \big[\tfrac{1}{2},1\big) \rightarrow \mathbb{R} , \quad f(m) := \int_0^{\pi- {\alpha(m)} } \frac{1- 2m \sin^2\theta}{\sqrt{1- m \sin^2 \theta}} \; \mathrm{d} \theta.
\end{equation}
Existence and uniqueness of $m_T$ follow from

\begin{proposition}[Proof in Appendix \ref{app:compu}]\label{prop:monotonequantity}
For all $m \in (\frac{1}{2},1)$ one has $f'(m) < 0$. Moreover, 
  $  f(\frac{1}{2}) > 0$ and $ f(m_8) < 0$, where $m_8 \in (0,1) $  is the % unique 
  root of $m \mapsto 2 E(m)- K(m)$ (%cf.\ 
  which exists and is unique due to \cite[Lemma B.4]{LiYau1}). In particular there exists a unique $m_T \in ( 0,1) $ such that $f(m_T) = 0$. Moreover, $m_T \in ( \frac{1}{2}, m_8) $.
\end{proposition}
 The numerical value of $m_T$ is $m_T \simeq 0.7312$, cf. Table \ref{tab1}.
 
In this section we will often fix a parametrization of wavelike elasticae that differs from the one in Proposition \ref{prop:elaclassi}. Namely, we define
\begin{equation}\label{eq:wavelikeparabulg}
    \gamma(x | m) := \begin{pmatrix} 2E(x,m) - F(x,m) \\ -2\sqrt{m} \cos(x)   \end{pmatrix}  \quad (x \in \mathbb{R}).
\end{equation}
for some fixed $m \in (0,1)$. Notice that $s \mapsto \gamma(\mathrm{am}(s,m) | m)$ exactly yields the prototypical wavelike elastica in Proposition \ref{prop:elaclassi}. In this way $\gamma(\cdot | m)$ enjoys `(anti)periodic behavior', i.e.\ for any $m \in (0,1)$ and $x \in \mathbb{R}$,
\begin{equation}\label{eq:antiperiod}
\gamma(x + \pi  | m ) = \begin{pmatrix} \gamma^{(1)}(x | m) \\ -\gamma^{(2)}(x | m)  \end{pmatrix} + \begin{pmatrix} 2(2E(m)- K(m)) \\ 0  \end{pmatrix},
\end{equation}
and hence also
\begin{equation}\label{eq:gammaperiod}
    \gamma(x + 2\pi  | m ) = \gamma(x | m) + \begin{pmatrix} 4(2E(m)- K(m)) \\ 0  \end{pmatrix}. 
\end{equation}
The main advantage of our chosen parametrization is now that the period does not depend on the modulus $m$. 

For the proofs to come it is convenient to define for $x \in \mathbb{R}$ and $m \in (0,1)$,
\begin{align}
\begin{split}
    G(x,m) := \gamma^{(1)}(x | m) & = 2 E(x,m) - F(x,m) \\
   &=  \int_0^x \frac{1-2m \sin^2 \theta}{\sqrt{1-m \sin^2\theta}} \; \mathrm{d}\theta \label{eq:2.7}.
\end{split}
\end{align}

In the sequel we will use many properties of $G$, summarized in the following 

\begin{lemma}
\label{lem:B2}
For all $m \in (0,1)$, $l\in \mathbb{Z}$, and $x \in \mathbb{R}$ there holds
    \begin{enumerate}[label={\upshape(\roman*)}]
        \item $G(-x,m) = -G(x,m)$; 
        \item $G(x+l\pi,m) = G(x,m) + 2l (2E(m) - K(m)) = G(x,m)+ G(l\pi, m)$; 
        \item $G( \frac{\pi}{2},m) =  2 E(m) - K(m)$;
        \item if $m <m_T$, then $G(x,m) = 0$ implies $x= 0$; 
        \item the equation $G(x,m_T) = 0$ has exactly three solutions: $x= 0$ and \\ $x = \pm \left( \pi - {\alpha(m)} \right)$.
    \end{enumerate}
\end{lemma}

\begin{proof}
Statements (i), (ii), (iii) are immediate {using Proposition \ref{prop:identities}}. 
We prove (iv) and (v), thus assuming $m\leq m_T$ throughout.
Clearly $G(0,m)=0$.
Since if $m<\frac{1}{2}$ ($<m_T$) $G(\cdot,m)$ is strictly increasing (see $\partial_xG$ below) and thus (iv) is trivial, we may hereafter assume that $m\geq\frac{1}{2}$.
In view of symmetry in  (i), it is sufficient to prove that  $m\in[\frac{1}{2},m_T)$ {implies} $G(x,m)>0$ for all $x>0$, while if $m=m_T$ then $\{x>0\mid G(x,m_T)=0\}=\{\pi-\alpha(m)\}$.
We compute
 \begin{align}
     \partial_x G(x,m) & = \frac{1-2m\sin^2(x)}{\sqrt{1-m \sin^2(x)}} 
     \\ &  \begin{cases}
       = 0 & x= k \pi \pm \alpha(m) \quad (k\in\mathbb{Z}),\\
       > 0 & x \in  (k \pi - \alpha(m) , k \pi +  \alpha(m)) \quad  (k \in \mathbb{Z}), \\
       <  0 &  x \in (k \pi + \alpha(m) , (k+1) \pi -\alpha(m)) \quad (k \in \mathbb{Z}).
     \end{cases}
 \end{align}
The following key behavior becomes visible: $G(\cdot,m)$ is strictly increasing on $(0,\alpha(m))$, decreasing on $(\alpha(m),\pi-\alpha(m))$, and again increasing on $(\pi-\alpha(m),\pi)$.
By Proposition \ref{prop:monotonequantity} we deduce that $G(\pi-\alpha(m),m)\geq0$ (since $m\leq m_T$) with equality if and only if $m=m_T$.
Hence $G(x,m)\geq\min\{G(0,m),G(\pi-\alpha(m),m)\}=0$ for all $x\in(0,\pi)$, and equality holds if and only if $m=m_T$ and $x=\pi-\alpha(m_T)$.
Now it is sufficient to show that $G(x,m)>0$ for all $x\geq\pi$.
Let $x\in[k\pi,(k+1)\pi]$ with a positive integer $k\geq1$.
By the above behavior of $G$ on $[0,\pi]$ it is clear that $G(\pi,m)>0$.
By property (ii) and by the fact that $G(\cdot,m)\geq0$ on $[0,\pi]$,
$$G(x,m)=G(x-k\pi,m)+2k(2E(m)-K(m))\geq2k(2E(m)-K(m)).$$
Then by the estimate $m\leq m_T<m_8$ in Proposition \ref{prop:monotonequantity}, and by the fact that {$2E(m)- K(m) > 0$ for all $m < m_8$ (cf. \cite[Proof of Lemma B.4]{LiYau1})},
we deduce that
$G(x,m)>0$ for any $x\geq\pi$.
The proof is now complete.
\end{proof}

We next define the teardrop elastica rigorously.

\begin{definition}[Teardrop elastica]\label{def:teardropelastica}
 Let $a_T := - \pi +{\alpha(m_T)}$ and $b_T := \pi - {\alpha(m_T)}$. Then $\gamma_T := \gamma(\cdot  | m_T) \vert_{(a_T,b_T)} \in C^\infty_{imm} ( [a_T, b_T]; \mathbb{R}^2) $ is called \emph{teardrop elastica}. We will also call rescalings, isometric images and reparametrizations teardrop elasticae. However we will use the notation $\gamma_T$ only for the curve defined above.  
\end{definition} 

\begin{proposition}[Existence of wavelike ECEs]\label{prop:2.16}
Each teardrop elastica is an ECE.
\end{proposition}

\begin{proof}
It suffices to show that $\gamma_T$ is an ECE.
We first compute that 
$$\gamma(a_T | m_T) = \gamma(b_T | m_T).$$
Indeed, Lemma \ref{lem:B2} (v) and \eqref{eq:2.7} yield 
$\gamma^{(1)} (a_T | m_T) = - \gamma^{(1)}(b_T | m_T) = 0,$ 
while properties of $\cos$ yield that 
$\gamma^{(2)}(a_T | m_T) = \gamma^{(2)}(b_T | m_T).$
Next we look at $\gamma'(\cdot | m_T)$. Observe that by \eqref{eq:2.7}, the definition of $b_T$, {\eqref{eq:def alpha}} and $\sin^2(b_T) = \frac{1}{2m_T}$,
\begin{equation}
    (\gamma^{(1)})'(b_T |  m_T) = \tfrac{1 - 2m_T\sin^2(x)}{\sqrt{1-m_T \sin^2(x) }} \Big\vert_{x= b_T} = 0.
\end{equation}
Analogously, one shows $(\gamma^{(1)})'(a_T | m_T) =  0$. 
Now note that $(\gamma^{(2)})'(x | m_T) = 2 \sqrt{m_T} \sin(x)$ and hence $(\gamma^{(2)})'(a_T | m_T) = - (\gamma^{(2)})'(b_T | m_T)$. 
We thus find that $\gamma'(a_T | m_T) = - \gamma'(b_T | m_T)$ and hence 
$$T_{\gamma(\cdot | m_T)}(a_T)  = -T_{\gamma(\cdot | m_T)}(b_T).$$ 

Finally we show that $\gamma(\cdot | m_T)$ is embedded on $[a_T,b_T)$. To this end, assume that there exist $x_1,x_2 \in [a_T,b_T)$, $x_1 < x_2$, such that $\gamma(x_1 | m_T) = \gamma(x_2  | m_T)$.
By definition of $a_T, b_T$ one has $-\pi < x_1 < x_2 <\pi$. Since $\gamma^{(2)}(x_2 | m_T) = \gamma^{(2)}(x_1 | m_T)$, i.e.\ $\cos(x_1)=\cos(x_2)$, one has $x_2 = -x_1 > 0$.
Since $\gamma^{(1)}(x_1 | m_T) = \gamma^{(1)}(x_2 | m_T) = \gamma^{(1)}(-x_1 | m_T)$ and $\gamma^{(1)}(\cdot | m_T)$ is odd, we infer that $\gamma^{(1)}(x_1 | m_T) = \gamma^{(1)}(x_2 | m_T) = 0.$
Hence  $x_2 \in (0, b_T)$ satisfies $G(x_2,m_T) =0$.
By Lemma \ref{lem:B2} (v) however $G(\cdot,m_T)=0$ has no solution in $(0,b_T)$. 
This is a contradiction.
\end{proof}

The rest of this section is devoted to the proof of the following fact.

\begin{proposition}[Uniqueness of wavelike ECEs]\label{prop:wavelikeunique} Let $a<b$ and suppose that $\gamma  \in C^\infty_{imm}([a,b];\mathbb{R}^2)$ is a wavelike ECE. Then $\gamma$ is a teardrop elastica.
\end{proposition}

Before the proof we need some preparatory lemmas.

\begin{lemma}\label{lem:smallmodulus}
Let $m < m_T$. Then $\gamma(\cdot | m)$ given by \eqref{eq:wavelikeparabulg} does not have any self-intersection on $\mathbb{R}$.  
\end{lemma}
\begin{proof}
Let $m < m_T$. We show that $\gamma(\cdot  | m)$ is injective on $\mathbb{R}$. 
We may without loss of generality assume that $m\geq \frac{1}{2}$ since for $m < \frac{1}{2}$, $\gamma^{(1)}(\cdot | m)$ is strictly increasing and hence injective. Thus from now on $m \in [\frac{1}{2}, m_T)$. 
 For a contradiction assume that there exist $x_1,x_2 \in \mathbb{R}$, $x_1 \neq x_2$ such that $\gamma(x_1 | m) = \gamma(x_2 | m)$. By comparing first and second components we infer from \eqref{eq:wavelikeparabulg} that
 $G(x_1,m) = G(x_2,m)$ and $\cos(x_1) = \cos(x_2)$.
 The latter equation yields $x_2 = \pm x_1 + 2l\pi$ for some $l \in \mathbb{Z}$. Now Lemma \ref{lem:B2} (i),(ii) implies
 \begin{align}
 \begin{split}
     G(x_1,m) = G(x_2,m) &= G(\pm x_1 + 2l\pi,m) \\
     &= \pm G(x_1,m) + 4l(2 E(m) - K(m)) . 
 \end{split}    
 \end{align}
 In the case of ``$\pm = + $'' we obtain $0 = 4l (2E(m) -K(m))$. However, $x_1 \neq x_2$ yields $l \neq 0$ and hence we infer that $2E(m) - K(m)= 0$. This implies $m = m_8$, which contradicts $m<m_T < m_8$, %cf.\
 due to Proposition \ref{prop:monotonequantity}. In the case of ``$\pm = -$'' we obtain
 $
     G(x_1,m)= 2 l(2 E(m) -K(m)) = G(l \pi,m). 
 $
 Using once more Lemma \ref{lem:B2} (ii) we infer that $G(x_1- l \pi,m ) = 0.$ We infer from Lemma \ref{lem:B2} (iv) that $x_1 - l \pi = 0$. However then $x_2 = - x_1 + 2l\pi= l\pi =x_1$, a contradiction. 
\end{proof}

\begin{lemma}\label{lem:largemodulus}
Let $m > m_T$. Then there exist $x_1, x_2 \in [0, 2\pi],$ $x_1 \neq x_2$ such that $\gamma(x_1 | m) = \gamma(x_2 | m).$
\end{lemma}

\begin{proof}
Since $m > m_T$ we infer from Proposition \ref{prop:monotonequantity} that 
\begin{equation}
     G\left( \pi - {\alpha(m)}, m \right) =  \int_0^{\pi - {\alpha(m)}} \frac{1-2m\sin^2 \theta}{ \sqrt{1- m \sin^2 \theta}} \; \mathrm{d}\theta  < 0.
\end{equation}
Since the integrand is positive for small $\theta > 0$ we infer that there must exist $y \in (0 , \pi - {\alpha(m)})$ such that 
$
    G(y,m) = 0
$
by continuity.
We claim that $x_1 = \pi - y$ and $x_2 = \pi + y$ yield a self-intersection. First note that 
$\gamma^{(2)}(\pi - y |  m) = - 2 \sqrt{m} \cos(\pi - y) = - 2 \sqrt{m} \cos( \pi + y) = \gamma^{(2)}( \pi + y  | m )  $. We conclude by Lemma \ref{lem:B2} (i),(ii)
\begin{align}
    \gamma^{(1)}(\pi + y | m) - \gamma^{(1)}(\pi - y | m )  &= G(\pi + y,m) -G(\pi - y, m) \\
    &= 2G(y,m) = 0.
\end{align}
The claim follows.
\end{proof}

\begin{proof}[Proof of Proposition \ref{prop:wavelikeunique}]
Let $\gamma$ and  $a<b$ be as in the statement. Up to isometries, scaling and reparametrization we may assume that $\gamma= \gamma(\cdot | m)$ for some $m \in (0,1)$. Without loss of generality we may assume $a \in [-\pi,0]$,
otherwise we use the periodicity properties \eqref{eq:antiperiod} and \eqref{eq:gammaperiod} and examine an appropriate isometric image of $\gamma$.
We need to show that $m = m_T $, $a= a_T$ and $b = b_T$. We first show $m = m_T$. Assume the opposite. Note that $m < m_T$ is impossible by Lemma \ref{lem:smallmodulus}. Hence we assume $m > m_T$. By Proposition \ref{prop:monotonequantity} we obtain (using $f$ defined there)
\begin{align}
    0 & > f(m) 
     = G\left(\pi - {\alpha(m)} , m \right). 
\end{align}
Since $G(x,m) > 0$ for small $x> 0$ one obtains that there exists $a_0 \in (0, \pi - {\alpha(m)}) $ such that $G(a_0,m) =0 $, i.e.\ $\gamma^{(1)}(a_0 | m) =0$. Using this, the evenness of $\gamma^{(2)}(\cdot | m)$ and the periodicity \eqref{eq:gammaperiod}, we obtain in particular $\gamma(-a_0 | m) = \gamma(a_0 | m)$ and $\gamma(2\pi - a_0 | m) = \gamma(2\pi + a_0 | m)$.
Combining these with $a \in [-\pi,0]$ and the embeddedness of $\gamma(\cdot | m) \vert_{[a,b)}$, we find that there are only two possible cases:
\begin{equation}\label{eq:possiab}
   \mbox{$a \in [-\pi,-a_0]$, $b\leq a_0$, \quad or  \quad $a \in (-a_0,0]$, $b \leq 2\pi + a_0$.}
\end{equation}
Now we note that $\gamma(a) = \gamma(b)$ and $T_\gamma(a) = - T_\gamma(b)$ yield a set of four equations
\begin{enumerate}
\item[(i)] $G(a,m) = G(b,m)$,
    \item[(ii)] $- 2 \sqrt{m} \cos(a) = - 2 \sqrt{m} \cos(b) $,
    \item[(iii)] $2 D(a)\sqrt{m} \sin(a)  =  -2 D(b)\sqrt{m} \sin(b) $,\\ where $D(x)= |\gamma'(x)|^{-1} = \left(\sqrt{4 m \sin^2(x) + \frac{(1-2m \sin^2(x))^2}{1-m \sin^2(x)}} \right)^{-1}$,
    \item[(iv)] $D(a) \frac{1-2m\sin^2(a)}{\sqrt{1- m \sin^2(a)}} = - D(b) \frac{1-2 m \sin^2(b)}{\sqrt{1-m \sin^2(b)}}$, where $D(a)$ and $D(b)$ are as in (iii).
\end{enumerate}
Note that equation (ii) implies $\cos(a)= \cos(b)$ and hence also $\cos^2(a) = \cos^2(b)$, whereupon also $\sin^2(a)= \sin^2(b).$ Since $D(x)$ depends only on $\sin^2(x)$ we infer $D(a) =D(b)$. With this in hand we obtain $(\cos(a),\sin(a)) = (\cos(b), - \sin(b))$ and $1- 2m \sin^2(a)= 0$.
We conclude from these equations that 
\begin{equation*}
    \mbox{$a = - b + 2\pi l$ \ for some $l \in \mathbb{Z}$,}
\end{equation*}
and 
$$\text{$a = k \pi \pm \arcsin \sqrt{\frac{1}{2m}}$ \ for some $k \in \mathbb{Z}$.}$$
Combining these with \eqref{eq:possiab}, we need to consider only $a=- \arcsin \sqrt{\frac{1}{2m}}={-\alpha(m)}\in(-\frac{\pi}{2},0)$ {(%cf.\
by \eqref{eq:def alpha})}, and for $b$ only the two possibilities $b=-a$ or $b=2\pi-a$.
The former case can be ruled out since in this case one has $1-2m \sin^2x>0$ (i.e.\ $G'(x,m)> 0 $) for all $x \in (a,b)$, a contradiction to equation (i).
The latter case can also be ruled out since it yields $a<0$ and $b>2\pi$, which contradict Lemma \ref{lem:largemodulus} and the embeddedness requirement.
We have shown that $m = m_T$. Thereupon it is straightforward with the explicit formula \eqref{eq:wavelikeparabulg} and Lemma \ref{lem:B2} (v) to prove that (up to translations and isometries) $a=a_T$ and $b=b_T$.
 \end{proof}

\subsubsection{Orbitlike ECEs}\label{sec:orbECE}
 In this section we examine orbitlike ECEs. For this purpose we choose again reparametrizations of orbitlike elasticae in the same fashion as in the previous section. More precisely we define for this section
 \begin{equation}\label{eq:explorbit}
     \gamma(x | m) := \frac{1}{m}\begin{pmatrix}  2E(x,m) + (m-2) F(x,m) \\  -2\sqrt{1-m \sin^2(x)} \end{pmatrix} \quad (x \in \mathbb{R}),
 \end{equation}
 for arbitrary $m \in (0,1)$. Again $s \mapsto \gamma(\mathrm{am}(s,m) | m)$ is a prototype of an orbitlike elastica in the sense of Proposition \ref{prop:elaclassi}.
The curve $\gamma(\cdot | m)$ is $\pi$-periodic modulo shifts, more precisely
 \begin{equation}\label{eq:orbitperiod}
     \gamma(x+ \pi  | m) = \gamma(x | m) + \frac{1}{m} \begin{pmatrix} 2 E(m) + (m-2) K(m) \\ 0   \end{pmatrix} \quad  (x \in \mathbb{R}, m \in (0,1)).
 \end{equation}
It also has a reflection symmetry around $x = \frac{\pi}{2}$, more precisely
 \begin{equation}\label{eq:orbitantiper}
     \gamma(\tfrac{\pi}{2} + x | m ) - \gamma( \tfrac{\pi}{2} | m ) = R \big( \gamma(\tfrac{\pi}{2} - x | m) - \gamma( \tfrac{\pi}{2} | m )  \big), \quad R = \begin{pmatrix} - 1 & 0 \\ 0 & 1 \end{pmatrix}. 
 \end{equation}
It is also convenient to express the first component by
\begin{align}\label{eq:orbitfirst}
    \gamma^{(1)}(x | m) = \int_0^x \frac{1- 2 \sin^2\theta}{\sqrt{1-m \sin^2\theta}} \; \mathrm{d}\theta.
\end{align}

As in the previous section we are interested in which configurations yield orbitlike ECEs. 
It will turn out that ECEs occur only for one unique modulus $m_H \in (0,1)$ that is characterized by the unique solution $m \in (0,1)$ to  
\begin{equation}\label{eq:elaheart}
 g(m) :=  \int_{-\frac{\pi}{4}}^{\frac{5\pi}{4}} \frac{1-2 \sin^2\theta}{\sqrt{1- m \sin^2\theta}} \; \mathrm{d}\theta = 2 \int_{-\frac{\pi}{4}}^{\frac{\pi}{2}} \frac{1-2 \sin^2\theta}{\sqrt{1- m \sin^2\theta}} \; \mathrm{d}\theta =  0 .
\end{equation}
Existence and uniqueness of such $m_H$ are ensured by
\begin{proposition}[Proof in Appendix \ref{app:compu}]\label{prop:gdecr} 
The function $g$ defined in \eqref{eq:elaheart} is strictly decreasing in $(0,1)$. Moreover there exists a unique $m_H \in (0,1)$ such that $g(m_H) = 0$.
\end{proposition}
The numerical value of $m_H$ is $m_H \simeq 0.8436$, cf. Table \ref{tab1}.

\begin{definition}[Heart-shaped elastica]\label{def:heartela}
Let $a_H := - \frac{\pi}{4}$ and $b_H := \frac{5\pi}{4}$. Then $\gamma_H :=\gamma(\cdot | m_H)\vert_{[a_H,b_H]} \in C_{imm}^\infty([a_H,b_H];\mathbb{R}^2)$ is called \emph{heart-shaped elastica}. We will also call rescalings, isometric images and reparametrizations of $\gamma_H$ heart-shaped elasticae, but the notation $\gamma_H$  will always fix the representative defined above. 
\end{definition}
{Note carefully that the picture of the heart-shaped elastica in Figure \ref{img:heart} is a translated, rescaled and reflected version of the explicit parametrization $\gamma_H$.}

We will show that the heart-shaped elastica is, up to invariances, the unique orbitlike ECE. 

Our observations rely on a preparatory lemma which we will use very often in the sequel.
\begin{lemma}[Proof in Appendix \ref{app:compu}]\label{lem:compulemorbit}
    For all $m \in (0,1)$ one has 
  $
        2 E(m) + (m-2) K(m) < 0.
$
\end{lemma}
\begin{proposition}[Existence of orbitlike ECEs]
Each heart-shaped elastica is an ECE.
\end{proposition}
\begin{proof}
It suffices to show that $\gamma_H$ is an ECE. 
By the representation \eqref{eq:explorbit}, and by using \eqref{eq:orbitfirst} and \eqref{eq:elaheart} for $\gamma_H^{(1)}$ and $\sin^2(\frac{5\pi}{4}) = \sin^2(- \frac{\pi}{4}) = \frac{1}{2}$ for $\gamma_H^{(2)}$, we find that $\gamma( - \frac{\pi}{4}  | m_H) = \gamma( \frac{5\pi}{4}  |  m_H).$
For the derivative we compute
$$(\gamma^{(1)})'( x|m_H) = \frac{1- 2 \sin^2(x) }{\sqrt{1- m_H \sin^2(x)}}, \quad (\gamma^{(2)})'(x | m_H) = \frac{2 \sin(x) \cos (x)}{\sqrt{1- m_H\sin^2(x)}}.$$
A direct computation yields that $(\gamma^{(1)})'( -\frac{\pi}{4}  | m_H)=(\gamma^{(1)})'( \frac{5\pi}{4} | m_H) = 0$ and also $(\gamma^{(2)})'(- \frac{\pi}{4} | m_H ) = - (\gamma^{(2)})'( \frac{5\pi}{4} | m_H )$, and hence $T_{\gamma( \cdot  | m_H)}(-\frac{\pi}{4})  = -T_{\gamma(\cdot | m_H)} ( \frac{5\pi}{4})$.

It remains to show that $\gamma(  \cdot  |  m_H) \vert_{[- \frac{\pi}{4} , \frac{5\pi}{4})}$ is injective. 
To this end assume that there exist $x_1,x_2$ such that $- \frac{\pi}{4} \leq x_1 < x_2 < \frac{5\pi}{4}$ and $\gamma(x_1 | m_H) = \gamma(x_2  | m_H).$
For the function $H(z):=\gamma^{(1)}(z | m_H) - \gamma^{(1)}(\frac{\pi}{2} | m_H)$ we notice by \eqref{eq:orbitfirst} and \eqref{eq:elaheart} that $H(-\frac{\pi}{4})= H( \frac{\pi}{2}) = 0 $. 
Moreover by \eqref{eq:orbitfirst} we have $H' > 0$ on $( -\frac{\pi}{4}, \frac{\pi}{4})$ and $H' < 0$ on $(\frac{\pi}{4},\frac{\pi}{2})$. 
This implies that $H > 0$ on $( -\frac{\pi}{4},\frac{\pi}{2})$.
Similarly $H<0$ on $(\frac{\pi}{2},\frac{5\pi}{4})$, and hence we only need to consider $x_1,x_2\in[-\frac{\pi}{4},\frac{\pi}{2}]$ or $x_1,x_2\in[\frac{\pi}{2},\frac{5\pi}{4}]$.
By reflection symmetry \eqref{eq:orbitantiper} we may assume that $x_1,x_2\in[-\frac{\pi}{4},\frac{\pi}{2}]$. 
By comparing the second components we infer that $\sin^2(x_1) = \sin^2(x_2)$ so that (by $x_1,x_2\in[-\frac{\pi}{4},\frac{\pi}{2}]$) $x_2 = -x_1$, and hence $x_1,x_2 \in [- \frac{\pi}{4},\frac{\pi}{4}].$
However observe that 
\begin{equation}
    0 = \gamma^{(1)}(x_2 | m_H ) - \gamma^{(1)} (x_1 | m_H) = \int_{x_1}^{x_2} \frac{1-2\sin^2 \theta}{\sqrt{1- m_H \sin^2 \theta}} \; \mathrm{d}\theta,
\end{equation}
which is a contradiction since $1- 2 \sin^2 \theta > 0$ for all $ \theta \in (- \frac{\pi}{4}, \frac{\pi}{4}).$
\end{proof}

\begin{proposition}[Uniqueness of orbitlike ECEs] \label{prop:uniqueorbit}
{Let $a<b$ and suppose} that $\gamma \in C^\infty_{imm}([a,b]; \mathbb{R}^2)$ is an orbitlike ECE. Then $\gamma$ is a heart-shaped elastica.
\end{proposition}

For the proof we need a preparatory lemma, similar to the wavelike case. 
\begin{lemma}\label{lem:interorbit}
Let $m \in (0,1)$ be arbitrary. Then there exist distinct points $x_1,x_2 \in (- \frac{\pi}{2}, \frac{\pi}{2})$ such that
$\gamma(x_1 | m)= \gamma(x_2 | m).$  
\end{lemma}
\begin{proof}
    Note that, by \eqref{eq:orbitfirst}, $ \gamma^{(1)}(x | m)$ is positive for small $x > 0$ %cf.\ \eqref{eq:orbitfirst}
    but for $x= \frac{\pi}{2}$ we obtain by Lemma \ref{lem:compulemorbit}, $\gamma^{(1)}(\frac{\pi}{2} | m)=\frac{1}{m} (2E(m) + (m-2) K(m) )<0$.
     Hence there exists $y \in (0, \frac{\pi}{2})$ such that $\gamma^{(1)}(y | m)=0$. We claim that $x_1 = -y$ and $x_2 = y$ yield a self-intersection. Indeed, since $\gamma^{(2)}(\cdot | m)$ is an even function we infer $\gamma^{(2)}(y | m) = \gamma^{(2)}(-y | m)$ and by the choice of $y$ and oddness of $\gamma^{(1)}$ we obtain $\gamma^{(1)}(y | m) = \gamma^{(1)}(-y | m) =0.$ 
\end{proof}

\begin{proof}[Proof of Proposition \ref{prop:uniqueorbit}]
  Let $\gamma$ and $a<b$ be as in the statement. Up to isometries, scaling and reparametrization we may assume that $\gamma = \gamma(\cdot  | m)$ for some $m \in (0,1)$.
  
  We may also assume (performing possibly another shift and using \eqref{eq:orbitantiper}) that $a \in [- \frac{\pi}{2},0 ]$. 
  We will now show that $a = - \frac{\pi}{4}$, $b = \frac{5\pi}{4}$ and $m$ satisfies \eqref{eq:elaheart}. 
  By \eqref{eq:explorbit} and \eqref{eq:orbitfirst} we deduce that the conditions $\gamma(a | m) = \gamma(b | m)$ and $T_{\gamma(\cdot | m)}(a)  = -T_{\gamma(\cdot | m)}(b)$ amount to the following set of equations
 \begin{enumerate}
 \item[(i)] $\int_a^b \frac{1- 2 \sin^2(\theta)}{\sqrt{1-m \sin^2(\theta)}} \; \mathrm{d}\theta = 0$,
     \item[(ii)] $\sqrt{1- m \sin^2(a)} = \sqrt{1-m \sin^2(b)}$,
     \item[(iii)] $D(a) \left( \frac{  2\cos(a)\sin(a)}{\sqrt{1- m \sin^2(a)}} \right) =- D(b) \left( \frac{  2 \cos(b)\sin(b)}{\sqrt{1- m \sin^2(b)}} \right)$, \\ where $D(x) = |\gamma'(x)|^{-1}= \left(  \frac{ 4 \cos^2(x) \sin^2(x) + (1-2\sin^2(x))^2 }{1-m \sin^2(x)}\right)^{-1/2}.$
     \item[(iv)] $D(a) \frac{1-2\sin^2(a)}{\sqrt{1-m \sin^2(a)}} =- D(b)  \frac{1- 2 \sin^2(b)}{\sqrt{1-m \sin^2(b)}}$, where $D(x)$ is as in (iii).
 \end{enumerate}
 Note that (ii) implies that $\sin^2(a) = \sin^2(b)$ and hence also $\cos^2(a) = \cos^2(b)$ which yields also $D(a) = D(b)$.
 From (ii), (iii) and (iv) we conclude thereupon $\sin^2(a)= \sin^2(b)$, $\cos(a) \sin(a) = - \cos(b) \sin(b)$ and $1-2 \sin^2(a) = 1-2 \sin^2(b) = 0$.
 As a consequence of these equations we obtain
     $$
     \mbox{$a,b \in \left\lbrace l \pi \pm \frac{\pi}{4}  : l \in \mathbb{Z} \right\rbrace = \left\lbrace \frac{(2l+1) \pi}{4} : l \in \mathbb{Z}\right\rbrace$}$$
and 
$$\text{$a = \pm b +k \pi$ for some $k \in \mathbb{Z}$}.$$
Since $a \in [-\frac{\pi}{2}, 0]$ the only possibility for $a$ is $a = - \frac{\pi}{4}$.
By Lemma \ref{lem:interorbit} $\gamma(x_1 | m) = \gamma(x_2  | m)$  for some $x_1 \neq  x_2 \in [-\frac{\pi}{2}, \frac{\pi}{2}]$ and by \eqref{eq:orbitperiod} we also have $\gamma(x_1 + \pi | m ) = \gamma(x_2 + \pi  | m)$. 
The fact that $\gamma(\cdot | m)\vert_{[a,b)}$ needs to be embedded and $a = - \frac{\pi}{4}$ implies hence that $b < \frac{3\pi}{2}.$ All the previous considerations leave only three cases
 \begin{align}
     % \text{Case A:} &\ a = - \tfrac{\pi}{4},\ b= \tfrac{\pi}{4},\\
     % \text{Case B:} &\ a = - \tfrac{\pi}{4},\ b= \tfrac{3\pi}{4},\\
     % \text{Case C:} &\ a = - \tfrac{\pi}{4},\ b= \tfrac{5\pi}{4}.
     %\text{Case A: $a = - \frac{\pi}{4}$, $b= \frac{\pi}{4}$, \quad Case B:  $a = - \frac{\pi}{4}$, $b = \frac{3\pi}{4}$, \quad Case C: $a = - \frac{\pi}{4}$, $b = \frac{5\pi}{4}$.}
     (a,b)=
     \begin{cases}
         (- \tfrac{\pi}{4},\tfrac{\pi}{4}) & \text{(Case A)},\\
         (- \tfrac{\pi}{4},\tfrac{3\pi}{4}) & \text{(Case B)},\\
         (- \tfrac{\pi}{4},\tfrac{5\pi}{4}) & \text{(Case C)}.
     \end{cases}
 \end{align}
Case A can be ruled out since $1- 2 \sin(\theta) > 0$ on $(- \frac{\pi}{4}, \frac{\pi}{4})$ and this is a contradiction to equation (i). 
Case B would contradict $\cos(a) \sin(a) = - \cos(b) \sin(b)$, and hence this case is ruled out by equation (iii).
The only remaining case is Case C, i.e.\ $a= - \frac{\pi}{4}$, $b = \frac{5\pi}{4}$. 
This with equation (i) and the definition of $m_H$ directly imply, using \eqref{eq:elaheart}, that $m=m_H$. %, cf. \eqref{eq:elaheart}.
 The claim is shown. 
\end{proof}

\subsection{Uniqueness results for the variational inequality}

Now that we have found all ECEs, there are only three types of candidates for solutions of \eqref{eq:pertfunc} --- and hence only three types of candidates for minimizers; compositions of two teardrop elasticae, one teardrop elastica and one heart-shaped elastica, and two heart-shaped elasticae.

From now on we use the shorthand notation $\gamma= \gamma_1 \oplus \gamma_2$ if $\gamma$ is the concatenation of two curves $\gamma_1$ and $\gamma_2$. 
In this sense we can say that each solution of \eqref{eq:pertfunc} is of the form $\gamma = [S_1(a_1 \gamma_{T/H}) \circ \Phi_1] \oplus [S_2 (a_2 \gamma_{T/H}) \circ \Phi_2]$, where $a_1,a_2 > 0$ are scaling factors, $S_1,S_2$ are Euclidean isometries and $\Phi_1, \Phi_2$ are reparametrizations. Sometimes our notation will swallow the reparametrizations $\Phi_1,\Phi_2$ -- but only if it is ensured that reparametrizations can be chosen in such a way that the curves lie in $\mathcal{A}_0$. Notice that this point is actually delicate, since passing through one of the components in a reverse direction will affect the total curvature $N[\gamma]$. Luckily $\gamma_T$ and $\gamma_H$ have a symmetry: Passing through $\gamma_T$ and $\gamma_H$ in a reverse direction is actually the same as passing through an isometric image of $\gamma_T$ or $\gamma_H$ in forward direction. 
This is easily checked since for $\gamma_H$ we already know the reflection symmetry \eqref{eq:orbitantiper} with the fact that $\frac{\pi}{2}$ is the midpoint of $a_H$ and $b_H$, while for $\gamma_T$ we infer from \eqref{eq:wavelikeparabulg} and Lemma \ref{lem:B2} the simpler symmetry
\begin{equation}\label{eq:gammaT gammH symmetry}
    \gamma_T(-x) = R\gamma_T(x),\ \mbox{where $R$ is the same matrix as in \eqref{eq:orbitantiper}}. 
\end{equation}
Hence we may actually assume that $\Phi_1$ and $\Phi_2$ are orientation-preserving --- and can safely be disregarded.

What remains unclear is whether all  configurations above actually yield solutions of the variational inequality \eqref{eq:pertfunc}. In this section we will finally show that only the \emph{elastic two-teardrop}, rigorously defined as follows, yields a solution to \eqref{eq:pertfunc}.

\begin{definition}[Elastic two-teardrop] \label{def:two_teardrop}
A curve $\gamma\in \mathcal{A}_0$ is called \emph{elastic two-teardrop} if it coincides up to scaling, isometries and reparametrization with 
% \begin{equation}\label{eq:def gamma2T}
%       \gamma_{2T}(x)  := \begin{cases}
%         \gamma_{T}(x-\pi+\alpha(m_T) ) & x \in (0, 2(\pi-\alpha(m_T))), \\  2\gamma_T(\pi-\alpha(m_T)) - \gamma_T(x+\pi-\alpha(m_T)) & x \in (-2(\pi-\alpha(m_T)), 0),
%       \end{cases}
% \end{equation}
$\gamma_{2T}$ defined by
$$\gamma_{2T}(x):=\gamma_{T}(x-\pi+\alpha(m_T) )$$ 
for $x \in (0, 2(\pi-\alpha(m_T)))$, and 
$$\gamma_{2T}(x):=2\gamma_T(\pi-\alpha(m_T)) - \gamma_T(x+\pi-\alpha(m_T))$$
for $x \in (-2(\pi-\alpha(m_T)), 0)$,
where $\alpha(m_T)$ is as in \eqref{eq:def alpha}. The notation $\gamma_{2T}$ will be used exclusively for the above parametrization.
We also define
    \begin{equation}\label{eq:defC2T}
        C_{2T}:=\bar{B}[\gamma_{2T}] \Big(=32(2m_T-1)F({\pi-\alpha(m_T)},m_T)^2\Big).
    \end{equation}
\end{definition}
We remark that this shape corresponds to $\gamma = S_1 \gamma_T \oplus S_2 \gamma_T$, for suitably chosen $S_1,S_2$. 
An important observation is that $a_1 = a_2$ needs to be ensured.

In the sequel we will rule out different combinations of $\gamma_T$ and $\gamma_H$ and different scaling factors $a_1,a_2$. 
We first rule out compositions of two heart-shaped elasticae.

\begin{lemma}\label{lem:227}
    Let $\gamma \in H^2_{imm}(\mathbb{T}^1; \mathbb{R}^2)$ be composed of two (possibly rescaled and reparametrized) isometric copies of $\gamma_H$. Then $\gamma \not \in \mathcal{A}_0$.
\end{lemma}

\begin{proof}
    We compute
    \begin{equation}
        \int_{\gamma_H} k \; \mathrm{d}s = \int_{-\frac{\pi}{4}}^{\frac{5\pi}{4}} k[\gamma_H] |\gamma_H'| \; d \theta =  \int_{-\frac{\pi}{4}}^{ \frac{5\pi}{4}} \frac{2\sqrt{1- m_H \sin^2(\theta)}}{\sqrt{1- m_H \sin^2(\theta)}} \; \mathrm{d} \theta = 3 \pi.
    \end{equation}
    In particular each concatenation $\gamma$ of two copies of $\gamma_H$ satisfies either $N[\gamma] = 0$ or $N [\gamma] =  \frac{1}{2\pi}(3\pi + 3\pi) = 3$. Hence $N[\gamma] = 1$ is impossible, implying $\gamma \not \in \mathcal{A}_0$.
\end{proof}

Another type to discuss is a combination of a teardrop elastica and a heart-shaped elastica. If $\gamma \in \mathcal{A}_0$ is such combination then --- according to  Lemma \ref{lem:optireg} --- $k[\gamma]$ is continuous. From this condition one can read off the admissible scaling factors $a_1,a_2$. 

\begin{lemma}\label{lem:rescaling}
Suppose that $\gamma \in H^2_{imm}( \mathbb{T}^1; \mathbb{R}^2)$ is of the form $\gamma = S_1(a_1 \gamma_T) \oplus S_2(a_2 \gamma_H)$ for some $a_1,a_2 > 0$ and isometries $S_1,S_2$ of $\mathbb{R}^2$. Suppose further that $k[\gamma]$ is continuous. Then $\frac{a_2}{a_1} = \frac{\sqrt{2-m_H}}{\sqrt{2m_T-1}}$. 
\end{lemma}

\begin{proof}
Let $\gamma$ be as in the statement. 
Let $[a,b]$ denote the interval on which $\gamma$ is a reparameterization of $S_1(a_1 \gamma_T)$; then $\gamma$ is a reparameterization of $S_2(a_2 \gamma_H)$ on $[b,a]$.
     % We may assume that $\gamma$ is a reparametri\-zation of $S_1(a_1 \gamma_T)$ on $[a,b]$ and a reparametrization of $S_2(a_2 \gamma_H)$ on $[b,a]$.
     {Notice that $|k[\gamma_T]|$ (resp.\ $|k[\gamma_H]|$) takes the same value at the endpoints $a_T,b_T$ (resp.\ $a_H,b_H$).}
    % This at hand 
    With this in hand we can compute $|k[\gamma](a)|$ in two ways. Firstly using \eqref{eq:def alpha}
    \begin{align}
        |k[{\gamma}](a)| & = \frac{1}{a_1} |k[\gamma_T](a_T)|  =  \frac{2}{a_1}  \sqrt{m_T} \left\vert\cos\left(\pi - \arcsin \sqrt{\frac{1}{2m_T}}\right)\right\vert \\ &= \frac{2}{a_1}  \sqrt{m_T} \sqrt{1- \frac{1}{2m_T}} =  \frac{1}{a_1} \sqrt{2} \sqrt{2m_T-1}.
    \end{align}
    Note that we have no way tell whether the isometry (or the reparametrization) connects $a$ to the left endpoint $a_T$ or the other endpoint $b_T$, but since $|k[\gamma_T](a_T)| = |k[\gamma_T](b_T)|$ this does not make a difference. In this context we also use that the isometry can only change the sign of $k[\gamma]$.
    Secondly, we obtain with the same arguments
    \begin{align}
        |k[{\gamma}](a)| & = \frac{1}{a_2} |k[\gamma_H](a_H)| = \frac{1}{a_2} 2 \sqrt{1- m_H \sin^2(-\tfrac{\pi}{4})}
        = \frac{1}{a_2} \sqrt{2} \sqrt{2-m_H}.
    \end{align}
    By the continuity assumption on $k[\gamma]$ we obtain 
    $$\frac{\sqrt{2}}{a_1}\sqrt{2m_T-1} = \frac{\sqrt{2}}{a_2}\sqrt{2-m_H},$$
    which proves the claim.
\end{proof}

Having determined the rescaling ratio we will show that this ``drop-heart''-type combination does not yield a solution of \eqref{eq:pertfunc}. We will argue that each combination with the above rescaling ratio must have more than one point of self-intersection. This will contradict Proposition \ref{prop:selfinters} (%cf.\ 
see Figure \ref{fig:fig3}).

\begin{figure}[ht]
    \centering
    \begin{subfigure}[b]
    {0.4\textwidth}
    \centering
    \includegraphics[height=45mm]{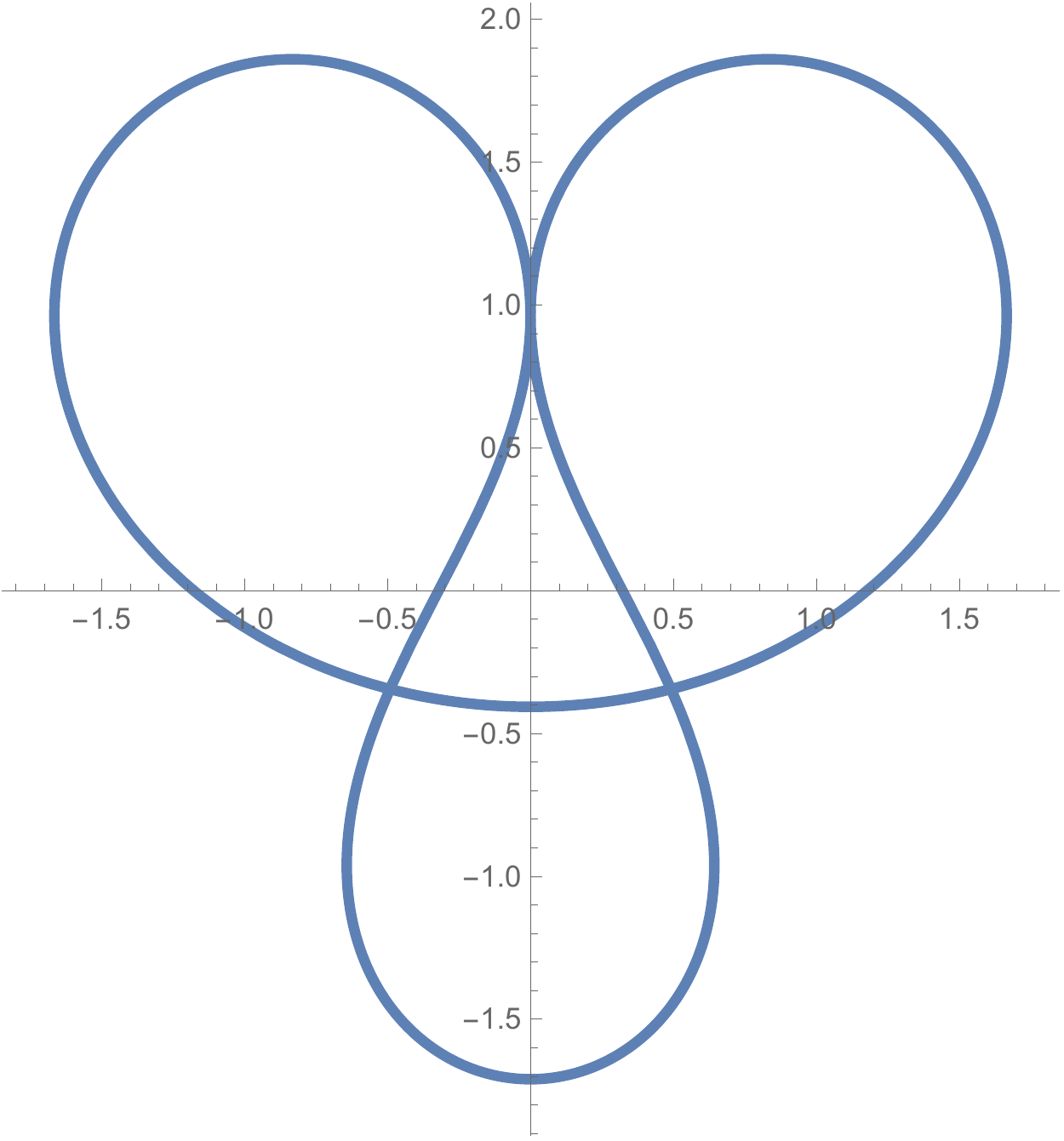}
    \end{subfigure}
    \hfill
    \begin{subfigure}[b]
    {0.55\textwidth}
    \centering
    \includegraphics[height=45mm]{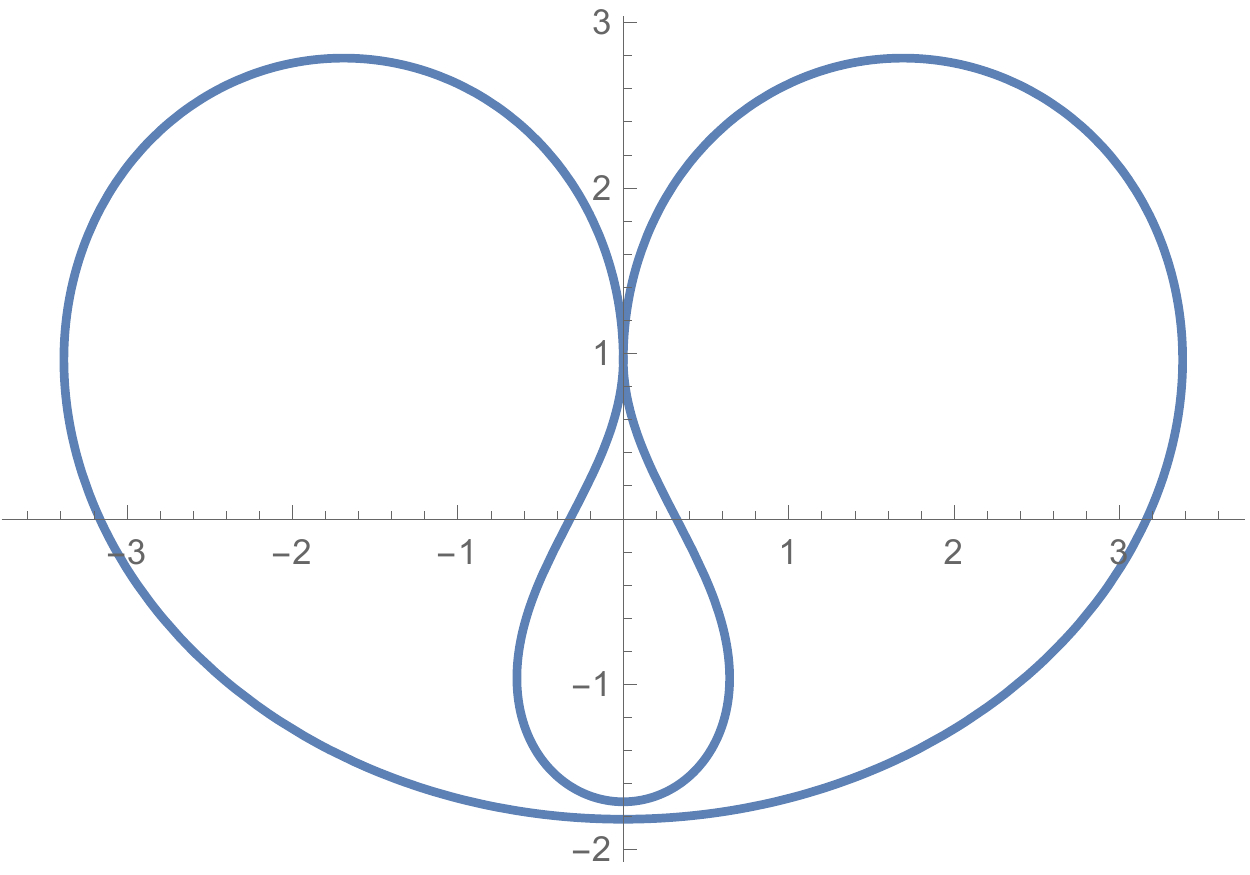}
    \end{subfigure}
    \caption{Plots of two different concatenations of $\gamma_T$ and $\gamma_H$. The rescaling ratio found in Lemma \ref{lem:rescaling} will yield (exactly)  the left hand figure --- which has more self-intersections than predicted in Proposition \ref{prop:selfinters}.}
    \label{fig:fig3}
\end{figure}

\begin{lemma}\label{lem:230}
    There exists no solution $\gamma \in \mathcal{A}_0$ of \eqref{eq:pertfunc} that is composed of one teardrop elastica and one heart-shaped elastica.
\end{lemma}

\begin{proof}
    Assume that $\gamma = S_1(a_1 \gamma_T) \oplus S_2(a_2 \gamma_H)$ solves \eqref{eq:pertfunc}, where $S_1,S_2: \mathbb{R}^2 \rightarrow \mathbb{R}^2$ are isometries and $a_1,a_2> 0$ are rescaling factors.
    The proof will be divided in two major steps.
    
    \textbf{Step 1:} We first determine the parameters we introduced more accurately. 
    Up to isometries and rescalings we may assume that $S_1 = \mathrm{id}$  and $a_1 = 1$, which implies that $a_2 = \sqrt{\frac{2-m_H}{2m_T-1}}$ by the previous lemma. 
    After those reductions we find that there exists an isometry $S: \mathbb{R}^2 \rightarrow \mathbb{R}^2$ such that 
    $\gamma =  \gamma_T \oplus S (a_2 \gamma_H)$.
  We observe that $Sx = \widetilde{S}x + v$ for some $v \in \mathbb{R}^2$ and some $\widetilde{S} \in O_2( \mathbb{R})$ satisfying $\mathrm{det}(\widetilde{S}) = - 1$, where the determinant formula holds true since  $N[\gamma] = 1 $ implies
    \begin{equation}
       \pm 2\pi = \int_\gamma k \; \mathrm{d}s = \int_{\gamma_T} k[\gamma_T] \; \mathrm{d}s + \mathrm{det}(\widetilde{S}) \int_{\gamma_H} k[\gamma_H] \; \mathrm{d}s =  (\pi + \mathrm{det}(\widetilde{S}) 3 \pi ).  
    \end{equation}
    Since $\gamma_T'(b_T)  = (0,\sqrt{2})$, $\gamma_H'(a_H) = (0,-\frac{2}{\sqrt{1- \frac{m_H}{2}}})$, and $ T_{\gamma_T}(b_T) = \widetilde{S} T_{\gamma_H}(a_H)$, and since $\gamma \in C^1(\mathbb{T}^1; \mathbb{R}^2)$, we obtain $\widetilde{S}(0,-1) = (0,1)$. Since $\widetilde{S} \in O_2(\mathbb{R})$ we obtain also $\widetilde{S}(1,0) =  \pm (1,0)$ and since $\mathrm{det}(\widetilde{S}) = -1$ we infer that $\widetilde{S} (1,0) = (1,0)$. Hence
    \begin{equation}
        \widetilde{S} = \begin{pmatrix} 1 & 0 \\ 0 & -1 \end{pmatrix}.
    \end{equation}
    We infer that
    $\gamma = \gamma_T \oplus [\widetilde{S} (a_2\gamma_H)+v] $, where $a_2$ and $\widetilde{S}$ are as determined above and $v \in \mathbb{R}^2$ is a constant translation, which is determined by $v:= \gamma_T(a_T) - \widetilde{S} (a_2 \gamma_H)(b_H)$, so that the endpoints of both curves are the same.
     
    \textbf{Step 2:} We now prove that the above $\gamma= \gamma_T \oplus S(a_2 \gamma_H)$ has a self-intersection different from $\gamma_T(a_T)$ ($=S(a_2\gamma_H)(b_H)$). This would contradict the self-intersection properties in Proposition \ref{prop:selfinters} and thus complete the proof of Lemma \ref{lem:230}.
    
    % We now prove for contradiction that the above $\gamma= \gamma_T \oplus S(a_2 \gamma_H)$ has a self-intersection different from $\gamma_T(a_T)$ ($=S(a_2\gamma_H)(b_H)$).
    
    By the explicit representation of $\gamma_T$ %cf.\
    with  \eqref{eq:wavelikeparabulg} %with 
    and representation \eqref{eq:2.7}, in particular by the second component being strictly decreasing on $(a_T,0)$, we find that $\gamma_T|_{[a_T,0]}$ can be represented by the graph $(u_T(y),y)$ of a continuous function $u_T:I_T\to\mathbb{R}$, where $I_T=[A_T,B_T]:=[\gamma_T^{(2)}(0),\gamma_T^{(2)}(a_T)]$, such that $u_T(A_H)=u_T(B_H)=0$ at the endpoints.
    By Lemma \ref{lem:B2} (v) and by the fact that $G'(x,m_T)>0$ around $x=0$ we deduce that $u_T<0$ on $(A_T,B_T)$.
On the other hand, also by looking at the explicit representation of $\gamma_H$ %cf.\
 with \eqref{eq:explorbit} and \eqref{eq:orbitfirst}, we deduce that $\gamma_H|_{[\frac{\pi}{2},\frac{3\pi}{4}]}$ can be represented by the graph $(u_H(y),y)$ of a continuous function $u_H:I_H\to\mathbb{R}$, where $I_H=[A_H,B_H]:=[\gamma_H^{(2)}(\frac{3\pi}{4}),\gamma_H^{(2)}(\frac{\pi}{2})]$, such that $u_H(A_H)<\gamma_H^{(1)}(\frac{\pi}{2})=u_H(B_H)$, where the first inequality follows by \eqref{eq:orbitantiper} and \eqref{eq:elaheart} with $x=\frac{\pi}{4}$.
Therefore, using the above expression of $S$ and the fact that  $v^{(1)}=-a_2\gamma_H^{(1)}(b_H)=-a_2\gamma_H^{(1)}(\frac{\pi}{2})$, %cf.\
following from \eqref{eq:orbitantiper} with $x=\frac{3\pi}{4}$ and \eqref{eq:elaheart}, we find that $S(a_2 \gamma_H)|_{[\frac{\pi}{2},\frac{3\pi}{4}]}$ is represented by $(\tilde{u}_H(y),y)$ with $\tilde{u}_H:\tilde{I}_H\to\mathbb{R}$ defined by $\tilde{I}_H=[\tilde{A}_H,\tilde{B}_H]:=[-a_2B_H+v^{(2)},-a_2A_H+v^{(2)}]$ and $\tilde{u}_H(y):=a_2\big(u_H(-\frac{y-v^{(2)}}{a_2})-u_H(B_H)\big)$.
In particular, $\tilde{u}_H(\tilde{A}_H)=0$ and $\tilde{u}_H(\tilde{B}_H)<0$.
Noting that by \eqref{eq:explorbit} one has $\gamma_H^{(2)}(\frac{3\pi}{4})=\gamma_H^{(2)}(\frac{5\pi}{4})=\gamma_H^{(2)}(b_H)$ %cf.\ 
 and recalling that $v$ is chosen so that $v^{(2)}=\gamma_T^{(2)}(a_T)-(\widetilde{S}(a_2\gamma_H))^{(2)}(b_H)=\gamma_T^{(2)}(a_T)+a_2\gamma_H^{(2)}(b_H)$, we deduce that $B_T=\tilde{B}_H$.

Now for the desired self-intersection property, in view of the intermediate value theorem for $u_T-\tilde{u}_H$, it is sufficient to prove that $\tilde{A}_H>A_T$, namely
\begin{equation}\label{eq:IVTgoal}
    -a_2\gamma_H^{(2)}(\tfrac{\pi}{2})+\gamma_T^{(2)}(a_T)+a_2\gamma_H^{(2)}(b_H) > \gamma_T^{(2)}(0).
\end{equation}
By direct computations using {\eqref{eq:def alpha} and} \eqref{eq:wavelikeparabulg} we have $\gamma_T^{(2)}(0)=-2\sqrt{m_T}$ and $\gamma_T^{(2)}(a_T) = -2 \sqrt{m_T} \cos \left( \pi - \arcsin \sqrt{ \frac{1}{2m_T}} \right) = \sqrt{2} \sqrt{2m_T -1}$, and by using \eqref{eq:explorbit} we also have $\gamma_H^{(2)}( b_H ) = \gamma_H^{(2)}(\frac{5\pi}{4}) = -\frac{2}{m_H} \sqrt{1- \frac{m_H}{2}}$ and $\gamma_H^{(2)}(\frac{\pi}{2}) =  -\frac{2}{m_H}\sqrt{1-m_H}$.
Therefore, also by using $a_2=\sqrt{\frac{2-m_H}{2m_T-1}}$, we find that \eqref{eq:IVTgoal} is equivalent to
\begin{align}
    Y &:= \sqrt{2}\sqrt{2m_T-1} + 2 \sqrt{\frac{2-m_H}{2m_T-1}} \left( \frac{1}{m_H} \sqrt{1- m_H} - \frac{1}{m_H} \sqrt{1- \frac{m_H}{2}} \right) \\
    &> -2\sqrt{m_T}.
\end{align}
This follows by 
     \begin{align*}
     Y &=\sqrt{2}\sqrt{2m_T-1} -  \sqrt{\frac{2-m_H}{2m_T-1}} \frac{1}{\sqrt{1-m_H}+ \sqrt{1- \frac{m_H}{2}}}
     \\ &\geq \sqrt{2}\sqrt{2m_T- 1} -  \sqrt{\frac{2-m_H}{2m_T-1}} \frac{1}{\sqrt{1-\frac{m_H}{2}}} = \sqrt{2} \frac{2m_T-2}{\sqrt{2m_T-1}} > - 2 \sqrt{m_T},
     \end{align*}
     where the last inequality follows by elementary computations with the analytic estimate $m_T>\frac{2}{3}$ independently proved in Lemma \ref{lem:m_T}.
     Hence we obtain the desired contradiction to the self-intersection properties in Proposition \ref{prop:selfinters}.
\end{proof}

Finally, we examine combinations of two teardrop elasticae. The remaining task here is to determine all the scalings and isometries that may yield solutions of \eqref{eq:pertfunc}.
Since existence of minimizers is already ensured by Proposition \ref{prop:existence}, we know that there must be at least one configuration that yields a solution. 

\begin{lemma}\label{lem:228}
    Suppose that $\gamma \in \mathcal{A}_0$ is a solution to \eqref{eq:pertfunc} composed of two teardrop elasticae. Then $\gamma = \gamma_{2T}$ (up to rescaling, reparametrization and isometries).  
\end{lemma}

\begin{proof}
Suppose that $\gamma = [S_1 (a_1 \gamma_T)] \oplus [S_2 (a_2 \gamma_T)]$ solves \eqref{eq:pertfunc}. Up to isometries and rescaling we may assume that $S_1 = \mathrm{id}$ and $a_1 = 1$.
We need to show  that also $a_2 = 1$ and  $S_2 = v - \mathrm{id}$ for some translation vector $v \in \mathbb{R}^2$ (which is uniquely determined by the condition $\gamma_T(a_T) = S_2(a_2\gamma_T(b_T))$).
Let $a,b \in \mathbb{T}^1$ be such that $\gamma(a) = \gamma(b)$.
First notice that $S_2x = \widetilde{S}x + v$ for some $\widetilde{S} \in O_2(\mathbb{R})$ and $v \in \mathbb{R}^2$.
Comparing tangent vectors at the endpoints as in Step 1 of the proof of Lemma \ref{lem:230}, we deduce that 
\begin{equation}\label{eq:matrxA}
    \widetilde{S} = \begin{pmatrix} \pm 1 & 0 \\ 0 & - 1 \end{pmatrix}.
\end{equation}
To determine the sign of the first entry we observe by Lemma \ref{lem:optireg}
\begin{equation}\label{eq:compaarecurv}
    k[\gamma_T](b_T) = k[\gamma](b) = k[S_2(a_2\gamma_T)](a_T) = \frac{\mathrm{det}(\widetilde{S})}{a_2} k [\gamma_T](a_T). 
\end{equation}
An easy computation using Propositions \ref{prop:elaclassi} and \ref{prop:monotonequantity} reveals that 
$$k[\gamma_T](a_T) = k[\gamma_T](b_T)=-  \sqrt{2}\sqrt{2m_T-1}\neq0,$$
whereupon \eqref{eq:compaarecurv} yields $\frac{\mathrm{\det}(\widetilde{S})}{a_2} = 1$. As $a_2>0$ and $|\mathrm{det}(\widetilde{S})|= 1$ we obtain $a_2=1$ and $\mathrm{det}(\widetilde{S})= 1$, so that $\widetilde{S}=-\mathrm{id}$, cf.\ \eqref{eq:matrxA}.
In particular, also $S_2 = v - \mathrm{id}$ and it follows that $\gamma= \gamma_{2T}$. Now one would actually have to compute that $\gamma_{2T} \in \mathcal{A}_0$ (e.g.\ $N[\gamma_{2T}] = 1$ and $\gamma_{2T}$ solves \eqref{eq:pertfunc}). This however is not needed since existence of a solution to \eqref{eq:pertfunc} is already ensured by Proposition \ref{prop:existence} and $\gamma_{2T}$ is now (up to invariances) the only candidate.
\end{proof}

 \begin{proof}[Proof of Theorem \ref{thm:1.4}]
 We have shown in Proposition \ref{prop:existence} that a minimizer $\gamma_0 \in \mathcal{A}_0$ exists.  
 We have then formulated the variational inequality \eqref{eq:pertfunc} as a necessary criterion for a minimizer. 
From  Proposition \ref{prop:selfinters} we conclude that each solution of the variational inequality must be composed of exactly two ECEs (cf.\ Definition \ref{def:selftouchdrop}), all of which we have classified in Section \ref{sec:ECE}.
By Lemma \ref{lem:227}, Lemma \ref{lem:230}, and Lemma \ref{lem:228} only a two-teardrop can yield a solution of \eqref{eq:pertfunc}. Since existence is already ensured, we obtain that each two-teardrop must be a minimizer. The claim follows by definition of $C_{2T}$.
\end{proof}

We finally give a remark on the classification of solutions to \eqref{eq:pertfunc} and their stability.

\begin{remark}\label{rem:classification}
If a self-intersecting curve $\gamma\in H^2_{imm}(\mathbb{T}^1;\mathbb{R}^2)$ has $N[\gamma]\neq1$ and solves \eqref{eq:pertfunc}, then $\gamma$ must be an elastica.
Indeed, if $N\neq1$, then any local perturbation keeps the value of $N$ and thus retains a self-intersection by Hopf's Umlaufsatz (cf.\ \cite[Lemma A.5]{LiYau1}), so that any solution to \eqref{eq:pertfunc} must be globally an elastica.
The known classification of closed planar elasticae (%cf.
see e.g. \cite[Theorem 0.1 and Corollary p. 87]{LangerSingerMinmax}) implies that for $N=0$ any solution to \eqref{eq:pertfunc} must be a figure-eight elastica (stable) or its multiple covering (unstable), and for $N=\nu\geq2$ a $\nu$-fold circle (stable), where the (in)stability means that the curve is a local minimizer (or not) in the $H^2$-topology.
Therefore, by Theorem \ref{thm:1.4}, we completely classify all possible solutions to the variational inequality \eqref{eq:pertfunc} and their stability among self-intersecting planar closed curves.
\end{remark}

\section{Consequences for the elastic flows}\label{sect:elasticflows}

In this section, we will prove that the energy threshold for preservation of embeddedness in Theorem \ref{thm:main} is sharp, i.e.\ for any larger energy threshold we will construct an initially embedded curve which develops self-intersections in finite time.

Our main ingredient is the smooth dependence of the elastic flow on the initial datum.

\subsection{Well-posedness of the flows}

    We have the following well-posedness result for the elastic flow of smooth curves, see also \cite[Theorem 2.1]{Blatt} for a general result in codimension one.
    
    \begin{theorem}\label{thm:smooth dependence flow}
    Let $C^\infty_ {imm}(\mathbb{T}^1;\mathbb{R}^n)$ denote the space of smoothly immersed curves and let $n\geq 2$. Then for each $\gamma_0\in C^\infty_ {imm}(\mathbb{T}^1;\mathbb{R}^n)$ there exists a unique solution $\gamma \in C^\infty(\mathbb{T}^1\times[0,\infty);\mathbb{R}^n)$ of the elastic flow \eqref{eq:elasticflow} with either $\lambda>0$ fixed or $\lambda$ given by \eqref{eq:def lambda}. Moreover, the map $C^\infty_ {imm}(\mathbb{T}^1;\mathbb{R}^n)\times[0,\infty)\mapsto C^\infty_ {imm}(\mathbb{T}^1;\mathbb{R}^n), (\gamma_0, t)\mapsto \gamma(\cdot,t)$ is smooth.
    \end{theorem}

We will not prove Theorem \ref{thm:smooth dependence flow} here, but we remark that a way to obtain the relevant well-posedness for small times is already roughly sketched in \cite{DKS}, where also long-time existence is proven. The idea is to prescribe an explicit tangential motion for the flow which transforms the initial value problem of the elastic flow into a quasilinear parabolic system. That system can then be solved by standard methods, after observing that the Lagrange multiplier (in the length-preserving case) is only of third order after integration by parts, see \cite{EscherSimonett} for a related result. Moreover, for general geometric flows, a local well-posedness result has been proven in \cite{HuiskenPolden} and \cite{MantegazzaMartinazzi}. However, these results do not cover the case of general Lagrange multipliers or of codimension larger than one. %We plan to address this problem in a future work.

\begin{remark}\label{rem:3.2}
In the case of non-smooth initial data, it is still possible to find (unique) solutions to suitable weak formulations of the elastic flow, cf. \cite{OkabePozziWheeler,OkabeWheeler,Length Preserving,BVH21}. As long as these flows possess spatial $H^2$-regularity at any time and decrease the bending energy $\bar{B}$ (respectively $E_\lambda$), we may apply Theorems \ref{thm:figure-eight} and \ref{thm:1.4} in order to conclude embeddedness.
\end{remark}

\subsection{Optimality of the threshold in codimension one}
We follow the ideas in \cite{Blatt} and construct a family of embeddings converging to %a non-embedded
an immersion with a tangential self-intersection. At this self-intersection, our example will have velocities pointing towards each other, which makes the self-intersection attractive for the flow. This is achieved by stacking the graph of
\begin{align}\label{eq:def u alpha}
    u_\alpha(x):= x^4+\alpha
\end{align}
on top of $\operatorname{graph}(-u_\alpha)$ for $\alpha>0$. For both $n=2$ and $n\geq 3$, we will perturb a suitable minimal shape, see Figures \ref{fig:perturbation_two_drop} and \ref{fig:perturbation_figure_eight} below for an illustration of the idea.  

In the case of codimension one, we will perturb an elastic two-teardrop $\gamma_{2T}$ %cf.\
(as in Definition \ref{def:two_teardrop}). 
The reason why we cannot directly work with $\gamma_{2T}$ is that it will immediately become embedded under an elastic flow --- in fact this follows from Theorem \ref{thm:1.4}, Remark \ref{rem:3.2}, the energy decay and the classification of closed elasticae. Geometrically, this means that the elastic flow pulls the self-intersection of $\gamma_{2T}$ apart. In contrast to that, the two arcs of the self-intersection of the perturbed curve $\eta_0$ in Lemma \ref{lem:existence counterexample n=2} below will be pulled towards each other.
By Definition \ref{def:teardropelastica}  and \eqref{eq:gammaT gammH symmetry}, after reparametrization and rotation we may assume that the elastic two-teardrop is given by $\gamma_{2T}^*\colon\mathbb{T}^1\to\mathbb{R}^2$ with $\gamma_{2T}^*(0)=\gamma_{2T}^*(\frac{1}{2})=0$, $T_{\gamma_{2T}^*}(0)=-T_{\gamma_{2T}^*}(\frac{1}{2})=e_1$ and satisfies the symmetry property
\begin{align}\label{eq:twodrop symmetry}
    \gamma_{2T}^*(x)=R\gamma_{2T}^*\Big(\frac{1}{2}-x\Big), \quad \text{for all }x\in \mathbb{T}^1,
\end{align}
where $R\in O_2(\mathbb{R})$ is the reflection across the $e_1$-axis, i.e.\ $R(u,v)=(-u,v)$ for $(u,v)\in \mathbb{R}^2$.

Moreover, for any curve $\gamma\colon \mathbb{T}^1\to\mathbb{R}^n$ we define the \emph{velocity field} for the elastic flow by
\begin{equation}
    V[\gamma] := -2 \nabla_s^2\kappa - |\kappa|^2\kappa+ \lambda \kappa,
\end{equation}
where either $\lambda>0$ is a fixed number or $\lambda = \lambda[\gamma]$ is given by \eqref{eq:def lambda}. With this notation, the elastic flow equation \eqref{eq:elasticflow} can be written as $\partial_t \gamma(\cdot,t) = V[\gamma(\cdot,t)]$ for all $t>0$.

\begin{lemma}\label{lem:existence counterexample n=2}
    Let $\varepsilon>0$. There exists a family of smooth curves $(\eta_\alpha)_{\alpha\in [0,1]} \subset  C^\infty_{imm}(\mathbb{T}^1;\mathbb{R}^2)$ such that
    \begin{enumerate}[label={\upshape(\roman*)}]
        \item $\bar B[\eta_\alpha] \leq C_{2T}+\varepsilon$ for all $\alpha\in [0,1]$;
        \item $\eta_\alpha$ is an embedding for all $\alpha\in (0,1]$;
        \item $\eta_\alpha\to \eta_0$ smoothly as $\alpha \searrow 0$;
        \item there exists $\rho>0$ such that we have $\eta_\alpha(x) = (x, \rho^2u_\alpha(x))$ and $\eta_\alpha(\frac{1}{2}-x)=(x, -\rho^2u_\alpha(x))$ for all $x\in [-\frac{\rho}{2},\frac{\rho}{2}]$.
        In particular, $\eta_0(0)=\eta_0(\frac{1}{2})=0$, $\eta_0^{(2)}(\pm \frac{\rho}{2}) >0$ and $ \eta^{(2)}_0(\frac{1}{2}\pm  \frac{\rho}{2}) <0$;
        \item $ V[\eta_0]^{(2)}(0)<0$ and $ V[\eta_0]^{(2)}(\frac{1}{2}) >0$;
        \item We have $\eta_\alpha(x) = R \eta_\alpha(\frac{1}{2}-x)$.
    \end{enumerate}
\end{lemma}
The shape of the curves $(\eta_\alpha)_{\alpha\in[0,1]}$ is illustrated in Figure \ref{fig:perturbation_two_drop}.
\begin{figure}
    \centering
    \includegraphics[width=0.7\textwidth]{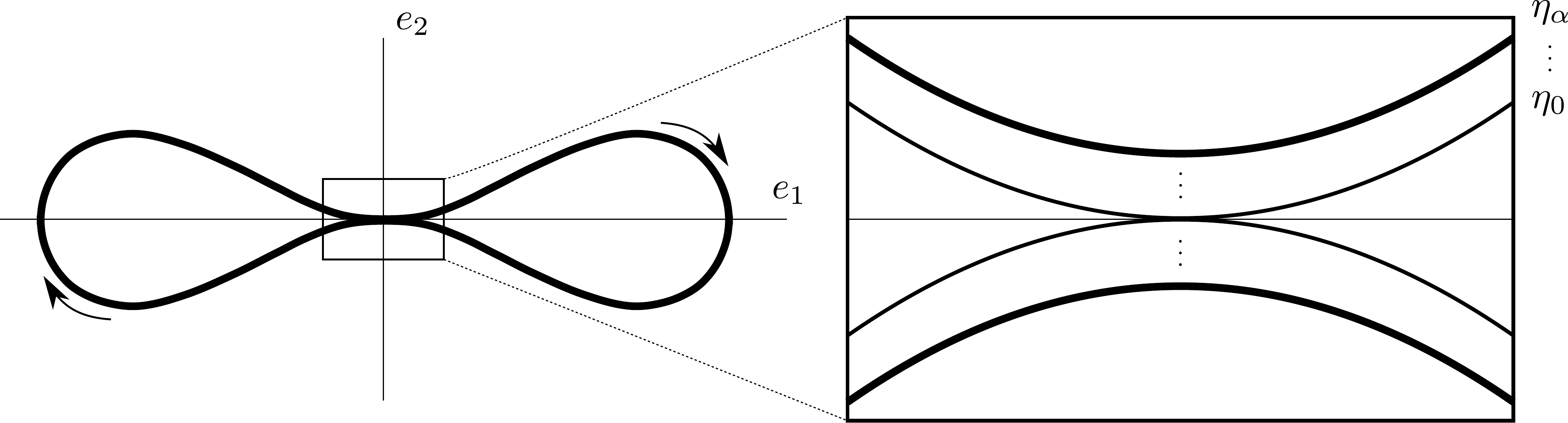}
    \caption{Perturbation of the elastic two-teardrop, cf.\ Lemma \ref{lem:existence counterexample n=2}.}
    \label{fig:perturbation_two_drop}
\end{figure}
\begin{proof}[Proof of Lemma \ref{lem:existence counterexample n=2}]
   Let $\varepsilon>0$ and let $\gamma_{2T}^*\in  H^{2}_{imm}(\mathbb{T}^1;\mathbb{R}^2)$ be as above.
After another appropriate reparametrization, we may assume that around the self-intersection point at $x=0$, the curve $\gamma_{2T}^*$ is locally given as the graph of a function $v\colon(-\rho_0, \rho_0)\to \mathbb{R}$ for $\rho_0>0$, which is smooth, except at the origin. Moreover, $v\in C^2((-\rho_0, \rho_0);\mathbb{R}^2)$ by Lemma \ref{lem:optireg} and satisfies  $v(0)=0$, $v'(0)=0$ and $v(x)=v(-x)$ as well as 
   \begin{align}\label{eq:v bounds}
       0<v(x)\leq Cx^2 \text{ and } |v'(x)|\leq C|x| \quad \text{for all } x\neq 0.
   \end{align}
   Let $\psi\in C^\infty_c(\mathbb{R})$ be a cut-off function with $\psi(x)=\psi(-x)$ for all $x\in \mathbb{R}$ and $\psi \equiv 1$ on $[-\frac{1}{2}, \frac{1}{2}]$ and $\operatorname{supp}\psi\subset(-1,1)$. 
   For $0\leq\rho<\rho_0$, we now replace $v$ by the smooth function $w_\alpha$, given by
    \begin{align}\label{eq:def u graph}
       w_\alpha(x) := \left(1-\psi\left(\frac{x}{\rho}\right)\right)v(x)+ \rho^2 \psi\left(\frac{x}{\rho}\right)u_\alpha(x),
   \end{align}
    for $\alpha\in [0,1]$.
   Clearly, we have $w_\alpha(x)=v(x)$ for $x\in (-\rho_0, -\rho]\cup [\rho, \rho_0)$ whereas $w_\alpha(x) = \rho^2u_\alpha(x)$ for $x\in [-\frac{\rho}{2}, \frac{\rho}{2}]$. Moreover, for all $\alpha\in [0,1]$, we have by \eqref{eq:v bounds} and direct estimates
   \begin{align}\label{eq:C^2 control perturbation}
       \norm{v''-w_\alpha''}{L^2(-\rho_0, \rho_0)}^2 &\leq C\norm{\psi''}{\infty}^2  \rho^{-4}\int_{-\rho}^{\rho}\left( |v(x)|^2+\rho^4|u_\alpha(x)|^2\right)\mathrm{d}x \nonumber\\
       & \quad + C\norm{\psi'}{\infty}^2 \rho^{-2}\int_{-\rho}^{\rho}\left(|v'(x)|^2+ \rho^4 |u_\alpha'|^2 \right) \mathrm{d}x \nonumber\\
       &\quad + C \norm{\psi}{\infty}\int_{-\rho}^{\rho} \left(|v''(x)|^2+\rho^4 |u_\alpha''(x)|^2\right)\mathrm{d}x\nonumber
       \leq C \rho.
   \end{align}
   Similarly, one obtains $\norm{v-w_\alpha}{H^2(-\delta_0, \delta_0)}^2\leq C\rho$.
   Around the self-intersection at $x=\frac{1}{2}$ we proceed similarly by symmetry. This way, we have constructed a smooth curve $\Tilde{\gamma} = \Tilde{\gamma}(\rho, \alpha)\colon\mathbb{T}^1\to\mathbb{R}^2$. Now we get by continuity of the normalized bending energy, if we choose $\rho>0$ small enough, that $\bar{B}[\Tilde{\gamma}]\leq C_{2T}+\varepsilon$ for all $\alpha\in [0,1]$. Fix any such $\rho>0$ and define $\eta_\alpha := \Tilde{\gamma}(\rho, \alpha)$ for $\alpha\in [0,1]$. Then (i), (ii) and (iii) are satisfied. 
   
   Property (iv) follows directly from the construction.
   
   For (v), we note that at $x=0$ we have $\partial_x^k w_\alpha(0) = 0$ for $k=1,2,3$. 
   Hence, by the explicit representation of the elastic flow \eqref{eq:elasticflow} in coordinates (see for instance \cite[(A.4)]{DLPSTE}), we have at $x=0$
   \begin{align}
   \begin{split}
       V[\eta_0](0) &= \left.-2 \nabla_s^2 \kappa - |\kappa|^2\kappa + \lambda\kappa \right\vert_{x=0} \\
       &= \left.-2\left( \frac{\partial_x^4 \eta_0}{|\partial_x \eta_0|^4}\right)^{\perp_{\eta_0}}\right\vert_{x=0} = -48 \rho^2 (0,1),
   \end{split}
   \end{align}
   such that $ V[\eta_0]^{(2)}(0)<0$, where we used that $\partial_x \eta_0(0) = (1, \rho^2\partial_x u_0(0)) = (1,0)$.
   The statement at $x=\frac{1}{2}$ follows similarly.
   
   Property (vi) follows by \eqref{eq:twodrop symmetry} and the symmetry of our construction.
\end{proof}

We will now conclude that the flow of $\eta_\alpha$ develops self-intersections in finite time, if $\alpha>0$ is small enough. 

\begin{proposition}\label{prop:loss of emb n=2} Let $\varepsilon>0$ and let $(\eta_\alpha)_{\alpha\in [0,1]}$ be as in Lemma \ref{lem:existence counterexample n=2}. Then, for $\alpha>0$ small enough, the elastic flow with initial datum $\eta_\alpha$ develops at least two self-intersections in finite time.
\end{proposition}

\begin{proof}
    Let $\eta_\alpha$ and $\rho>0$ be as in Lemma \ref{lem:existence counterexample n=2} for $\alpha\in[0,1]$ and denote by $\Gamma_\alpha\colon \mathbb{T}^1\times [0,\infty)\to\mathbb{R}^2$ the elastic flow with initial datum $\eta_\alpha$. First, by Lemma \ref{lem:existence counterexample n=2} (iv) and by continuity of the flow $\Gamma_0$, we find for all $t>0$ small enough
    \begin{align}
        \Gamma^{(2)}_0(\pm  \frac{\rho}{4},t) >0,
    \end{align}
    and using the flow equation \eqref{eq:elasticflow} and Lemma \ref{lem:existence counterexample n=2} (v), we can also assume
    \begin{align}
        \Gamma_0^{(2)}(0,t)<0.
    \end{align}
    Using 
    Lemma \ref{lem:existence counterexample n=2} (iii) and Theorem \ref{thm:smooth dependence flow}, we find  for $t>0$ and $\alpha>0$ small enough
    \begin{align}\label{eq:3.3}
        \Gamma_\alpha^{(2)}(\pm \frac{\rho}{4},t)>0 \text{ and }\Gamma_\alpha^{(2)}(0,t)<0.
    \end{align}
    It is a straightforward computation that if $R\in O_2(\mathbb{R})$ denotes the reflection over the $e_1$-axis %cf.\ using Lemma \ref{lem:existence counterexample n=2} (vi)
     the family of curves $(x,t)\mapsto R \Gamma_\alpha(\frac{1}{2}-x,t)$ is an elastic flow with initial datum $R\eta_\alpha(\frac{1}{2}-\cdot)$. By Lemma \ref{lem:existence counterexample n=2} (vi) and the uniqueness of the elastic flow (see Theorem \ref{thm:smooth dependence flow}), we thus find $\Gamma_\alpha(x,t) = R\Gamma_\alpha(\frac{1}{2}-x,t)$ for all $x\in \mathbb{T}^1$ and $t>0$. However, by \eqref{eq:3.3} and the classical intermediate value theorem, we find the existence of $x_1 \in (-\frac{\rho}{4},0)$ and $x_2\in (0, \frac{\rho}{4})$ such that $\Gamma^{(2)}_\alpha(x_j,t)=0$ for $j=1,2$. For any $j\in \{1,2\}$, the symmetry then yields $\Gamma_\alpha^{(1)}(x_j,t)=\Gamma_\alpha^{(1)}(\frac{1}{2}-x_j,t)$ and $\Gamma_\alpha^{(2)}(x_j,t)=-\Gamma_\alpha^{(2)}(\frac{1}{2}-x_j,t)=0$. Consequently, $\Gamma_\alpha(\cdot,t)$ possesses at least two self-intersections. 
\end{proof}

\subsection{Optimality in \texorpdfstring{$\mathbb{R}^3$}{R^3}}

We now wish to prove the optimality of the energy threshold also for spatial curves. As in the two-dimensional case, this will be a consequence of a continuity argument for a small perturbation of a minimal curve, which in this case is the (planar) figure-eight elastica in $\mathbb{R}^3$.

Let $\gamma_8^*\in C^\infty(\mathbb{T}^1;\mathbb{R}^2)$ be a parametrization of the figure-eight elastica $\gamma_8$ %cf.\
(see Definition \ref{def:def_fig_eight}) with self-intersection at $\gamma_8^*(0)=\gamma_8^*(\frac{1}{2})=0$. Identifying $\mathbb{R}^2 = \mathbb{R}^2\times \{0\}\subset \mathbb{R}^3$, we can view $\gamma_8^*$ as a space curve. Let $T_1, T_2\in \mathbb{S}^2\cap \mathbb{R}^2$ denote the tangent vectors at the self-intersections, i.e.\ $T_1:=T_{\gamma_8^*}(0)$, $T_2:=T_{\gamma_8^*}(\frac{1}{2})$, see Figure \ref{fig:perturbation_figure_eight} below. With $e_3 := (0,0,1)\in \mathbb{R}^3$, we have that $(T_1,T_2,e_3)$ is a (non-orthogonal) basis for $\mathbb{R}^3$ by \cite[Lemma 5.6]{LiYau1}. For the rest of this subsection, we will express vectors in $\mathbb{R}^3$ with respect to this coordinate system, i.e.\ $(v^{(1)},v^{(2)},v^{(3)}) = v^{(1)}T_1+v^{(2)}T_2+v^{(3)}e_3$ for $v^{(1)},v^{(2)},v^{(3)}\in \mathbb{R}$.

\begin{lemma}\label{lem:existence counterexample n=3}
    Let $\varepsilon>0$. There exists a family of smooth curves $(\eta_\alpha)_{\alpha\in [0,1]}\subset C^\infty_{imm}(\mathbb{T}^1;\mathbb{R}^3)$ such that
    \begin{enumerate}[label={\upshape(\roman*)}]
        \item $\bar B[\eta_\alpha]\leq C_8+\varepsilon$ for all $\alpha\in [0,1]$;
        \item $\eta_\alpha$ is an embedding for all $\alpha\in (0,1]$;
        \item $\eta_\alpha\to\eta_0$ smoothly as $\alpha\searrow 0$;
        \item there exists $\rho>0$ such that $\eta_\alpha(x) = (x, 0, \rho^2u_\alpha(x))$ and $\eta_\alpha(x+\frac{1}{2}) = (0,x, -\rho^2u_\alpha(x))$ for $x\in [-\frac{\rho}{2}, \frac{\rho}{2}]$. In particular $\eta_0(0) = \eta_0(\frac{1}{2}) = 0$;
        \item we have $V[\eta_0]^{(3)}(0)<0$ and $V[\eta_0]^{(3)}(\frac{1}{2}) >0$.
    \end{enumerate}
\end{lemma}
A sketch of our construction can be found in Figure \ref{fig:perturbation_figure_eight} below.
\begin{proof}[Proof of Lemma \ref{lem:existence counterexample n=3}]
    Let $\varepsilon>0$. In a neighborhood of $x=0$, we can assume that $\gamma_8^*$ is given as the graph of a function $v^{(2)}\colon(-\rho_0, \rho_0)\to\mathbb{R}$ over the $T_1$-axis, i.e.\ $\gamma_8^*(x)=(x, v^{(2)}(x), 0)$ for all $x\in (-\rho_0, \rho_0)$. The choice of our coordinate system implies $|v^{(2)}(x)|\leq Cx^2$ and $|(v^{(2)})'(x)|\leq C|x|$ for all $x\in [-\rho_0,\rho_0]$. With $\psi$ as in Lemma \ref{lem:existence counterexample n=2}, we define smooth functions 
    $w_\alpha^{(2)}, w_\alpha^{(3)}\colon(-\rho_0, \rho_0)\to\mathbb{R}^2$ by
    \begin{align}
        w_\alpha^{(2)}(x) &:= \left(1-\psi\left(\frac{x}{\rho}\right)\right)v^{(2)}(x),\\
        w_\alpha^{(3)}(x) &:= \rho^2\psi\left(\frac{x}{\rho}\right) u_\alpha(x),
    \end{align}
    where $u_\alpha$ is as in \eqref{eq:def u alpha}. Hence, the function
    \begin{align}
        \begin{cases}
          \gamma_8^*(x) & x\in (-\rho_0, -{\rho}]\cup[{\rho}, \rho_0),\\
          (x, w_\alpha^{(2)}(x), w_\alpha^{(3)}(x)) & x\in (-\rho, \rho),
        \end{cases}
    \end{align}
    is smooth. Around $x=\frac{1}{2}$, we can perform a similar perturbation, writing $\gamma_8^*$ locally as a graph over the $T_2$-axis and using $-u_\alpha$ instead of $u_\alpha$. This yields a closed curve $\eta_\alpha$ for all $\alpha\in[0,1]$. Estimating the $H^2$-norm as in \eqref{eq:C^2 control perturbation} and choosing  $\rho>0$  small enough, we find $\bar{B}[\eta_\alpha] \leq C_8+\varepsilon$ by continuity . 
    
    As in  Lemma \ref{lem:existence counterexample n=2}, the remaining statements (ii)-(v) can directly be deduced from the construction.
\end{proof}

\begin{figure}
    \centering
    \includegraphics[width=0.7\textwidth]{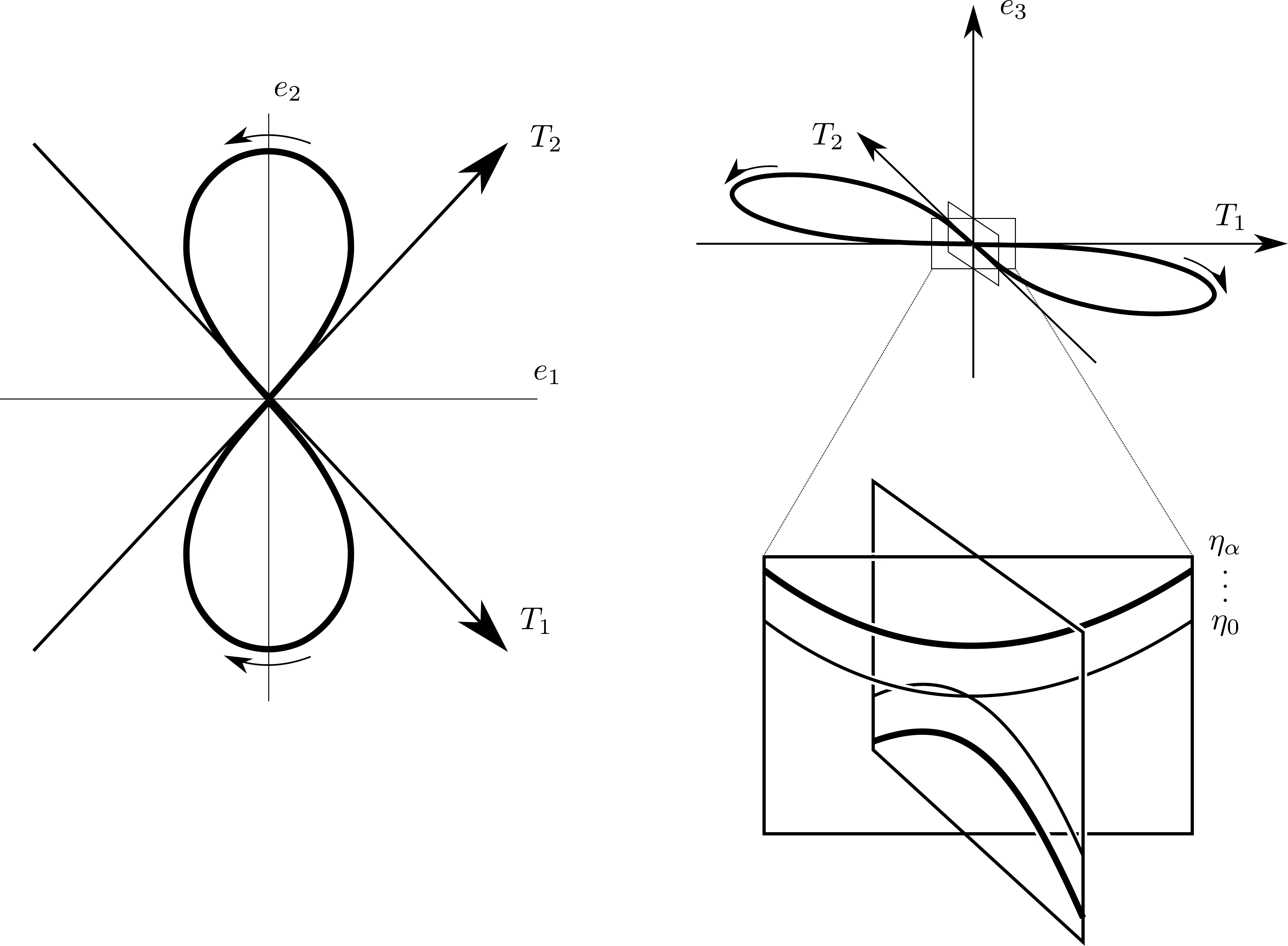}
    \caption{Out-of-plane perturbation of the figure-eight elastica, %cf.\
    see Lemma \ref{lem:existence counterexample n=3}.}
    \label{fig:perturbation_figure_eight}
\end{figure}

This is again enough to %guarantee
ensure that the curves $\eta_\alpha$ become non-embedded in finite time under the elastic flow.

\begin{proposition}\label{prop:loss of emb n=3}
    Let $\varepsilon>0$ and let $(\eta_\alpha)_{\alpha\in [0,1]}$ be as in Lemma \ref{lem:existence counterexample n=3}. Then for $\alpha>0$ small enough, the elastic flow with initial datum $\eta_\alpha$ develops a self-intersection in finite time.
\end{proposition}

\begin{proof}
    Let $(\eta_\alpha)_{\alpha\in [0,1]}$ and $\rho>0$ be as in Lemma \ref{lem:existence counterexample n=3} and denote by $\Gamma_\alpha\colon\mathbb{T}^1\times [0, \infty)\to\mathbb{R}^3$ the elastic flow with initial datum $\eta_\alpha$.

    Using Lemma \ref{lem:existence counterexample n=3} (v) and the smoothness of $\Gamma_0$, for some $c= c(\eta_0)>0$, $\delta=\delta(\eta_0)\in (0,\frac{\rho}{2})$ with $\delta<\frac{1}{4}$ and $\tau=\tau(\eta_0)>0$ we have
    \begin{align}\label{eq:3.5 1}
        \partial_t \Gamma^{(3)}_0(x,t)\leq -c,\quad \partial_t \Gamma^{(3)}_0(x+\frac{1}{2},t)\geq c \quad \forall x\in [-\delta, \delta], t\in [0,\tau].
    \end{align}
   Since the map $\Gamma_0$ is smooth by Theorem \ref{thm:smooth dependence flow}, we find some $M=M(\eta_0, \tau)>0$ such that
    \begin{align}\label{eq:3.5 2}
        \norm{\Gamma_0(t, \cdot)-\eta_0}{C^1}\leq Mt \quad \text{for all }t\in [0,\tau].
    \end{align}
    Considering the planar curve $\zeta_0:=(\eta_0^{(1)}, \eta_0^{(2)})$ and using Lemma \ref{lem:existence counterexample n=3} (iv), we find that $\zeta_0$ possesses a unique non-tangential  self-intersection at $\zeta_0(0)=\zeta_0(\frac{1}{2})=0$.
    Now, by the transversality of the self-intersection, %cf.\ 
   by \cite[Lemma 5.12]{LiYau1}, there exists $\omega_0=\omega_0(\zeta_0)>0$ such that any planar curve $\zeta$ with $\norm{\zeta-\zeta_0}{C^1} < \omega_0$ possesses a unique self-intersection at $\zeta(x)=\zeta(\Tilde{x})$. Moreover, this self-intersection is also non-tangential and the map 
    \begin{align}\label{eq:intersection map C^1}
        \{ \zeta \in C^1(\mathbb{T}^1;\mathbb{R}^2) : \norm{\zeta-\zeta_0}{C^1}<\omega_0\} \to \mathbb{R}^2, \zeta\mapsto (x, \Tilde{x})
    \end{align}
    is $C^1$, in particular Lipschitz continuous, after possibly reducing $\omega_0$. Thus, there exists $a=a(\zeta_0)>0$ such that $x\in [-a\omega, a\omega]$ and $\Tilde{x}\in [\frac{1}{2}-a\omega, \frac{1}{2}+a\omega]$ for all $\zeta$ with $\norm{\zeta-\zeta_0}{C^1}\leq \omega< \omega_0$.
    
    Now, we successively pick parameters
    \begin{enumerate}
       \item $\delta' = \delta'(\eta_0,\tau)       \in (0, \delta)$ small enough such that $\frac{\rho^2\delta'^4}{c} < \frac{\delta'}{2aM}<\tau$ and $\omega:=\frac{\delta'}{a}<\omega_0$;
       \item $\tau'=\tau'(\eta_0, \tau)>0$ such that $\frac{\rho^2\delta'^4}{c} < \tau' < \frac{\delta'}{2aM}= \frac{\omega}{2M}$;
       \item $\alpha_0 = \alpha_0(\eta_0,\delta',\tau)$ sufficiently small such that $\norm{\Gamma_\alpha(\cdot,t)-\Gamma_0(\cdot,t)}{C^1}\leq \frac{\omega}{2}$ for all $t\in [0,\tau], \alpha\in [0, \alpha_0]$, which is possible by Lemma \ref{lem:existence counterexample n=3} and Theorem \ref{thm:smooth dependence flow}.
    \end{enumerate}
    We observe, that by \eqref{eq:3.5 1}, Lemma \ref{lem:existence counterexample n=3} (iv) and the choice of $\tau'$, we have
    \begin{align}
        \Gamma_0^{(3)}(x, \tau')\leq \rho^2 x^4-c\tau'\leq  \rho^2\delta'^4 - c\tau' <0 \quad \text{for all }x\in [-\delta',\delta'].
    \end{align}
    Similarly, one obtains $\Gamma_0^{(3)}(x+\frac{1}{2}, \tau')>0$ for all $x\in [-\delta', \delta']$. Thus, fixing some sufficiently small $\alpha=\alpha(\eta_0, \delta, \delta', \tau, \tau')\in (0, \alpha_0)$, by Lemma \ref{lem:existence counterexample n=3} (iii) and Theorem \ref{thm:smooth dependence flow}, we may also assume
    \begin{align}\label{eq:3.5 2+eps}
        \Gamma_\alpha^{(3)}(x, \tau') <0, \quad \Gamma_\alpha^{(3)}(\frac{1}{2}+x, \tau')>0 \quad \text{for all }x\in [-\delta', \delta'].
    \end{align}
    Moreover, for all $t\in [0, \tau']$ by (iii) and \eqref{eq:3.5 2} we have
    \begin{align}\label{eq:3.5 3}
        \norm{\Gamma_\alpha(\cdot,t)-\eta_0}{C^1} &\leq \norm{\Gamma_\alpha(\cdot,t)-\Gamma_0(\cdot,t)}{C^1}+\norm{\Gamma_0(\cdot,t)-\eta_0}{C^1}\\
        &\leq \frac{\omega}{2}+M\tau' = \omega.
    \end{align}
    Hence, we may apply the above argument for transversal self-intersections, based on \cite[Lemma 5.12]{LiYau1}, to the projected planar curves 
    $$Z_\alpha(\cdot,t) := (\Gamma_\alpha^{(1)}(\cdot,t), \Gamma_\alpha^{(2)}(\cdot,t))$$ 
    and deduce that for all $t\in [0,\tau']$ the curve $Z_\alpha(\cdot,t)$ possesses a unique self intersection at $Z_\alpha(x(t),t))=Z_\alpha(\Tilde{x}(t),t)$, where by the choice of $\delta'=a\omega$, cf.\ (i), we have $x(t)\in [-\delta', \delta']$ and $\Tilde{x}(t)\in [\frac{1}{2}-\delta',\frac{1}{2}+\delta']$ for all $t\in [0, \tau']$.

    Now, using that by Theorem \ref{thm:smooth dependence flow} the flow $Z_\alpha$ is smooth, %cf.\
    and the properties of the map in \eqref{eq:intersection map C^1} we deduce that $[0, \tau']\ni t\mapsto x(t)$ and $[0, \tau']\ni t\mapsto\Tilde{x}(t)$ are continuous. 
    Thus, the function $f(t):= \Gamma_\alpha^{(3)}(x(t),t)-\Gamma_\alpha^{(3)}(\Tilde{x}(t),t)$ is continuous. By Lemma \ref{lem:existence counterexample n=3} (iv), we have $x(0)=0$ and $\Tilde{x}(0)=\frac{1}{2}$ and we find $f(0) = 2\rho^2u_\alpha(0)>0$ as $\alpha>0$. On the other hand, we have $x(\tau')\in [-\delta', \delta']$ and hence by \eqref{eq:3.5 2+eps}, we find $\Gamma_\alpha^{(3)}(x(\tau'), \tau')<0$ and similarly $\Gamma_\alpha^{(3)}(\Tilde{x}(\tau'), \tau')>0$, so $f(\tau')<0$. Consequently, there exists $t\in (0, \tau')$ with $f(t)=0$ and hence $\Gamma_\alpha(x(t),t)=\Gamma_\alpha(\Tilde{x}(t),t)$, so $\Gamma_\alpha$ has a self-intersection.
\end{proof}

\subsection{Preservation of embeddedness}

In this section, we will finally prove that below the energy thresholds in Theorem \ref{thm:main}, the respective elastic flows remain embedded.

First, we recall the following consequences of the gradient flow nature of the elastic flows \eqref{eq:elasticflow}, %cf.\
see also \cite[Section 4.2]{Length Preserving} for a precise discussion of the length-preserving case. 
\begin{remark}\label{rem:energy decay}
If $\gamma\colon\mathbb{T}^1\times [0,\infty)\to \mathbb{R}^n$ is an elastic flow with initial datum $\gamma_0$, then for all $t\in (0,\infty)$ we have
\begin{enumerate}
    \item if $\lambda>0$ is fixed, then $E_\lambda[\gamma(\cdot,t)] \leq E_\lambda[\gamma_0]$ with equality if and only if $\gamma_0$ is a $\lambda$-elastica (%cf.
    in the sense of Definition \ref{def:Eelastica});
    \item if $\lambda$ is given by \eqref{eq:def lambda}, then $\bar{B}[\gamma(\cdot,t)]\leq \bar{B}[\gamma_0]$ with equality if and only if $\gamma_0$ is an elastica.
\end{enumerate}
\end{remark}

\begin{proof}[Proof of Theorem \ref{thm:main}]
    First, we assume that $\gamma_0$ is an elastica (resp. a $\lambda$-elastica). Then by Remark \ref{rem:energy decay}, the flow  $\gamma(\cdot,t)\equiv \gamma_0$ is constant.

    In the case of the length-preserving flow, using \cite[Proposition 4.4]{LiYau2} we find that $\gamma(\cdot,t)\equiv \gamma_0$ is an embedded circle for all $t>0$ and the claim follows. If $\lambda>0$ is fixed, by the simple estimate $ab \leq \frac{1}{4\lambda}(a+\lambda b)^2$ for $a,b\geq 0$ and the assumption we have 
    \begin{align}
        \bar{B}[\gamma(\cdot,t)] \leq \frac{1}{4\lambda}E_\lambda[\gamma(\cdot,t)]^2\leq C^*(n)\quad\text{ for all }t\geq 0.\label{eq:Elambda vs bar B}
    \end{align}
     Now, since $\gamma(\cdot,t)\equiv \gamma_0$ is embedded by assumption, \cite[Proposition 4.4]{LiYau2} yields that $\gamma(\cdot, t)=\gamma_0$ is an embedded circle for all $t\geq 0$, and the statement follows.
    
    Hence, by Remark \ref{rem:energy decay} we may now assume $\bar{B}[\gamma(\cdot,t)]< \bar{B}[\gamma_0]$ (respectively  $E_\lambda[\gamma(\cdot,t)]<E_\lambda[\gamma_0]$) for all $t>0$. For both $\lambda>0$ fixed and $\lambda$ as in \eqref{eq:def lambda}, from \eqref{eq:Elambda vs bar B}, we find $\bar{B}[\gamma(\cdot,t)]<C^*(n)$ for all $t>0$. If $n\geq 3$, the embeddedness then directly follows from \cite[Theorem 1.1]{LiYau2}, cf.\ 
    Theorem \ref{thm:figure-eight}. If $n=2$, we observe that $N[\gamma(\cdot,t)]=N[\gamma_0]=1$ since the rotation number is invariant under regular homotopies. Therefore, since $\bar{B}[\gamma(\cdot,t)]<C^*(2)=C_{2T}$ for all $t>0$, the claim follows from Theorem \ref{thm:1.4}.
    
    For the optimality of the threshold, let $\varepsilon>0$ and let $\eta_\alpha$ be as in Lemma \ref{lem:existence counterexample n=2} for $n=2$  and as in Lemma \ref{lem:existence counterexample n=3} for $n\geq 3$, with the identification $\mathbb{R}^3 \cong \mathbb{R}^3\times\{0\}\subset \mathbb{R}^n$ for $n>3$.   By Propositions \ref{prop:loss of emb n=2} and \ref{prop:loss of emb n=3}, the elastic flows of $\eta_\alpha$ become non-embedded in finite time. For the length-preserving case, we observe that $\bar{B}[\eta_\alpha]\leq C^*(n)+\varepsilon$ by Lemmas \ref{lem:existence counterexample n=2} (i) and \ref{lem:existence counterexample n=3} (i) and the claimed optimality of the energy threshold follows. For the case of the $\lambda$-elastic flow with $\lambda>0$, we define $r:= \sqrt{\frac{B[\eta_\alpha]}{\lambda L[\eta_\alpha]}}>0$. Then, also the $\lambda$-elastic flow of $r\eta_\alpha$ becomes non-embedded in finite time. For the energy of $r\eta_\alpha$, we observe that
    \begin{align*}
        \frac{1}{4\lambda} E_\lambda[r\eta_\alpha]^2 &= \frac{B[r\eta_\alpha]^2 + 2\lambda \bar{B}[r\eta_\alpha]+ \lambda^2 L[r\eta_\alpha]^2}{4\lambda} \\
        &=  \frac{r^{-2}B[\eta_\alpha]^2 + 2\lambda \bar{B}[\eta_\alpha]+ \lambda^2 r^2 L[\eta_\alpha]^2}{4\lambda} = \bar{B}[\eta_\alpha]\leq C^*(n)+\varepsilon,
    \end{align*}
    using the scaling behavior of the energies and Lemmas \ref{lem:existence counterexample n=2} (i) and \ref{lem:existence counterexample n=3} (i). Thus, also in this case the optimality property is proven. 
\end{proof}

\appendix

\section{Jacobi Elliptic functions}\label{app:elliptic}
We provide some elementary properties of Jacobi elliptic functions, which can be found for example in \cite[Chapter 16]{Abramowitz}.
\begin{definition}[Amplitude Function, Complete Elliptic Integrals]
Fix $m \in [0,1) $. We define the \emph{Jacobi-amplitude function} $ \am(\,\cdot\,,m) \colon \mathbb{R} \rightarrow \mathbb{R} $ with \emph{modulus} $m$
 to be the inverse function of 
\begin{equation}
\mathbb{R} \ni z \mapsto \int_0^z \frac{1}{\sqrt{1- m\sin^2(\theta)}} \; \mathrm{d}\theta \in \mathbb{R}
\end{equation}
We define the \emph{complete elliptic integral of first} and \emph{second kind} as
\begin{align}
K(m) &:= \int_0^\frac{\pi}{2} \frac{1}{\sqrt{1- m \sin^2(\theta)}} \; \mathrm{d}\theta,& E(m) &:= \int_0^\frac{\pi}{2} \sqrt{1- m \sin^2(\theta)} \; \mathrm{d}\theta
\end{align}
and the \emph{incomplete elliptic integral of first} and \emph{second kind} as
\begin{align}
F(x,m) &:= \int_0^x \frac{1}{\sqrt{1- m \sin^2(\theta)}} \; \mathrm{d}\theta,& E(x, m) &:= \int_0^x \sqrt{1- m \sin^2(\theta)} \mathrm{d}\theta .
\end{align}
Note that $F(\cdot,m) = \am(\cdot,m)^{-1}$. 
\end{definition}

\begin{definition}[Elliptic Functions]\label{def:B2}
For $m\in [0,1)$ the \emph{Jacobi elliptic functions} are given by
\begin{align}
\cn(\cdot,m)\colon \mathbb{R} \rightarrow \mathbb{R}, \;\; &\cn(x,m) := \cos(\am(x,m)), \\ \sn(\cdot,m)\colon \mathbb{R} \rightarrow \mathbb{R}, \;\; &\sn(x,m) := \sin(\am(x,m)), \\ \dn(\cdot,m)\colon \mathbb{R} \rightarrow \mathbb{R}, \;\; &\dn(x,m) := \sqrt{1-m\sin^2(\am(x,m))}.
\end{align}
\end{definition}
The following proposition summarizes all relevant properties and identities for the elliptic functions. They can all be found in \cite[Chapter 16]{Abramowitz}.

\begin{proposition}\label{prop:identities}
\leavevmode

\begin{enumerate}[label={\upshape(\roman*)}]
\item (Derivatives and Integrals of Jacobi Elliptic Functions)  
For each $x \in \mathbb{R}$ and $m \in (0,1)$ we have
\begin{align}
\frac{\partial}{\partial x} \cn(x,m) & = - \sn(x,m) \dn(x,m), &
\frac{\partial}{\partial x} \sn(x,m) & =  \cn(x,m) \dn(x,m) , \\
\frac{\partial}{\partial x} \dn(x,m) & = - m \cn(x,m) \sn(x,m), &
\frac{\partial}{\partial x} \am(x,m) & = \dn(x,m).
\end{align}
\item (Derivatives of Complete Elliptic Integrals) For $m \in (0,1)$ $E$ is smooth and 
\begin{align}
\frac{\mathrm{d}}{\mathrm{d} m}E(m)  &= \frac{E(m) - K(m)}{2m}, 
& \frac{\mathrm{d}}{\mathrm{d} m}K(m)   = \frac{(m-1) K(m) +E(m) }{2 m(1-m)}.
\end{align}
\item (Trigonometric Identities) For each $m \in [0,1)$ and $x \in \mathbb{R}$ the Jacobi elliptic functions satisfy
\begin{align} 
  \cn^2(x,m) + \sn^2(x,m) & = 1, &  \dn^2(x,m) + m \sn^2(x,m) &= 1 .
  \end{align}
\item (Periodicity) All periods of the elliptic functions are given as follows, where $l \in \mathbb{Z}$ and $x \in \mathbb{R}$:
\begin{align} 
\am(lK(m),m) & = l \frac{\pi}{2}, & \nonumber
\cn(x+ 4 l K(m), m ) &  = \cn(x,m)  ,\\
\sn(x+ 4 l K(m), m ) &  = \sn(x,m) , &
\dn(x+ 2 l K(m), m ) &  = \dn(x,m) ,\label{eq:PeriodDN}\\
F(  \tfrac{l\pi}{2} , m )& = lK(m)  ,  & E(  \tfrac{l\pi}{2} , m )& = lE(m) , \nonumber
\end{align}
\begin{equation}
 \am(x+ 2lK(m),m)  =  l \pi + \am(x,m),
\end{equation}
\begin{equation}
F(x + l \pi, m) = F(x,m) + 2lK(m) , 
\end{equation}
\begin{equation}
E(x + l\pi , m ) = E(x,m) + 2lE(m). 
\end{equation}
\item (Asymptotics of the Complete Elliptic Integrals)
\begin{equation}
\lim_{m \rightarrow 1 } K (m) = \infty, \quad\lim_{m \rightarrow 0 } K(m) = \frac{\pi}{2}, \quad 
\lim_{m \rightarrow 1 } E (m) = 1, \quad \lim_{m \rightarrow 0 } E(m) = \frac{\pi}{2}.
\end{equation}
\end{enumerate} 
\end{proposition}

\section{Some computational lemmas}\label{app:compu}

\begin{proof}[Proof of Proposition \ref{prop:monotonequantity}]
   One readily computes with %definition 
   notation {\eqref{eq:def alpha},} standard trigonometric identities and the estimate $\sin^4(\theta) \leq \sin^2(\theta)$,
   \begin{align}
       f'(m) & = \frac{1-2m \sin^2(\pi- \arcsin\sqrt{\frac{1}{2m}} )}{\sqrt{1- m \sin^2(\pi- \arcsin\sqrt{\frac{1}{2m}}) }} \frac{\mathrm{d}}{\mathrm{d} m}  \left( \pi - \arcsin \sqrt{\frac{1}{2m}}  \right)   \\ 
       & \;\; + \int_0^{\pi- \arcsin\sqrt{\frac{1}{2m}} } \left( \frac{-2 \sin^2(\theta)}{\sqrt{1-m \sin^2(\theta)}} + \frac{1}{2} \frac{(1-2m \sin^2(\theta)) \sin^2(\theta)}{(1-m \sin^2(\theta))^\frac{3}{2}} \right) \; \mathrm{d}\theta\\ 
       % & \quad + \int_0^{\pi- \arcsin\sqrt{\frac{1}{2m}} } \frac{-2 \sin^2(\theta)}{\sqrt{1-m \sin^2(\theta)}} \; \mathrm{d}\theta \\
       % & \qquad + \int_0^{\pi- \arcsin\sqrt{\frac{1}{2m}} }  \frac{1}{2} \frac{(1-2m \sin^2(\theta)) \sin^2(\theta)}{(1-m \sin^2(\theta))^\frac{3}{2}} \; \mathrm{d}\theta \\ 
       & = \int_0^{\pi- \arcsin\sqrt{\frac{1}{2m}}} \frac{-\frac{3}{2} \sin^2(\theta) + m \sin^4(\theta)}{(1- m \sin^2(\theta))^\frac{3}{2}} \; \mathrm{d}\theta \\
       &\leq  \int_0^{\pi- \arcsin\sqrt{\frac{1}{2m}}} \frac{\left( -  \frac{3}{2} + m \right)\sin^2(\theta)}{(1- m \sin^2(\theta))^\frac{3}{2}} \; \mathrm{d}\theta.
   \end{align}
   This expression is negative as $m < \frac{3}{2}$. 
   Now note that 
   \begin{equation}
       f( \frac{1}{2} ) = \int_0^\frac{\pi}{2} \frac{\cos^2(\theta)}{1- \frac{1}{2}\sin^2(\theta)} \; \mathrm{d}\theta > 0.
   \end{equation}
   Moreover,
   \begin{align}
       f(m_8) & = \int_0^{\pi- \arcsin\sqrt{\frac{1}{2m_8}} } \frac{1-2m_8\sin^2(\theta)}{\sqrt{1-m_8 \sin^2(\theta) }}
       \\ & =\int_0^{\frac{\pi}{2} } \frac{1-2m_8\sin^2(\theta)}{\sqrt{1-m_8 \sin^2(\theta) }} \; \mathrm{d}\theta + \int_{\frac{\pi}{2}}^{\pi- \arcsin\sqrt{\frac{1}{2m_8}} } \frac{1-2m_8\sin^2(\theta)}{\sqrt{1-m_8 \sin^2(\theta) }} \; \mathrm{d}\theta.
   \end{align}
   This is smaller than zero since the first integral equals $2E(m_8)-K(m_8)=0$ and the second integral is negative as $1-2m_8 \sin^2(\theta) < 0$ for all $\theta \in [\frac{\pi}{2}, \pi - \arcsin{\sqrt{\frac{1}{2m_8}}}].$
   Existence and uniqueness of the root $m = m_T$ follows from the intermediate value theorem and strict monotonicity. 
\end{proof}

\begin{proof}[Proof of Proposition \ref{prop:gdecr}]
We first show that $g$ is decreasing.
    To this end we expand $g =g(m)$ in a power series on $(0,1)$ and analyze the coefficients. We define for all $k \in \mathbb{N}$
    \begin{equation}
        A_k := \int_{-\frac{\pi}{4}}^\frac{5\pi}{4} \sin^{2k}(\theta) \; \mathrm{d}\theta.
    \end{equation}
    We compute 
    \begin{align}
        g(m) & = \int_{- \frac{\pi}{4}}^\frac{5\pi}{4} (1- 2 \sin^2(\theta)) \sum_{k = 0}^\infty {{-\frac{1}{2}}\choose{k}} (-1)^k \sin^{2k}(\theta) m^k \; \mathrm{d} \theta
        \\ & = \sum_{k = 0}^\infty \frac{(-1)^k}{k!}  \prod_{l = 0}^{k-1} \left( - \frac{1}{2} - l \right) (A_k - 2 A_{k+1}) m^k  
        \\ & = \sum_{k = 0}^\infty \frac{1}{k!} \prod_{l=0}^{k-1} \left( l + \frac{1}{2} \right) (A_k- 2 A_{k+1} ) m^k.  \label{eq:powerser}
    \end{align}
    Next we have a closer look at $A_k$. To this end observe that for all $k \in \mathbb{N}_0$
    \begin{align}
        A_{k + 1 } & = \int_{- \frac{\pi}{4}}^{\frac{5\pi}{4}} \sin^{2k+1} \theta \sin \theta \; \mathrm{d}\theta \\
        &= \left[ - \sin^{2k+1} \theta \cos \theta \right]_{- \frac{\pi}{4}}^{\frac{5\pi}{4}}+ \int_{-\frac{\pi}{4}}^\frac{5\pi}{4}(2k+1) \sin^{2k} \theta \cos^2 \theta \; \mathrm{d} \theta
        \\ & = - \frac{1}{2^k} + (2k+1) A_k - (2k+1) A_{k+1}.
    \end{align}
    One infers that 
    \begin{align}\label{eq:recursion}
        A_{k+1} = \frac{1}{2k+2} \left( -\frac{1}{2^k}+ (2k+ 1) A_k \right).
    \end{align}
    Using this we find that 
    \begin{equation}\label{eq:236}
        A_k - 2A_{k+1} = \frac{1}{2k+2} \left( \frac{1}{2^{k-1}} - 2 k A_k \right). 
    \end{equation}
    Next we show via induction that $\frac{1}{2^{k-1}} - 2k A_k \leq  0$ for all $k \geq 1$, equivalently $k A_k \geq \frac{1}{2^k}$ for all $k \geq 1$.
    One can compute for $k =1$ that $A_1 = \frac{3\pi-2}{4} > \frac{1}{2}$. Next we assume that $k A_k \geq \frac{1}{2^k}$ for some fixed $k \geq 1$ and compute with  \eqref{eq:recursion} and the induction hypothesis
    \begin{align}
        (k+1) A_{k+1} &= \frac{1}{2} \left( (2k+1) A_k - \frac{1}{2^k} \right) \geq \frac{1}{2} \left( 2k A_k - \frac{1}{2^k} \right) \\
        &\geq \frac{1}{2} \left( \frac{1}{2^{k-1}}- \frac{1}{2^k} \right) = \frac{1}{2^{k+1}}
.    \end{align}
This yields the claim also for $k+1$. By induction the claim follows. Going back to  \eqref{eq:236}  we find that $A_k - 2 A_{k+1} \leq 0$.
Note that in  the special case of $k = 1$ we can actually obtain 
\begin{equation}
    A_1 - 2 A_2 = \frac{1}{4} ( 1 - 2 A_1) = \frac{1}{2} \left( 1- \frac{3\pi}{4} \right)  < 0. 
\end{equation}
    Going back to \eqref{eq:powerser} we obtain  
    \begin{equation}
        g(m) = 1 + \sum_{k = 1}^\infty \beta_k m^k \quad \textrm{for all} \quad  m \in (0,1)
    \end{equation}
    for some real numbers $\beta_1,\beta_2,... \leq 0$ and $\beta_1 < 0$. This yields that 
    \begin{equation}
        g'(m) = \sum_{k = 1}^\infty k\beta_km ^{k-1} < 0 \quad \textrm{for all} \quad  m \in (0,1),
    \end{equation}
    meaning that $g$ is decreasing. 
    Next we show that $\lim_{m \rightarrow 0+} g(m) >0 $ and $\lim_{m \rightarrow 1-} g(m) < 0 $.
    It is easy to compute that 
    \begin{equation}
        \lim_{m \rightarrow 0} g(m) = \int_{- \frac{\pi}{4}}^\frac{5\pi}{4} (1- 2 \sin^2(\theta) ) \; \mathrm{d}\theta = 1. 
    \end{equation}
    For the behavior as $m \rightarrow 1 $ we write
    \begin{align}
        g(m) &=  \int_{-\frac{\pi}{4}}^0  \frac{1- 2 \sin^2(\theta)}{\sqrt{1- m \sin^2(\theta)}} \; \mathrm{d}\theta \\
        &\qquad + \int_{\pi}^\frac{5\pi}{4}  \frac{1- 2 \sin^2(\theta)}{\sqrt{1- m \sin^2(\theta)}} \; \mathrm{d}\theta  + \int_{0}^\pi \frac{1- 2 \sin^2(\theta)}{\sqrt{1-m \sin^2 (\theta)}} \; \mathrm{d}\theta.
    \end{align}
    Now since $1- 2 \sin^2(\theta) \geq 0$ and $ \sin^2(\theta) < \frac{1}{2}$ in the range of the first two integrals we can estimate 
    \begin{align}
        g(m) & \leq \frac{1}{\sqrt{1- \frac{m}{2}}}  \left( \int_{-\frac{\pi}{4}}^0 (1- 2\sin^2 \theta) \; \mathrm{d}\theta + \int_{\pi}^\frac{5\pi}{4}(1- 2\sin^2 \theta) \; \mathrm{d}\theta \right) \\
        &\quad + \frac{2}{m}(2E(m) + (m-2)K(m)) 
       \\ & \leq \frac{\pi}{\sqrt{1-\frac{m}{2}}} + \frac{2}{m} (2E(m) + (m-2) K(m)).
    \end{align}
    Now we can take $m \rightarrow 1$ and infer from Proposition \ref{prop:identities} (v) that
    $$
        \lim_{m \rightarrow 1 } g(m) = - \infty.
    $$
    The existence and uniqueness of the root follows now from the intermediate value theorem.
\end{proof}
\begin{proof}[Proof of Lemma \ref{lem:compulemorbit}]
    Define $h(m) := 2 \frac{E(m)}{K(m)} + (m-2)$. Using the techniques of \cite[Proof of {Lemma B.4}]{LiYau1} we infer that $\frac{\mathrm{d}}{\mathrm{d}m} \frac{E(m)}{K(m)} < - \frac{1}{2}$ for all $m \in (0,1)$ and thus we infer that $h'(m) < 0$. Now note that
    $
        \lim_{m \rightarrow 0+ } h(m) = 2 \frac{\frac{\pi}{2}}{\frac{\pi}{2}} - 2 = 0,
    $
    {by Proposition \ref{prop:identities} (v).}
    This and the negative derivative imply $h(m) < 0 $ for all $m \in (0,1).$ The statement follows since $2 E(m) + (m-2) K(m) = h(m) K(m).$
\end{proof}

\begin{lemma}\label{lem:m_T}
    $m_T > \frac{2}{3}$. 
\end{lemma}
\begin{proof}
  Let $f$ be as in \eqref{eq:fwavelike}. Since $f$ is decreasing %cf.\
  by Proposition \ref{prop:monotonequantity} and $m_T$ is the unique root of $f$ it suffices to prove that $f(\frac{2}{3}) > 0.$ Since ${\alpha(\frac{2}{3})} = \frac{\pi}{3}$ we obtain 
  \begin{equation}
      f(\frac{2}{3}) = \int_0^{\frac{2}{3}\pi} \frac{1-\frac{4}{3}\sin^2(\theta)}{\sqrt{1-\frac{2}{3}\sin^2(\theta)}} \; \mathrm{d}\theta.
  \end{equation}
  Using Weierstrass substitution $u = \tan \frac{\theta}{2}$ we obtain
  \begin{align}
     f( \frac{2}{3} ) & = \int_0^{\sqrt{3}} \frac{1- \frac{16}{3} \frac{u^2}{(1+u^2)^2}}{\sqrt{1- \frac{8}{3}\frac{u^2}{(1+u^2)^2} }} \frac{2}{(1+u^2)} \; \mathrm{d}u  \\
     &= \frac{2}{\sqrt{3}} \int_0^{\sqrt{3}} \frac{3(1+u^2)^2- 16 u^2}{\sqrt{3(1+u^2)^2- 8u^2}} \frac{1}{(1+u^2)^2} \; \mathrm{d}u
      \\ & = \frac{2}{\sqrt{3}} \int_0^{\sqrt{3}} \frac{3+ 3u^4 - 10 u^2}{\sqrt{3+3u^4 - 2u^2}} \frac{1}{(1+u^2)^2} \; \mathrm{d}u. \label{eq:aestf}
  \end{align}
 We split the integral that appears into two parts. For $u \in [\frac{1}{\sqrt{3}}, \sqrt{3}]$ we {can} estimate $3 + 3u^4 -2u^2 = 3(1-u^2)^2  + 4u^2 \geq 4 u^2$. As a consequence, we have 
 \begin{align}
     \frac{3+ 3u^4 - 10 u^2}{\sqrt{3+3u^4 - 2u^2}} &= \sqrt{3+ 3u^4 - 2u^2} - 8 \frac{u^2}{\sqrt{3+3u^4 - 2u^2}} \\
     &\geq 2 u - \frac{8u^2}{(2u)} = -2u, 
 \end{align}
 and thus (with the substitution $z = u^2$)  
 \begin{align}
     \int_{\frac{1}{\sqrt{3}}}^{\sqrt{3}} \frac{3+ 3u^4 - 10 u^2}{\sqrt{3+3u^4 - 2u^2}(1+u^2)^2}  \; \mathrm{d} u &\geq - \int_\frac{1}{\sqrt{3}}^{\sqrt{3}} \frac{2u}{(1+u^2)^2} \; \mathrm{d}u \\
     &= - \int_{\frac{1}{3}}^{3} \frac{1}{(1+z)^2} \; \mathrm{d}z = - \frac{1}{2}. 
 \end{align}
 For $u \in [0,\frac{1}{\sqrt{3}}]$ we estimate $\sqrt{3+3u^4 - 2u^2} \leq \sqrt{3+3u^4 + 6u^2} = \sqrt{3}(1+u^2)$. Moreover, the numerator in the integrand $3 + 3u^4 - 10u^2 = 3 ( 3- u^2) (\frac{1}{3}- u^2)$ is nonnegative in $[0, \frac{1}{\sqrt{3}}]$. Hence we can estimate 
 \begin{equation}
     \int_0^{\frac{1}{\sqrt{3}}} \frac{3+3u^4 - 10u^2}{\sqrt{3+3u^4 - 2u^2}(1+u^2)^2} \; \mathrm{d}u \geq \int_0^\frac{1}{\sqrt{3}} \frac{\sqrt{3}(3-u^2)(\frac{1}{3}-u^2)}{(1+u^2)^3} \; \mathrm{d}u. 
 \end{equation} 
 One readily computes that  $\frac{1}{\sqrt{3}} \left( \arctan(u) - \frac{2u(u^2-1)}{(u^2+1)^2} \right)$ is an antiderivative for the integrand in the previous equation. Evaluating this antiderivative at the limits and using that $\arctan \frac{1}{\sqrt{3}} = \frac{\pi}{6}$ we obtain
 \begin{equation}
     \int_0^{\frac{1}{\sqrt{3}}} \frac{3+3u^4 - 10u^2}{\sqrt{3+3u^4 - 2u^2}(1+u^2)^2} \; \mathrm{d}u \geq  \frac{1}{4} + \frac{\pi}{6\sqrt{3}}. 
 \end{equation}
 Plugging in all the previous findings {into} \eqref{eq:aestf} we obtain 
 \begin{equation}
     f( \frac{2}{3} ) {\geq} \frac{2}{\sqrt{3}} \left( \frac{1}{4}+ \frac{\pi}{6\sqrt{3}} - \frac{1}{2} \right)  \geq \frac{2}{\sqrt{3}} \left( \frac{1}{2\sqrt{3}} - \frac{1}{4} \right) > 0,
 \end{equation}
 where we have used $\pi > 3$ and $\sqrt{3} < 2$ in the last two steps. 
\end{proof}
 
 \section{A detailed proof of optimal global regularity}\label{app:optireg}

 \begin{proof}[Proof of Lemma \ref{lem:optireg}]
    Let $\gamma \in \mathcal{A}_0$ be a solution of \eqref{eq:pertfunc}. By Proposition \ref{prop:selfinters}, $\gamma$ has only one point of self-intersection with multiplicity two, say $p = \gamma(a) = \gamma(b)$. Furthermore, $\gamma$ is smooth away from $a,b$. Recall from Proposition \ref{prop:selfinters} that $T_\gamma(a) = - T_\gamma(b)$. After rotation and translation we may assume that $p = 0$ and $T_\gamma(a) = (1,0)$. By the implicit function theorem we infer that there exists $\delta > 0$ and an open neighborhood $U \subset \mathbb{R}^2$ of $0$ such that  \begin{equation}
      \gamma(\mathbb{T}^1) \cap U = \mathrm{graph}(u_a) \cup \mathrm{graph}(u_b)
  \end{equation}
  for some functions $u_a, u_b \in W^{2,2}((-\delta, \delta))$.
  
We claim that we can choose $u_a, u_b$ in a way that $u_a(x) \leq u_b(x)$ for all $x \in [-\delta,\delta]$ with equality if and only if  $x = 0$ and there exist $V_a,V_b \subset \mathbb{T}^1$ open neighborhoods of $a$ respectively $b$ such that $x \mapsto (x,u_a(x))$ is a reparametrization of $\gamma \vert_{V_a}$ and $x \mapsto (x,u_b(x))$ is a reparametrization of $\gamma\vert_{V_b}$. Indeed, if one chooses arbitrary graph reparametrizations $u_a$ (resp. $u_b$) in $W^{2,2}((-\delta, \delta))$ on suitably small neighborhoods $V_a$ and $V_b$ then $u_a = u_b$ may only happen at $x= 0$ as this is the only point of self-intersection. If $u_b - u_a$ changes sign at $x=0$ then each small perturbation of $\gamma$ in $V_a$ will also have a self-intersection (by the intermediate value theorem). The same will apply to perturbations in $V_b$. Having this we conclude from Lemma \ref{lem:localize} that $\gamma \in \mathcal{A}_0$ is an elastica, a contradiction to Lemma \ref{lem:ruleoutelastica}. Hence $u_b - u_a$ may not change sign.
  
In the sequel we will frequently use the following expressions for our energies in terms of $u_a,u_b.$
  \begin{equation}
      \int_{\gamma\vert_{V_a}} k^2 \; \mathrm{d}s = \int_{-\delta}^{\delta} \frac{u_a''(x)^2}{(1+ u_a'(x)^2)^\frac{5}{2}} \; \mathrm{d}x,\quad  \int_{\gamma\vert_{V_a}} 1 \; \mathrm{d}s = \int_{-\delta}^{\delta} 
      \sqrt{1+ u_a'(x)^2} \; \mathrm{d}x. 
  \end{equation}
  
  Next fix $\phi \in C_0^\infty(-\delta,\delta)$  such that $\phi \geq 0$.  For $t> 0$  let  $\gamma_t$ be a curve that coincides with $\gamma$ outside of $V_a$ and with a suitable reparametrization of $x \mapsto u_a(x) + t \phi(x)$, $(x \in (-\delta,\delta))$ inside $V_a$.
  We claim that the perturbation curve $(t \mapsto \gamma_t)$ lies in $C^1([0,\varepsilon); \mathcal{A}_0)$. Indeed, one readily checks that for $t> 0$ small enough one has $\gamma_t \in H^2_{imm}( \mathbb{T}^1; \mathbb{R}^2)$ and $N[\gamma_t] = 1$. Moreover, we observe that
  $(u_a + t \phi) (0) \geq u_a(0) = u_b(0)$ but
  $(u_a + t \phi)(-\delta) = u_a(-\delta) < u_b(-\delta)$. By the intermediate value theorem there exists $x_t \in (-\delta, 0]$ such that $(u_a+ t\phi)(x_t) = u_b(x_t)$, implying that $\gamma_t$ is not injective. We conclude from \eqref{eq:pertfunc} that 
  \begin{align}
      0& \leq \frac{\mathrm{d}}{\mathrm{d} t}\Big\vert_{t = 0 } B[\gamma_t]L[\gamma_t]  \\ & = L[\gamma] \frac{\mathrm{d}}{\mathrm{d}t}\Big\vert_{t = 0 } \int_{-\delta}^\delta \frac{(u_a''+ t \phi'')^2}{(1+ (u_a'+ t \phi')^2)^\frac{5}{2}} \; \mathrm{d}x \\
      &\qquad + B[\gamma] \frac{\mathrm{d}}{\mathrm{d}t}\Big\vert_{t = 0 } \int_{-\delta}^{\delta} \sqrt{1 + (u_a'+ t \phi')^2} \; \mathrm{d}x
      \\ & = 2 L[\gamma] \int_{-\delta}^{\delta} \frac{u_a''\phi''}{(1+u_a'^2)^\frac{5}{2}} \; \mathrm{d}x - 5 L[\gamma] \int_{-\delta}^{\delta} \frac{u_a''^2 u_a' \phi'}{(1+ u_a'^2)^\frac{7}{2}} \; \mathrm{d}x \\
      &\qquad + B[\gamma] \int_{-\delta}^\delta \frac{u_a' \phi'}{\sqrt{1+ u_a'^2}} \; \mathrm{d}x.
   \end{align}
   Since $\phi \in C_0^\infty(-\delta, \delta), \phi \geq 0,$ was arbitrary, the Riesz--Markow--Kakutani theorem  yields a Radon measure $\mu$ on $(-\delta,\delta)$ such that for all $\phi \in C_0^\infty(-\delta, \delta)$ one has
   \begin{align}
        2 L[\gamma] \int_{-\delta}^{\delta} \frac{u_a''\phi''}{(1+u_a'^2)^\frac{5}{2}} \; \mathrm{d}x - 5 L[\gamma] \int_{-\delta}^{\delta} \frac{u_a''^2 u_a' \phi'}{(1+ u_a'^2)^\frac{7}{2}} \; \mathrm{d}x & \\
        + B[\gamma] \int_{-\delta}^\delta \frac{u_a' \phi'}{\sqrt{1+ u_a'^2}} \; \mathrm{d}x &= \int \phi \; \mathrm{d}\mu.
   \end{align}
   
   We show next that $\mu$ is a multiple of the Dirac measure $\delta_0$ concentrated in zero. To this end, it suffices  to show that for all $\phi \in C_0^\infty((-\delta,\delta) \setminus \{ 0 \})$ one has $\int \phi \; \mathrm{d}\mu =0.$
    Fix $\phi \in C_0^\infty((-\delta, \delta) \setminus \{0 \})$. Since $u_a < u_b$ on $(-\delta,\delta) \setminus \{ 0 \}$ and $\mathrm{supp}(\phi)$ is compact we can find $\varepsilon > 0$ such that {$u_a + \varepsilon ||\phi||_\infty < u_b$} on $\mathrm{supp}(\phi)$. In particular, for all $t \in (-\varepsilon,\varepsilon)$ one has $u_a + t \phi \leq u_b$ on $(-\delta,\delta)$ with equality only at $x= 0$. Now (by possibly shrinking $\varepsilon$) define for $t \in (-\varepsilon, \varepsilon)$  a curve in $\gamma_t \in \mathcal{A}_0$ that coincides with $\gamma$ outside $V_a$ and with a reparametrization of $u_a + t\phi$ inside $V_a$. One readily checks that $t \mapsto \gamma_t$ lies in $C^1((-\varepsilon,\varepsilon); \mathcal{A}_0)$. Equation \eqref{eq:varineqpm} yields 
    \begin{align}
        0 & = \frac{\mathrm{d}}{\mathrm{d} t} \Big\vert_{t = 0 } L[\gamma_t] B[\gamma_t]
        \\ & = 2 L[\gamma] \int_{-\delta}^{\delta} \frac{u_a''\phi''}{(1+u_a'^2)^\frac{5}{2}} \; \mathrm{d}x - 5 L[\gamma] \int_{-\delta}^{\delta} \frac{u_a''^2 u_a' \phi'}{(1+ u_a'^2)^\frac{7}{2}} \; \mathrm{d}x \\
        & \qquad + B[\gamma] \int_{-\delta}^\delta \frac{u_a' \phi'}{\sqrt{1+ u_a'^2}} \; \mathrm{d}x.
    \end{align}
    Since the right hand side coincides with $\int \phi \;\mathrm{d}\mu$ and $\phi \in C_0^\infty((-\delta,\delta) \setminus \{0 \} )$ was arbitrary we obtain $\mathrm{supp}(\mu) \subset \{0\}$. 
    
    Hence $\mu = c \delta_0$ for some $c \geq 0$ and hence for all $\phi \in C_0^\infty(- \delta, \delta)$ one has 
    \begin{align}
         2 L[\gamma] \int_{-\delta}^{\delta} \frac{u_a''\phi''}{(1+u_a'^2)^\frac{5}{2}} \; \mathrm{d}x &=  5 L[\gamma] \int_{-\delta}^{\delta} \frac{u_a''^2 u_a' \phi'}{(1+ u_a'^2)^\frac{7}{2}} \; \mathrm{d}x \\
         &\qquad - B[\gamma] \int_{-\delta}^\delta \frac{u_a' \phi'}{\sqrt{1+ u_a'^2}} \; \mathrm{d}x + c \phi(0).
    \end{align}
    Rewriting $\phi(0)= \int_{-\delta}^{\delta} \chi_{(-\delta,0)} \phi' \; \mathrm{d}x$ we infer that 
    \begin{align}
        &\int_{-\delta}^{\delta} \frac{u_a'' }{(1+ u_a'^2)^\frac{5}{2}} \phi'' \; \mathrm{d}x \\
        =& \int_{-\delta}^{\delta} \left( \frac{5u_a''^2u_a'}{2(1+ u_a'^2)^\frac{7}{2}} - \frac{B[\gamma]}{2L[\gamma]} \frac{u_a'}{\sqrt{1+ u_a'^2}}  + \frac{c}{2L[\gamma]} \chi_{(-\delta,0)} \right) \phi' \; \mathrm{d}x.
    \end{align}
    Note that the expression in parentheses lies in $L^1(-\delta,\delta)$. 
    A standard technique (%cf.\ 
   see e.g.\ \cite[Proof of Proposition 3.2]{AnnaObst}) shows now that 
   $$\frac{u_a''}{(1+u_a'^2)^\frac{5}{2}} \in W^{1,1}(-\delta,\delta)$$ 
   and 
    \begin{equation}\label{eq:bootstrap}
        \frac{\mathrm{d}}{\mathrm{d}x} \frac{u_a''}{(1+u_a'^2)^\frac{5}{2}} = \frac{5u_a''^2u_a'}{2(1+ u_a'^2)^\frac{7}{2}} - \frac{B[\gamma]}{2L[\gamma]} \frac{u_a'}{\sqrt{1+ u_a'^2}}  + \frac{c}{2L[\gamma]} \chi_{(-\delta,0)}+D 
    \end{equation}
    for a constant $D \in \mathbb{R}$.
    By the chain rule we infer that
    $$(1+ u_a'^2)^\frac{5}{2} \in {W^{1,1}(-\delta,\delta)}$$ 
    and by the product rule (using the fact that $W^{1,1}(-\delta, \delta) \subset C^0([-\delta,\delta])$) we conclude from \eqref{eq:bootstrap} that 
    $$u_a'' \in W^{1,1}(-\delta,\delta).$$ 
    In particular, also 
    $$u_a'' \in C^0([-\delta,\delta]) \subset L^\infty(-\delta,\delta).$$ Inserting this new information back into \eqref{eq:bootstrap} we obtain 
    $$\frac{\mathrm{d}}{\mathrm{d}x} \frac{u_a''}{(1 + u_a'^2)^\frac{5}{2}} \in L^\infty(-\delta,\delta).$$ Arguing again with the chain rule and the product rule we infer that $u_a'' \in W^{1,\infty}(-\delta,\delta)$, which implies 
    $$u_a \in W^{3,\infty}(-\delta,\delta).$$ 
    Analogously, one shows that $u_b \in W^{3,\infty}(-\delta,\delta)$. 
    The above being shown, one readily checks that the arclength reparametrizations of $\mathrm{graph}(u_a)$ and $\mathrm{graph}(u_b)$ also lie in $W^{3,\infty}$. We can conclude that each constant-speed reparametrization of $\gamma$ lies in $W^{3,\infty}(\mathbb{T}^1;\mathbb{R}^2)$. Indeed, such reparametrization of $\gamma$ is smooth outside of $a,b$  and given by a constant-speed reparametrization of a $W^{3,\infty}$-graph in neighborhoods of $a$ and $b$.  The $W^{3,\infty}$-regularity is shown. Continuity of the curvature follows from the fact that by the previous findings the curvature of the constant-speed parametrization is continuous. Here we used the transformation law for the curvature under reparametrization.
\end{proof}

\end{document}